\documentclass[a4paper,twoside,10pt]{article}
\usepackage[a4paper,left=3cm,right=3cm, top=3cm, bottom=3cm]{geometry}
\usepackage{mathrsfs}
\usepackage{amsfonts}
\usepackage[draft]{graphicx}
\usepackage{epstopdf}
\usepackage{cite}
\usepackage{amsmath}
\usepackage{amsthm}
\usepackage{subfigure}
\usepackage[font={footnotesize}]{caption}
\usepackage{cases}
\usepackage{stmaryrd}
\usepackage{amssymb}
\usepackage{soul,xcolor}
\usepackage{tikz}
\usepackage{float}
\usepackage{hyperref}
\usepackage{algorithm,algpseudocode}

\newtheorem{thm}{Theorem}[section]

\newtheorem{lem}[thm]{Lemma}
\newtheorem{prop}[thm]{Proposition}
\theoremstyle{definition}

\theoremstyle{remark}
\newtheorem{remark}{Remark}

\newcommand{\h}{h} 
\newcommand{\hE}{\h_\E} 
\newcommand{\hEi}{\h_{\Ei}} 
\newcommand{\hT}{\h_\T} 
\newcommand{\he}{\h_\e} 
\renewcommand{\u}{u} 
\newcommand{\vv}{v} 
\newcommand{\vn}{\vv_n} 
\newcommand{\w}{w} 
\newcommand{\f}{f} 
\newcommand{\fn}{\f_n} 
\newcommand{\un}{\u_n} 
\newcommand{\Pinabla}{\Pi^{\nabla}_{\pE}} 
\newcommand{\Pinablaa}{\Pi^{\nabla}} 
\newcommand{\PinablapuE}{\Pi^{\nabla}_{\pE}} 
\newcommand{\PinablapuEi}{\Pi^{\nabla}_{\pEi}} 
\newcommand{\Pinablai}{\Pi^{\nabla}_{\p_{\E_i}}} 
\newcommand{\PinablaE}{\Pi^{\nabla,\E}_{\pE}} 
\newcommand{\Pizpmd}{\Pi^0_{\pE-2}} 
\newcommand{\PizpmduE}{\Pi^0_{\pE-2}} 
\newcommand{\PizpmdE}{\Pi^{0,\E}_{\pE-2}} 
\newcommand{\e}{s} 
\newcommand{\etilde}{\widetilde \e} 
\newcommand{\E}{K} 
\newcommand{\Ei}{\E_i} 
\newcommand{\ES}{\E_S} 
\newcommand{\Etilde}{\widetilde \E} 
\newcommand{\V}{V} 
\newcommand{\Vn}{\V_n} 
\newcommand{\VnE}{\Vn(\E)} 
\newcommand{\chin}{\chi_n} 
\renewcommand{\a}{a} 
\newcommand{\aE}{a^\E} 
\newcommand{\aEp}{a^{\E'}} 
\newcommand{\aEi}{a^{\E_i}} 
\newcommand{\an}{\a_n} 
\newcommand{\anE}{\an^\E} 
\newcommand{\taun}{\mathcal T_n} 
\newcommand{\tautilden}{\widetilde{\mathcal T}_n} 
\newcommand{\En}{\mathcal E_n} 
\newcommand{\EnI}{\mathcal E_n^I} 
\newcommand{\lllbracket}{\left \llbracket} 
\newcommand{\rrrbracket}{\right \rrbracket} 
\newcommand{\n}{\mathbf n} 
\newcommand{\RE}{R_\E} 
\newcommand{\REp}{R_{\E'}} 
\newcommand{\REi}{R_{\E_i}} 
\renewcommand{\Re}{R_\e} 
\newcommand{\Rebar}{\overline {\Re}} 
\newcommand{\p}{p} 
\newcommand{\pE}{\p_\E} 
\newcommand{\pEi}{\p_{\Ei}} 
\newcommand{\omegae}{\omega_\e} 
\renewcommand{\S}{S} 
\newcommand{\SE}{\S^\E} 
\newcommand{\SEi}{\S^{\E_i}} 
\newcommand{\Clement}{Cl\'ement{}} 
\newcommand{\T}{T} 
\newcommand{\pT}{\p_\T} 
\newcommand{\pe}{\p_\e} 
\newcommand{\VERT}{\mathbf V} 
\newcommand{\pV}{\p_\VERT} 
\newcommand{\ptildebold}{\mathbf p} 
\newcommand{\uI}{\u_I} 
\newcommand{\uM}{\u_M} 
\newcommand{\q}{q} 
\newcommand{\omegaE}{\omega_\E} 
\newcommand{\err}{e} 
\newcommand{\errI}{\err_I} 
\newcommand{\xE}{\mathbf x _\E}
\newcommand{\boldalpha}{\boldsymbol{\alpha}} 
\newcommand{\m}{m} 
\newcommand{\Vbold}{\VERT} 
\newcommand{\wtilde}{\widetilde w} 
\newcommand{\utildeI}{\widetilde \u _I} 
\newcommand{\upi}{\u_\pi} 
\newcommand{\psiII}{\psi_{II}} 
\newcommand{\aposteriori}{a posteriori } 
\newcommand{\bE}{\psi_\E} 
\newcommand{\be}{\psi_\e} 
\newcommand{\etapredpE}{\eta_{\text{pred},\E,n}} 
\newcommand{\etapredpEi}{\eta_{\text{pred},\Ei,n}} 
\newcommand{\etabarn}{\overline\eta_a} 
\newcommand{\zetaE}{\zeta_{\pE}} 
\newcommand{\rhoE}{\rho_{\pE}} 
\newcommand{\g}{g} 

\title{A posteriori error estimation and adaptivity in $hp$ virtual elements}
\date{}
\author{\normalsize{
L. Beir\~ao da Veiga\footnote{Dip. di Matematica e Applicazioni,  Universit\`a degli Studi di Milano-Bicocca, Via R. Cozzi, 55 - 20125 Milano, Italy, E-mail: {\tt lourenco.beirao@unimib.it}} $\,\,$,
M. Manzini \footnote{Group T-5, Theoretical Division, Los Alamos National Laboratory, Los Alamos, New Mexico - 87545, USA, E-mail: {\tt gmanzini@lanl.gov}},$\,\,$
L. Mascotto
\footnote{Faculty of Mathematics, University of Vienna, Oskar Morgenstern Platz 1, 1060, Wien, Austria, E-mail: {\tt lorenzo.mascotto@univie.ac.at}}
}
}
\begin{document}
\maketitle

\begin{abstract}
An explicit and computable error estimator for the $\h\p$ version of the virtual element method (VEM), together with lower and upper bounds with respect to the exact energy error, is presented.
Such error estimator is employed to provide, following the approach of \cite{MelenkWohlmuth_hpFEMaposteriori}, $\h\p$ adaptive mesh refinements for very general polygonal meshes.
In addition, a novel VEM $\h\p$ \Clement\,quasi-interpolant, instrumental for the a posteriori error analysis, is introduced. The performances of the adaptive method are validated by a number of numerical experiments.

	

\end{abstract}
\emph{The copyright of the original paper is owned by Springer. The original publication (DOI 10.1007/s00211-019-01054-6)  appears on the journal
Numerische Mathematik and can be found at \url{https://link.springer.com/article/10.1007/s00211-019-01054-6}.}

\section{Introduction} \label{section introduction}
Among the novel technologies developed in the last two decades, polygonal methods \cite{VEMvolley, cockburn_HDG, dipietroErn_hho, SukumarTabarraeipolygonalintroduction, BLM_MFD, talischi2010polygonal, lipnikov2014mimetic, Weisser_basic}
received an increasing attention, thanks to the advantages that they entail compared to methods based on standard simplicial/quadrilateral meshes. 
One of the appealing features of polygonal methods is that they dovetail particularly well with adaptive mesh refinements, since polygonal meshes do not require any post-processing within the adaptive procedure.

In the \emph{mare magnum} of the various theoretical and practical applications of mesh adaptivity, a particular place is occupied by the $\h\p$ a posteriori analysis in Galerkin methods and in particular in finite element method (FEM)
\cite{MelenkWohlmuth_hpFEMaposteriori, ainsworth1992procedure, dolejsi2016hp, canuto2017convergence, HoustonSuli2005note, hpDGaposTime}.
Here, one combines mesh refinements with the increase of the dimension of the local approximation spaces.

This results in particularly competitive methods able to capture and to solve in an efficient way the singular behaviour of solutions to partial differential equations (PDEs), but comes at the price of a more involved framework.
For instance, one has to face with the problem of deciding whether to refine in $\h$ (mesh refinement) or in $\p$ (increasing the dimension of local space),
has to construct an error estimator with explicit bounds in terms of the distribution of local polynomial degrees, and has to construct spaces possibly tailored for high polynomial order (such that, e.g., the condition number of the stiffness matrix is sufficiently robust in terms of $\p$).

Amidst the polygonal methods, we pinpoint the virtual element method (in short VEM, introduced in \cite{VEMvolley,hitchhikersguideVEM}), which has established itself as one of the most successful technologies in this context and enjoyed an increasing community
(a very limited list of papers being \cite{VEM_DD_basic, BLV_StokesVEMdivergencefree, gain2014virtual, nonconformingVEMbasic, Helmholtz-VEM, vaccaBrink, AMV2018fully,Benedetto-VEM-2, Wriggers-contact, brennerVEM, Gatica-1, VEMchileans}).
Differently from other polygonal approaches, the VEM is built without an explicit knowledge of the functions in the approximation spaces.
The construction of the method is based on two main ingredients,
which are computable only via the degrees of freedom with whom the spaces are endowed:
projections into polynomial spaces and particular bilinear forms stabilizing the method.
Virtual elements are particularly suitable for the $hp$ framework not only due to the mesh refinement flexibility, but also due to the fact that variable (from element to element) polynomial degrees can be included very naturally in the formulation, without resorting to``ad hoc'' modifications as in finite elements.

The a posteriori version of VEM was investigated in \cite{cangianigeorgulispryersutton_VEMaposteriori, MoraRiveraRodriguez_aposSteklov, berrone2017residual, ManziniBeirao_VEMresidualaposteriori, mixedVEMapos}
but the $\h\p$ adaptivity has never been targeted before and is the topic of the present work.
This paper falls in a series of articles covering different aspects of the $\p$ and $\h\p$ version of VEM, namely a priori error analysis on quasi-uniform meshes \cite{hpVEMbasic}, approximation of corner singularities \cite{hpVEMcorner},
 multigrid algorithms \cite{pVEMmultigrid}, ill-conditioning in two \cite{fetishVEM} and three dimensional problems \cite{fetishVEM3D}, Trefftz and non-conforming approaches \cite{conformingHarmonicVEM, ncHVEM},
 theoretical and numerical analysis of the stabilization typical of VEM \cite{MascottoPhDthesis}.

In the present work, we introduce a VEM \emph{computable} error estimator for a two dimensional model problem and we prove lower and upper bounds with respect to the energy
error that are explicit both in the mesh size and in the distribution of local degrees of accuracy (which are the VEM counterpart of polynomial degrees for FEM).
Additionally, we recall from \cite{MelenkWohlmuth_hpFEMaposteriori, MascottoPhDthesis} some $\h\p$ inverse inequalities on triangles and polygons.
Besides, we introduce and prove error estimates of a novel $\h\p$ \Clement\,type quasi-interpolant, which is constructed starting from the quasi-interpolant of \cite{melenk2003hp} on subtriangulations of the underlying polygonal meshes.
Such estimates are an interesting result on their own, independently of their application in the a posteriori error analysis.

We also present many numerical experiments; more precisely, after verifying the lower and upper bounds of the error estimator, we recall from \cite{MelenkWohlmuth_hpFEMaposteriori}
an $\h\p$ refinement strategy and we apply it to the proposed adaptive scheme.
Such refinement strategy heuristically suggests which are the marked elements where the solutions is supposed to be ``smooth'' or ``singular'';
in the former case, $\p$ refinement are effectuated (since $\p$ VEM converges exponentially when approximating analytic solutions, see \cite{hpVEMbasic}),
in the latter, $\h$ refinement are performed (geometric refinements towards the singularities also give exponential convergence in terms of the number of degrees of freedom, see \cite{hpVEMcorner}).

Possible further developments of this work include extension to more complex problems and to the three dimensional case. Moreover, one could also take into account in the adaptive strategy an $\h\p$ derefinement process including the agglomeration of elements.

\medskip
The outline of the paper is the following.
We begin in Section \ref{section model problem} by presenting the model problem and the $\h\p$ version of the virtual element method and next,
in Section \ref{section technical tools}, we introduce a set of technical tools needed for the a posteriori error analysis; in particular, we discuss about approximation by functions in the virtual element space,
$\h\p$ polynomial inverse estimates and extension operators from an edge into the interior of a triangle.
Successively, in Section \ref{section a posteriori error analysis} we deal with the a posteriori error analysis:
we introduce an error estimator and we assert the ``reliability'' and the ``efficiency'' of such an estimator.
Finally, a number of numerical experiments, including a comparison with the pure $\h$ refinement strategy, are presented in Section~\ref{section numerical results}.

\paragraph*{Notation}
In the remainder of the paper, we will employ the following notation.
Given $D$ a measurable open set in $\mathbb R^2$, we denote by $L^2(D)$ and $H^s(D)$, $s \in \mathbb N$, the standard Lebesgue and Sobolev spaces endowed with the standard inner products and seminorms:
\[
(\cdot, \cdot)_{s,D};\quad \quad \vert \cdot \vert_{s,D}.
\]
The Sobolev norms are denoted by:
\[
\Vert \cdot \Vert^2_{s,D} = \sum_{k=0}^s \vert \cdot \vert^2_{k,D}.
\]
Sobolev spaces with fractional order are defined via interpolation theory \cite{Triebel}.

It is worth to highlight with a separate notation the $H^1$ inner product over domain $D$:
\[
\a^D (\cdot, \cdot) := (\cdot, \cdot)_{1,D}.
\]
We refer to \cite{adamsfournier} for the definition of spaces, inner products and (semi)norms.

Given $\ell \in \mathbb N$, we set $\mathbb P_\ell(D)$ to be the space of polynomials of degree $\ell$ over $D$.
Finally, we write $a \lesssim b$ and $a \approx b$ in lieu of: there exist positive constants $c_1$, $c_2$ and $c_3$ independent of the polynomial degree and of the mesh-size such that $a \le c_1 b$ and $c_2 a \le b \le c_3 a$, respectively.

\section{Model problem and the $\h\p$ version of the virtual element method} \label{section model problem}
In this section, we briefly review the model problem and the $\h\p$ version of the virtual element method.

Given $\Omega \subset \mathbb R^2$ a polygonal domain and $\f \in L^2(\Omega)$, we consider the Poisson problem with homogeneous Dirichlet boundary conditions
\begin{equation} \label{strong Poisson problem}
\begin{cases}
-\Delta  \u = \f & \text{in } \Omega\\
\u = 0 & \text{on } \partial \Omega\\
\end{cases}
\end{equation}
and its weak formulation
\begin{equation} \label{Poisson problem homogeneous Dirichlet weak formulation}
\begin{cases}
\text{find } \u \in \V \text{ such that}\\
\a(\u,\vv) = (\f, \vv)_{0,\Omega} \quad \forall \vv \in \V,
\end{cases}
\end{equation}
where we have set
\[
\V = H_0^1 (\Omega),\quad \a(\u, \vv) = \a^\Omega (\u, \vv) = \int _\Omega \nabla \u \cdot \nabla \vv\quad \forall \u,\,\vv \in \V.
\]
Next, we introduce a virtual element method based on polygonal meshes and with nonuniform local degree of accuracy for the approximation of problem \eqref{Poisson problem homogeneous Dirichlet weak formulation}.
We note that virtual element methods with nonuniform degree of accuracy were firstly introduced in \cite{hpVEMcorner}.

Given $\{\taun\}_{n\in \mathbb N}$ a sequence of decomposition  of $\Omega$ into polygons with straight edges, we associate to each $\taun$, $n \in \mathbb N$, its set of vertices $\mathcal V _n$ and boundary vertices $\mathcal V_n^b$,
its set of edges $\mathcal E_n$ and boundary edges $\mathcal E_n^b$.
To each $\E$ in $\taun$, we associate its diameter $\h_\E$, its barycenter $\xE$, its set of vertices $\mathcal V^\E$ and its set of edges $\mathcal E^\E$.
To each edge $\e \in \mathcal E_n^b$, we associate once and for all a unit normal $\n_{\mathbf \e} = \n$.

We require that $\taun$ is a conforming polygonal decomposition of $\Omega$ for every $n \in \mathbb N$, i.e.
every internal edge belongs to the intersection of the boundary of two polygons.

Besides, we demand the two following geometric assumptions on sequence $\taun$.
\begin{itemize}
\item[(\textbf{D1})] Every $\E \in \taun$ is star-shaped with respect to a ball (see \cite{BrennerScott}) with radius greater than or equal to $\gamma \h_\E$, where $\gamma$ is a universal positive constant.
\item[(\textbf{D2})] Given $\E \in \taun$, each of his edges $\e \in \mathcal E^\E$ has length greater than or equal to $\widetilde \gamma \hE$, where $\widetilde \gamma$ is a universal positive constant.
\end{itemize}
A consequence of assumptions (\textbf{D1}) and (\textbf{D2}) which will be extensively used in the following is highlighted in Remark \ref{remark triangular subdecomposition}.
\begin{remark} \label{remark triangular subdecomposition}
Owing to assumptions (\textbf{D1}) and (\textbf{D2}), the following fact holds true.
The subtriangulation $\tautilden = \tautilden (\E)$ of $\E$, obtained by joining the vertices of $\E$ to the center of the ball with respect to which $\E$ is star-shaped, is made of triangles that are star-shaped with respect to balls with radius
greater than or equal to $\overline \gamma \hE$, where $\overline \gamma$ is a universal positive constant depending only on $\gamma$ and $\widetilde \gamma$.
For a proof of this fact, see \cite[Chapter 2]{MascottoPhDthesis}.
In the forthcoming analysis, we will often make use of standard functional inequalities (Poincar\'e, trace, \dots).
If not explicitly mentioned, the constants in such inequalities, and therefore also in all the results we will present, depend solely on the shape regularity of $\taun$ and $\tautilden$.
\end{remark}
Next, we associate to each element $\E \in \taun$ a local degree of accuracy $\pE$. To each boundary edge $\e \in \mathcal E^b_n$,
we associate $\pe$ equal to $\p_{\Etilde}$, where $\Etilde$ is the unique polygon having $\e$ as an edge.
On the other hand, to each internal edge $\e \in \mathcal E_n \setminus \mathcal E_n^b$ we associate $\pe$ equal to $\max(\p_{\E_1}, \p_{\E_2})$, where $\e = \overline{\E_1} \cap \overline{\E_2}$.

Given now $\tautilden$ the regular subtriangulation of $\Omega$ obtained by gluing the local triangulations introduced in Remark \ref{remark triangular subdecomposition},
we let $\widetilde {\mathcal V}_n$ and $\widetilde{\mathcal E}_n$ be the sets of vertices and edges of $\tautilden$, respectively.
To each $\T \in \tautilden$ with $\T \subset \E$, we associate $\pT$ equal to $\pE$.
We denote by $\mathbf \p$ the vector of degrees of accuracy associated with polygonal mesh $\taun$,
whereas we denote by $\widetilde{\mathbf \p}$ the vector of degrees associated with triangular subdecomposition $\tautilden$.
To each edge~$\etilde$ of~$\T$ in the boundary of~$\E$ we associate $\p_{\etilde}$ equal to $\pe$, where $\e =  \overline T \cap \overline \E = \widetilde \e$.
To the other edges $\etilde$ of $\T$, we simply associate $\p_{\etilde}$ equal to $\pE$.
We also associate with each vertex $\VERT \in \widetilde{\mathcal V}_n$
\begin{equation} \label{power vertex}
\pV =  \max \left\{ \pT \mid \VERT \in \overline \T,\, \T \in \tautilden \right\}.
\end{equation}
See Figure \ref{figure 1} for a graphical idea regarding the distribution of the degrees over~$\taun$, $\tautilden$, $\mathcal E_n$, and~$\widetilde{\mathcal E}_n$.

\begin{figure}[H]
\centering
\begin{minipage}{0.23\textwidth}
\begin{center}
\begin{tikzpicture}[scale=1.2]
\draw[black, thick, -] (0,0) -- (1,0) -- (1,1) -- (0,1) -- (0,0);
\draw[black, thick, -] (0,0) -- (-0.5, -0.5) -- (-1,0) -- (-1,1) -- (-0.5, 1.5) -- (0,1) -- (0,0);
\draw[black, thick, -] (0,0) -- (1,0) -- (1,-1) -- (-0.5, -1) -- (-0.5, -0.5) --  (0,0);
\draw (0.5,0.5) node[black] {\tiny{$2$}}; \draw (-0.5,0.5) node[black] {\tiny{$4$}}; \draw (0.23,-0.5) node[black] {\tiny{$3$}};
\draw (-1,0) node[black, fill=white] {\tiny{$4$}}; \draw (-1,1) node[black, fill=white] {\tiny{$4$}}; \draw (-0.5,1.5) node[black, fill=white] {\tiny{$4$}};
\draw (0,1) node[black, fill=white] {\tiny{$4$}}; \draw (0,0) node[black, fill=white] {\tiny{$4$}}; \draw (-0.5,-0.5) node[black, fill=white] {\tiny{$4$}};
\draw (1,0) node[black, fill=white] {\tiny{$4$}}; \draw (1,-1) node[black, fill=white] {\tiny{$3$}}; \draw (-0.5,-1) node[black, fill=white] {\tiny{$3$}};
\draw (1,1) node[black, fill=white] {\tiny{$2$}};
\end{tikzpicture}
\end{center}
\end{minipage}
\begin{minipage}{0.23\textwidth}
\begin{center}
\begin{tikzpicture}[scale=1.2]
\draw[black, thick, -] (0,0) -- (1,0) -- (1,1) -- (0,1) -- (0,0);
\draw[black, thick, -] (0,0) -- (-0.5, -0.5) -- (-1,0) -- (-1,1) -- (-0.5, 1.5) -- (0,1) -- (0,0);
\draw[black, thick, -] (0,0) -- (1,0) -- (1,-1) -- (-0.5, -1) -- (-0.5, -0.5) --  (0,0);
\draw (0.5,1) node[black, fill=white] {\tiny{$2$}}; \draw (1,0.5) node[black, fill=white] {\tiny{$2$}}; \draw (0.5,0) node[black, fill=white] {\tiny{$3$}}; \draw (0,0.5) node[black, fill=white] {\tiny{$4$}};
\draw (1,-0.5) node[black, fill=white] {\tiny{$3$}}; \draw (0.25,-1) node[black, fill=white] {\tiny{$3$}}; \draw (-0.5,-0.75) node[black, fill=white] {\tiny{$3$}}; \draw (-0.25,-0.25) node[black, fill=white] {\tiny{$4$}};
\draw (-0.75,-0.25) node[black, fill=white] {\tiny{$4$}};  \draw (-1,0.5) node[black, fill=white] {\tiny{$4$}}; \draw (-0.75,1.25) node[black, fill=white] {\tiny{$4$}}; \draw (-0.25,1.25) node[black, fill=white] {\tiny{$4$}};
\end{tikzpicture}\end{center}
\end{minipage}
\begin{minipage}{0.23\textwidth}
\begin{center}
\begin{tikzpicture}[scale=1.2]
\draw[black, thick, -] (0,0) -- (1,0) -- (1,1) -- (0,1) -- (0,0);
\draw[black, thick, -] (0,0) -- (-0.5, -0.5) -- (-1,0) -- (-1,1) -- (-0.5, 1.5) -- (0,1) -- (0,0);
\draw[black, thick, -] (0,0) -- (1,0) -- (1,-1) -- (-0.5, -1) -- (-0.5, -0.5) --  (0,0);
\draw[black, dashed,thin, -]   (-0.5,-0.5) -- (-0.5,0.5); \draw[black, dashed,thin, -]   (-1,0) -- (-0.5,0.5); \draw[black, dashed,thin, -]   (-1,1) -- (-0.5,0.5);
\draw[black, dashed,thin, -]   (-0.5,1.5) -- (-0.5,0.5); \draw[black, dashed,thin, -]   (0,1) -- (-0.5,0.5); \draw[black, dashed,thin, -]   (0,0) -- (-0.5,0.5);
\draw[black, dashed,thin, -]   (0.5,0.5) -- (1,1); \draw[black, dashed,thin, -]   (0.5,0.5) -- (1,0); \draw[black, dashed,thin, -]   (0.5,0.5) -- (0,1); \draw[black, dashed,thin, -]   (0.5,0.5) -- (0,0);
\draw[black, dashed,thin, -]   (0.25,-0.5) -- (0,0); \draw[black, dashed,thin, -]   (0.25,-0.5) -- (1,0); \draw[black, dashed,thin, -]   (0.25,-0.5) -- (1,-1); \draw[black, dashed,thin, -]   (0.25,-0.5) -- (-0.5,-1); \draw[black, dashed,thin, -]   (0.25,-0.5) -- (-0.5,-0.5);
\draw (0.8,0.5) node[black] {\tiny{$2$}}; \draw (0.2,0.5) node[black] {\tiny{$2$}}; \draw (0.5,0.2) node[black] {\tiny{$2$}}; \draw (0.5,0.8) node[black] {\tiny{$2$}}; \draw (0.5,0.5) node[black, fill=white] {\tiny{$2$}}; 
\draw (-.2,0.5) node[black] {\tiny{$4$}}; \draw (-.8,0.5) node[black] {\tiny{$4$}}; \draw(-.3,1) node[black] {\tiny{$4$}}; \draw(-.7,1) node[black] {\tiny{$4$}}; \draw(-.3,0) node[black] {\tiny{$4$}}; \draw(-.7,0) node[black] {\tiny{$4$}}; \draw (-.5,0.5) node[black, fill=white] {\tiny{$4$}};
\draw (.7,-0.5) node[black] {\tiny{$3$}}; \draw (.35,-.2) node[black] {\tiny{$3$}}; \draw (.25,-.8) node[black] {\tiny{$3$}}; \draw (-0.05,-.35) node[black] {\tiny{$3$}}; \draw (-.3,-.7) node[black] {\tiny{$3$}}; \draw (.3,-.5) node[black, fill=white] {\tiny{$3$}};
\end{tikzpicture}
\end{center}
\end{minipage}
\begin{minipage}{0.23\textwidth}
\begin{center}
\begin{tikzpicture}[scale=1.2]
\draw[black, thick, -] (0,0) -- (1,0) -- (1,1) -- (0,1) -- (0,0);
\draw[black, thick, -] (0,0) -- (-0.5, -0.5) -- (-1,0) -- (-1,1) -- (-0.5, 1.5) -- (0,1) -- (0,0);
\draw[black, thick, -] (0,0) -- (1,0) -- (1,-1) -- (-0.5, -1) -- (-0.5, -0.5) --  (0,0);
\draw[black, dashed,thin, -]   (-0.5,-0.5) -- (-0.5,0.5); \draw[black, dashed,thin, -]   (-1,0) -- (-0.5,0.5); \draw[black, dashed,thin, -]   (-1,1) -- (-0.5,0.5);
\draw[black, dashed,thin, -]   (-0.5,1.5) -- (-0.5,0.5); \draw[black, dashed,thin, -]   (0,1) -- (-0.5,0.5); \draw[black, dashed,thin, -]   (0,0) -- (-0.5,0.5);
\draw[black, dashed,thin, -]   (0.5,0.5) -- (1,1); \draw[black, dashed,thin, -]   (0.5,0.5) -- (1,0); \draw[black, dashed,thin, -]   (0.5,0.5) -- (0,1); \draw[black, dashed,thin, -]   (0.5,0.5) -- (0,0);
\draw[black, dashed,thin, -]   (0.25,-0.5) -- (0,0); \draw[black, dashed,thin, -]   (0.25,-0.5) -- (1,0); \draw[black, dashed,thin, -]   (0.25,-0.5) -- (1,-1); \draw[black, dashed,thin, -]   (0.25,-0.5) -- (-0.5,-1); \draw[black, dashed,thin, -]   (0.25,-0.5) -- (-0.5,-0.5);
\draw (0.75,0.75) node[black, fill=white] {\tiny{$2$}}; \draw (0.75,0.25) node[black, fill=white] {\tiny{$2$}}; \draw (0.25,0.75) node[black, fill=white] {\tiny{$2$}}; \draw (0.25,0.25) node[black, fill=white] {\tiny{$2$}};
\draw (0.65,-0.75) node[black, fill=white] {\tiny{$3$}}; \draw (0.65,-0.25) node[black, fill=white] {\tiny{$3$}}; \draw (0.2,-0.25) node[black, fill=white] {\tiny{$3$}}; \draw (-.125,-0.5) node[black, fill=white] {\tiny{$3$}}; \draw (-.125,-0.75) node[black, fill=white] {\tiny{$3$}};
\draw (-.5,0) node[black, fill=white] {\tiny{$4$}}; \draw (-.5,1) node[black, fill=white] {\tiny{$4$}}; \draw (-.25,.75) node[black, fill=white] {\tiny{$4$}}; \draw (-.75,.75) node[black, fill=white] {\tiny{$4$}}; 
\draw (-.25,.25) node[black, fill=white] {\tiny{$4$}};  \draw (-.75,.25) node[black, fill=white] {\tiny{$4$}};
\end{tikzpicture}
\end{center}
\end{minipage}
\caption{Local degrees of accuracy distribution. Left: distribution over $\taun$, the polygonal decomposition, and $\mathcal V_n$ of $\taun$.
Center-left: distribution over $\mathcal E_n$, the set of edges of $\taun$. Center-right: distribution over $\tautilden$, the subtriangulation of $\taun$, and $\widetilde{\mathcal V}_n$, the set of vertices of $\tautilden$.
Right: distribution over $\widetilde{\mathcal E}_n$, the set of edges of $\tautilden$.}
\label{figure 1}
\end{figure}
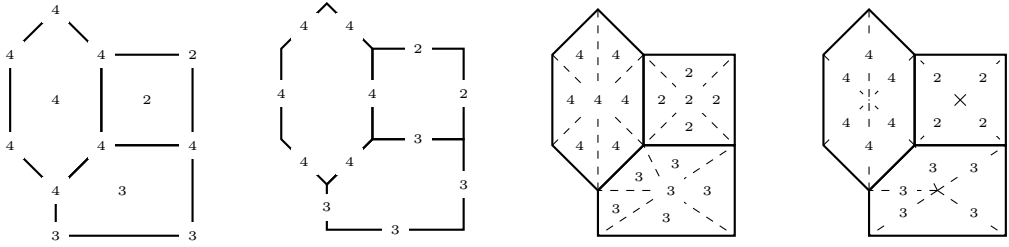
We also ask for the following assumption on the local degree of accuracy distribution associated with $\taun$:
\begin{itemize}
\item[(\textbf{P1})] for every $\E$ and $\E'$ in $\taun$ with~$\overline \E \cap \overline{\E'}\ne \emptyset $, there exists a positive universal constant $c$ such that
\begin{equation} \label{assumption local degrees of accuracy}
\vert \pE - \p_{\E'}    \vert  \le c.
\end{equation}
\end{itemize}
As a consequence of the assumption (\textbf{P1}) and the above construction, it also holds $\pe \approx \pE$ whenever $\e$ is an edge of $\E$.

We introduce now the virtual element method associated with the Poisson problem \eqref{Poisson problem homogeneous Dirichlet weak formulation}. In particular, following \cite{VEMvolley}, we introduce a finite dimensional subspace $\Vn$ of $\V$, a discrete bilinear form
$\an : \Vn \times \Vn \rightarrow \mathbb R$ and a discrete right-hand side $\fn$, such that the method
\begin{equation} \label{hp VEM}
\begin{cases}
\text{find } \un \in \Vn \text{ such that}\\
\an(\un, \vn) = \langle \fn , \vn \rangle_n \quad \forall \vn \in \Vn,
\end{cases}
\end{equation}
is well-posed and its output $\un$ approximates the solution $\u$ of \eqref{Poisson problem homogeneous Dirichlet weak formulation}.

In Section \ref{subsection virtual space}, we describe the virtual element space $\Vn$. Next, in Section \ref{subsection discrete bilinear form}, we present a computable choice for the discrete bilinear form $\an$.
Finally, in Section \ref{section discrete right-hand side}, we discuss the construction of the discrete right-hand side $\langle \fn, \cdot \rangle_n$.


\subsection{The virtual element space} \label{subsection virtual space}
The present section is devoted to introduce the virtual element space $\Vn$.

We begin by defining the local virtual element spaces. Given $\E \in \taun$, we set first
\[
\mathbb B (\partial \E) = \{\vn \in \mathcal C^0(\partial \E) \mid \vv _{n|_\e} \in \mathbb P_{\pe}(\e) \text{ for all edge } \e \text{ in } \mathcal E^\E\}.
\]
The local virtual element space over polygon $\E$ reads
\begin{equation} \label{local spaces}
\VnE = \{ \vn \in H^1(\E) \mid  \Delta \vn \in \mathbb P_{\pE-2}(\E),\; \vv_{n |_{\partial \E}} \in \mathbb B (\partial \E) \}.
\end{equation}
Let us consider the following set of linear functionals defined on $\V(\E)$: given $\vn \in \VnE$,
\begin{itemize}
\item the point-values of $\vn$ at the vertices of $\E$;
\item for every edge $\e \in \mathcal E^\E$, the point-values at the $\pe-1$ internal Gau\ss-Lobatto nodes on $\e$;
\item the internal moments
\begin{equation} \label{internal moments}
\frac{1}{\vert \E \vert} \int \m_{\boldalpha} \vn \quad \forall \boldalpha \in \mathbb N^2,\quad \vert \boldalpha \vert := \alpha_1+\alpha_2 \le \pE-2,
\end{equation}
where $\{\m_{\boldalpha},\, \boldalpha \in \mathbb N^2,\, \vert \boldalpha \vert \le \pE-2\}$ is \emph{any} basis of $\mathbb P_{\pE-2}(\E)$, provided that $\Vert \m_{\boldalpha} \Vert _{\infty ,\E}\approx 1$ and that it is invariant with respect to homothetic transformations.
\end{itemize}
These linear functionals are a set of unisolvent degrees of freedom (dofs), see \cite{VEMvolley}.
\begin{remark} \label{remark internal moments}
The choice of the polynomial basis dual to internal moments play a fundamental role in the prospective ill-conditioning of the high-order virtual element method.
We refer to \cite{fetishVEM, fetishVEM3D} for a deep insight on this topic and to Section \ref{section numerical results} for an explicit choice of such polynomial basis.
\end{remark}
Although the space $\VnE$ is not defined explicitly inside the element $\E$, by means of the set of degrees of freedom above,
it is possible to compute a couple of local projectors, see \cite{hitchhikersguideVEM}. The first operator is the $L^2$ projector $\PizpmdE : \VnE \rightarrow \mathbb P_{\pE - 2}(\E)$ defined as
\begin{equation} \label{L2 projection}
(\vn - \PizpmdE \vn, \q)_{0,\E} = 0 \quad \forall \vn \in \VnE,\; \forall \q \in \mathbb P_{\pE-2} (\E).
\end{equation}
The second one is an $H^1$ projector $\PinablaE : \VnE \rightarrow \mathbb P_{\pE} (\E)$ defined as
\begin{equation} \label{H1 projector}
\begin{cases}
\aE(\vn - \PinablaE \vn, \q) = 0\\
\int_{\partial \E} \vn - \PinablaE \vn = 0
\end{cases} \quad \forall \vn \in \VnE,\; \forall \q \in \mathbb P_{\pE}(\E).
\end{equation}
When no confusion occurs we will write $\Pizpmd$ and $\Pinabla$ in lieu of $\PizpmdE$ and $\PinablaE$.

The global virtual element space is obtained as in FEM by a standard conforming coupling of the local sets of degrees of freedom and by enforcing homogeneous boundary conditions on $\partial \Omega$:
\begin{equation} \label{global space}
\Vn = \{ \vn \in H^1_0(\Omega) \cap \mathcal C^0(\overline \Omega) \mid \vv_{n |_{\E}} \in \VnE\, \forall \E\in \taun  \}.
\end{equation}
We also introduce the set of global degrees of freedom and the global canonical basis
\begin{equation} \label{canonical basis}
\{\text{dof}_i\}_{i=1}^{\dim(\Vn)},   \quad\quad\{\varphi_j\}_{j=1}^{\dim(\Vn)} \text{ such that } \text{dof}_i(\varphi_j) =\delta_{i,j},
\end{equation}
where $\delta_{i,j}$ is the Kronecker delta.
\begin{remark} \label{remark on bcs}
We point out that if we aim to approximate a Poisson problem \eqref{strong Poisson problem} with a non-homogeneous Dirichlet boundary condition $\g \in H^{\frac{1}{2}}(\partial \Omega)$,
then the global space \eqref{global space} is modified into
\[
\Vn^{\g} = \{\vn \in H^1(\Omega) \mid \vv_{n|_\e} = \g_{GL|_\e}\, \forall \e \in \mathcal E_n^b,\quad   \vv_{n |_{\E}} \in \VnE\, \forall \E\in \taun\},
\]
where~$\g_{GL|_\e}$ denotes the Gau\ss-Lobatto interpolant of~$\g$ over the boundary edge~$\e$. In fact, Gau\ss-Lobatto interpolation on the boundary guarantees optimal approximation properties for the~$\p$ and~$\h\p$ versions of VEM, see~\cite{conformingHarmonicVEM}.
However, in order to avoid cumbersome technicalities in the following, we assume henceforth to deal with homogeneous Dirichlet boundary conditions.
\end{remark}

\subsection{The discrete bilinear form} \label{subsection discrete bilinear form}
Here, we introduce the discrete bilinear form $\an$ associated with method \eqref{hp VEM}.
To this purpose, we follow the guidelines of \cite{VEMvolley, hpVEMbasic}.
Given $\anE : \VnE \times \VnE \rightarrow \mathbb R$ local bilinear forms, the (global) discrete bilinear form is a sum of such local contributions:
\[
\an(\un, \vn) = \sum_{\E \in \taun} \anE(\un, \vn) \quad \forall \un, \, \vn \in \Vn,
\]
where the local discrete bilinear forms $\anE$ are assumed to satisfy
\begin{itemize}
\item[(\textbf{A1})]\textbf{polynomial consistency}: for every $\E \in \taun$,
\[
\anE(\q, \vn) = \aE(\q, \vn) \quad \forall \q \in \mathbb P_{\pE} (\E),\; \forall \vn \in \VnE;
\]
\item[(\textbf{A2})] \textbf{local stability}: for every $\E \in \taun$, there exist two positive constants $\alpha_*(\pE) \le \alpha^*(\pE)$ independent of $\hE$ but possibly depending on the local degree of accuracy $\pE$, such that
\begin{equation} \label{stability}
\alpha_*(\pE) \vert \vn \vert^2_{1,\E} \le \anE (\vn, \vn) \le \alpha^*(\pE) \vert \vn \vert^2 _{1,\E} \quad \forall \vn \in \VnE.
\end{equation}
\end{itemize}
The assumption (\textbf{A1}) implies that the discrete bilinear form is locally exact for piecewise polynomial functions having a proper degree.
On the other hand, the assumption (\textbf{A2}) guarantees the coercivity and the continuity of the global discrete bilinear form and hence the well-posedness of method \eqref{hp VEM}.

A number of computable discrete bilinear forms can be found in the literature. We refer to \cite{hpVEMcorner} for what concerns the first ``$\p$-explicit'' stabilization available in the VEM literature
and to \cite[Chapter 6]{MascottoPhDthesis} and to \cite{beiraolovadinarusso_stabilityVEM} for a survey of the topic.

Following \cite{VEMvolley}, the local bilinear forms $\anE$ can be built with the aid of an auxiliary local stabilizing bilinear form $\SE : \ker (\Pinabla) \times \ker (\Pinabla) \rightarrow \mathbb R$.
For every $\E \in \taun$, the bilinear form $\SE$ must satisfy
\begin{equation} \label{stabilizing bilinear form}
c_*(\pE) \vert \vn \vert_{1,\E}^2 \le \SE(\vn, \vn) \le c^*(\pE) \vert \vn \vert^2_{1,\E} \quad \forall \vn \in \ker (\Pinabla),
\end{equation}
where $c_*(\pE)$ and $c^*(\pE)$ are two positive constants independent of $\hE$ but possibly depending on the local degree of accuracy $\pE$. The local discrete bilinear form
\[
\anE(\un, \vn) = \aE(\Pinabla \un, \Pinabla \vn) + \SE((I- \Pinabla) \un, (I- \Pinabla) \vn) \quad \forall \un, \, \vn \in \VnE,
\]
where the energy projector $\Pinabla$ is defined in \eqref{H1 projector}, satisfies assumptions (\textbf{A1}) and (\textbf{A2}) with
$\alpha_*(\pE) = \min (1, c_*(\pE))$ and $\alpha^*(\pE) = \max (1, c^*(\pE))$, see \cite{VEMvolley}.

An explicit choice of the stabilization, and consequently also of the discrete bilinear form, will be discussed in Section \ref{section numerical results}.

\subsection{The discrete right-hand side} \label{section discrete right-hand side}
Here, we present the discrete right-hand side of method \eqref{hp VEM}.
In particular, we split the global discrete right-hand side $\langle \fn, \cdot \rangle _n$ as a sum of local contributions:
\[
\langle \fn , \vn \rangle_{n} = \sum_{\E \in \taun} \langle \fn, \vn \rangle_{n,\E},
\]
where the local discrete duality pairing are defined as
\[
\langle \fn, \vn \rangle_{n,\E} = 
\begin{cases}
\int_\E \f \, \Pizpmd\vn & \text{if } \pE \ge 2\\
\int_\E \left(\f \, \int_{\partial \E} \vn   \right) & \text{if } \pE =1.\\
\end{cases}
\]
For a deeper analysis concerning the discrete right-hand side, we refer to \cite{equivalentprojectorsforVEM, VEMelasticity, hpVEMbasic}.

\section{Technical tools} \label{section technical tools}
In this section, we introduce and discuss some technical tools that we will employ in the \aposteriori error analysis.
In Section \ref{subsection approximation estimates by means of functions in the Virtual Space}, we prove local $L^2$ and $H^1$ $\h\p$ approximation properties in terms of functions in the virtual element space.
In Section \ref{subsection inverse estimates on triangles and polygons}, we recall a couple of $\h\p$ polynomial inverse estimates, whereas,
in Section \ref{subsection a lifting operator}, we recall from \cite{MelenkWohlmuth_hpFEMaposteriori} the existence of a lifting operator of a polynomial from any edge of a triangle into its interior with good stability properties.

Henceforth, we adopt the following notation concerning spaces of piecewise continuous polynomials over triangular meshes. Given~$\tautilden$ a triangular mesh, we write
\begin{equation} \label{space piecewise continuous polynomials}
 S^{{\mathbf {\widetilde \p}},0}(\Omega, \tautilden) = \left\{ \q \in \mathcal C^0(\Omega) \mid \q_{|_\T} \in \mathbb P _{\widetilde {\p}_\T} (\T) \, \forall \T \in \tautilden,\;
					\q_{|_\e} \in \mathbb P_{\widetilde \p_\e}(\e)\, \forall \e \text{ edges of } \tautilden\right\},
\end{equation}
where the choice of the polynomial degrees $\widetilde \p_\e$ on the edges is picked as in Section \ref{section model problem},
see Figure~\ref{figure 1}.

\subsection{Local $\h\p$ approximation estimates} \label{subsection approximation estimates by means of functions in the Virtual Space}
In this section, we discuss about approximation properties of functions in the local virtual element spaces $\VnE$ defined in \eqref{local spaces}.
In particular, we prove local approximation estimates in the $L^2$ and $H^1$ (semi)norms and in the $L^2$ norm on the boundary.
Previous results on $\h\p$-VEM approximation can be found in \cite{hpVEMbasic, hpVEMcorner}.

We first need to recall a technical tool from \cite{MelenkWohlmuth_hpFEMaposteriori, melenk2003hp}.
Given $\tautilden$ the (regular) triangular subdecomposition of $\Omega$, defined locally by the subtriangulations of the polygons of $\taun$ introduced in Remark \ref{remark triangular subdecomposition},
and given a vertex $\VERT \in \widetilde{\mathcal V}_n$ of subtriangulation $\tautilden$, we set the \emph{triangular} patch around vertex $\VERT$ as
\begin{equation} \label{triangular patch around V}
\omega_{\Vbold} = \bigcup \left\{ \T \in \tautilden \mid \Vbold \text{ is a vertex of } \T  \right\}.
\end{equation}
Given $D$ either a polygon of $\taun$ or a triangle of $\tautilden$, we define the \emph{triangular} patch around $D$ as
\begin{equation} \label{triangular patch around D}
\omega_D := \bigcup \left\{ \omega_{\VERT} \mid \VERT \text{ is a vertex of } D \right\}.
\end{equation}
Besides, given an edge $\e \in \mathcal E_n$, we also define the \emph{triangular} patch around $\e$ as
\begin{equation} \label{triangular patch around e}
\omega_{\e} := \bigcup \left\{ \omega_{\VERT} \mid \VERT \text{ is an endpoint of } \e   \right\}.
\end{equation}
In Figures \ref{figure triangular patch around V}, \ref{figure triangular patch around K},  \ref{figure triangular patch around T} and \ref{figure triangular patch around e} we provide graphical examples of patches
around a vertex, a polygon, a triangle and an edge.

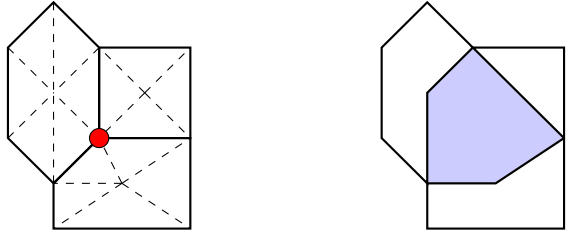
\begin{figure}[h]
\centering
\begin{minipage}{0.32\textwidth}
\begin{center}
\begin{tikzpicture}[scale=1.2]
\draw[black, thick, -] (0,0) -- (1,0) -- (1,1) -- (0,1) -- (0,0);
\draw[black, thick, -] (0,0) -- (-0.5, -0.5) -- (-1,0) -- (-1,1) -- (-0.5, 1.5) -- (0,1) -- (0,0);
\draw[black, thick, -] (0,0) -- (1,0) -- (1,-1) -- (-0.5, -1) -- (-0.5, -0.5) --  (0,0);
\draw[black, dashed,thin, -]   (-0.5,-0.5) -- (-0.5,0.5); \draw[black, dashed,thin, -]   (-1,0) -- (-0.5,0.5); \draw[black, dashed,thin, -]   (-1,1) -- (-0.5,0.5);
\draw[black, dashed,thin, -]   (-0.5,1.5) -- (-0.5,0.5); \draw[black, dashed,thin, -]   (0,1) -- (-0.5,0.5); \draw[black, dashed,thin, -]   (0,0) -- (-0.5,0.5);
\draw[black, dashed,thin, -]   (0.5,0.5) -- (1,1); \draw[black, dashed,thin, -]   (0.5,0.5) -- (1,0); \draw[black, dashed,thin, -]   (0.5,0.5) -- (0,1); \draw[black, dashed,thin, -]   (0.5,0.5) -- (0,0);
\draw[black, dashed,thin, -]   (0.25,-0.5) -- (0,0); \draw[black, dashed,thin, -]   (0.25,-0.5) -- (1,0); \draw[black, dashed,thin, -]   (0.25,-0.5) -- (1,-1); \draw[black, dashed,thin, -]   (0.25,-0.5) -- (-0.5,-1); \draw[black, dashed,thin, -]   (0.25,-0.5) -- (-0.5,-0.5);
\draw[fill=red] (0,0) circle (3pt);
\end{tikzpicture}
\end{center}
\end{minipage}
\begin{minipage}{0.32\textwidth}
\begin{center}
\begin{tikzpicture}[scale=1.2]
\draw[black, thick, -, fill=blue, opacity=0.2] (1,0) -- (0,1) -- (-0.5, 0.5) -- (-0.5, -0.5) -- (0.25, -0.5) -- (1,0);
\draw[black, thick, -] (1,0) -- (1,1) -- (0,1) -- (1,0);
\draw[black, thick, -] (0,1) -- (-0.5, 1.5) -- (-1,1) -- (-1,0) -- (-0.5, -0.5) -- (-0.5, 0.5) -- (0,1);
\draw[black, thick, -] (1,0) -- (1,-1) -- (-0.5, -1) -- (-0.5, -0.5) -- (0.25,-0.5) -- (1,0);
\end{tikzpicture}
\end{center}
\end{minipage}
\caption{Left: a polygonal mesh $\taun$ and its subtriangulation $\tautilden$, in a red circle a vertex $\Vbold$. Right: in light blue, $\omega_{\Vbold}$, the triangular patch around vertex $\Vbold$, defined in \eqref{triangular patch around V}.}
\label{figure triangular patch around V}
\end{figure}

\begin{figure}[h]
\centering
\begin{minipage}{0.32\textwidth}
\begin{center}
\begin{tikzpicture}[scale=1.2]
\draw[black, thick, -,fill=red, opacity=0.2] (0,0) -- (1,0) -- (1,1) -- (0,1) -- (0,0);
\draw[black, thick, -] (0,0) -- (-0.5, -0.5) -- (-1,0) -- (-1,1) -- (-0.5, 1.5) -- (0,1) -- (0,0);
\draw[black, thick, -] (0,0) -- (1,0) -- (1,-1) -- (-0.5, -1) -- (-0.5, -0.5) --  (0,0);
\draw[black, thick, -] (1,0)--(1,1)--(0,1);
\end{tikzpicture}
\end{center}
\end{minipage}
\begin{minipage}{0.32\textwidth}
\begin{center}
\begin{tikzpicture}[scale=1.2]
\draw[black, thick, -] (0,0) -- (1,0) -- (1,1) -- (0,1) -- (0,0);
\draw[black, thick, -] (0,0) -- (-0.5, -0.5) -- (-1,0) -- (-1,1) -- (-0.5, 1.5) -- (0,1) -- (0,0);
\draw[black, thick, -] (0,0) -- (1,0) -- (1,-1) -- (-0.5, -1) -- (-0.5, -0.5) --  (0,0);
\draw[black, dashed,thin, -]   (-0.5,-0.5) -- (-0.5,0.5); \draw[black, dashed,thin, -]   (-1,0) -- (-0.5,0.5); \draw[black, dashed,thin, -]   (-1,1) -- (-0.5,0.5);
\draw[black, dashed,thin, -]   (-0.5,1.5) -- (-0.5,0.5); \draw[black, dashed,thin, -]   (0,1) -- (-0.5,0.5); \draw[black, dashed,thin, -]   (0,0) -- (-0.5,0.5);
\draw[black, dashed,thin, -]   (0.5,0.5) -- (1,1); \draw[black, dashed,thin, -]   (0.5,0.5) -- (1,0); \draw[black, dashed,thin, -]   (0.5,0.5) -- (0,1); \draw[black, dashed,thin, -]   (0.5,0.5) -- (0,0);
\draw[black, dashed,thin, -]   (0.25,-0.5) -- (0,0); \draw[black, dashed,thin, -]   (0.25,-0.5) -- (1,0); \draw[black, dashed,thin, -]   (0.25,-0.5) -- (1,-1); \draw[black, dashed,thin, -]   (0.25,-0.5) -- (-0.5,-1); \draw[black, dashed,thin, -]   (0.25,-0.5) -- (-0.5,-0.5);
\end{tikzpicture}
\end{center}
\end{minipage}
\begin{minipage}{0.32\textwidth}
\begin{center}
\begin{tikzpicture}[scale=1.2]
\draw[black, thick, -, fill=blue, opacity=0.2] (1, - 1) -- (1,1) -- (0,1) -- (-0.5, 1.5) -- (-0.5,-0.5) -- (0.25, -0.5) -- (1, -1);
\draw[black, thick, -] (-0.5, -0.5) -- (-0.5,1.5) -- (-1,1) -- (-1, 0) -- (-0.5, -0.5);
\draw[black, thick, -] (-0.5, -0.5) -- (-0.5, -1) --  (1,-1) -- (0.25, -0.5) -- (-0.5, -0.5);
\draw[black, thick, -] (1,-1)--(1,1)--(0,1)--(-0.5,1.5)--(-0.5,-0.5)--(0.25,-0.5);
\end{tikzpicture}
\end{center}
\end{minipage}
\caption{Left: a polygonal mesh $\taun$, in light red a polygon $\E$. Center: the regular subtriangulation $\tautilden$.
Right: in light blue, $\omega_{\E}$, the triangular patch around $\E$, defined in \eqref{triangular patch around D}.} \label{figure triangular patch around K}
\end{figure}
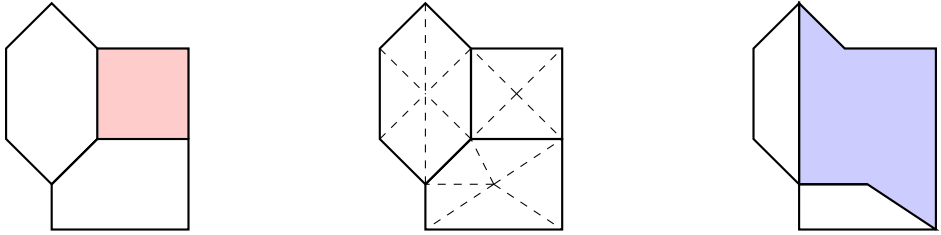

\begin{figure}[h]
\centering
\begin{minipage}{0.32\textwidth}
\begin{center}
\begin{tikzpicture}[scale=1.2]
\draw[black, thick, -] (0,0) -- (1,0) -- (1,1) -- (0,1) -- (0,0);
\draw[black, thick, -] (0,0) -- (-0.5, -0.5) -- (-1,0) -- (-1,1) -- (-0.5, 1.5) -- (0,1) -- (0,0);
\draw[black, thick, -] (0,0) -- (1,0) -- (1,-1) -- (-0.5, -1) -- (-0.5, -0.5) --  (0,0);
\end{tikzpicture}
\end{center}
\end{minipage}
\begin{minipage}{0.32\textwidth}
\begin{center}
\begin{tikzpicture}[scale=1.2]
\draw[black, thick, -] (0,0) -- (1,0) -- (1,1) -- (0,1) -- (0,0);
\draw[black, thick, -] (0,0) -- (-0.5, -0.5) -- (-1,0) -- (-1,1) -- (-0.5, 1.5) -- (0,1) -- (0,0);
\draw[black, thick, -] (0,0) -- (1,0) -- (1,-1) -- (-0.5, -1) -- (-0.5, -0.5) --  (0,0);
\draw[black, dashed,thin, -]   (-0.5,-0.5) -- (-0.5,0.5); \draw[black, dashed,thin, -]   (-1,0) -- (-0.5,0.5); \draw[black, dashed,thin, -]   (-1,1) -- (-0.5,0.5);
\draw[black, dashed,thin, -]   (-0.5,1.5) -- (-0.5,0.5); \draw[black, dashed,thin, -]   (0,1) -- (-0.5,0.5); \draw[black, dashed,thin, -]   (0,0) -- (-0.5,0.5);
\draw[black, dashed,thin, -]   (0.5,0.5) -- (1,1); \draw[black, dashed,thin, -]   (0.5,0.5) -- (1,0); \draw[black, dashed,thin, -]   (0.5,0.5) -- (0,1); \draw[black, dashed,thin, -]   (0.5,0.5) -- (0,0);
\draw[black, dashed,thin, -]   (0.25,-0.5) -- (0,0); \draw[black, dashed,thin, -]   (0.25,-0.5) -- (1,0); \draw[black, dashed,thin, -]   (0.25,-0.5) -- (1,-1); \draw[black, dashed,thin, -]   (0.25,-0.5) -- (-0.5,-1); \draw[black, dashed,thin, -]   (0.25,-0.5) -- (-0.5,-0.5);
\draw[black, dashed,thin, -, fill=red, opacity = 0.2]  (0,0) -- (-0.5,0.5) -- (0,1) -- (0,0);
\end{tikzpicture}
\end{center}
\end{minipage}
\begin{minipage}{0.32\textwidth}
\begin{center}
\begin{tikzpicture}[scale=1.2]
\draw[black, thick, -] (1,1) -- (1,0) -- (0.5, 0.5) -- (1,1);
\draw[black, thick, -] (0.25,-0.5) -- (1,0) -- (1,-1) -- (-0.5, -1) -- (-0.5, -0.5) --  (0.25,-0.5);
\draw[black,thin, -, fill=blue, opacity=0.2]  (1,0) -- (0.5,0.5) -- (1,1) -- (0,1) -- (-0.5,1.5) -- (-1,1) -- (-1,0) -- (-0.5,-0.5) -- (0.25, -0.5) -- (1,0);
\draw[black, thick, -] (1,1)--(0,1)--(-0.5,1.5)--(-1,1)--(-1,0)--(-0.5,-0.5);
\end{tikzpicture}
\end{center}
\end{minipage}
\caption{Left: a polygonal mesh $\taun$. Center: the regular subtriangulation $\tautilden$,
in light red a triangle $\T$. Right: in light blue, $\omega_{\T}$, the triangular patch around $\T$, defined in \eqref{triangular patch around D}.} \label{figure triangular patch around T}
\end{figure}
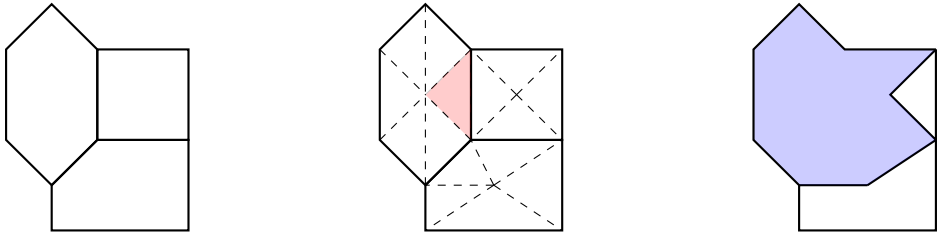

\begin{figure}[h]
\centering
\begin{minipage}{0.32\textwidth}
\begin{center}
\begin{tikzpicture}[scale=1.2]
\draw[black, thick, -] (0,0) -- (1,0) -- (1,1) -- (0,1) -- (0,0);
\draw[black, thick, -] (0,0) -- (-0.5, -0.5) -- (-1,0) -- (-1,1) -- (-0.5, 1.5) -- (0,1) -- (0,0);
\draw[black, thick, -] (0,0) -- (1,0) -- (1,-1) -- (-0.5, -1) -- (-0.5, -0.5) --  (0,0);
\draw[red, line width=1mm] (0,0) -- (-0.5, -0.5);
\end{tikzpicture}
\end{center}
\end{minipage}
\begin{minipage}{0.32\textwidth}
\begin{center}
\begin{tikzpicture}[scale=1.2]
\draw[black, thick, -] (0,0) -- (1,0) -- (1,1) -- (0,1) -- (0,0);
\draw[black, thick, -] (0,0) -- (-0.5, -0.5) -- (-1,0) -- (-1,1) -- (-0.5, 1.5) -- (0,1) -- (0,0);
\draw[black, thick, -] (0,0) -- (1,0) -- (1,-1) -- (-0.5, -1) -- (-0.5, -0.5) --  (0,0);
\draw[black, dashed,thin, -]   (-0.5,-0.5) -- (-0.5,0.5); \draw[black, dashed,thin, -]   (-1,0) -- (-0.5,0.5); \draw[black, dashed,thin, -]   (-1,1) -- (-0.5,0.5);
\draw[black, dashed,thin, -]   (-0.5,1.5) -- (-0.5,0.5); \draw[black, dashed,thin, -]   (0,1) -- (-0.5,0.5); \draw[black, dashed,thin, -]   (0,0) -- (-0.5,0.5);
\draw[black, dashed,thin, -]   (0.5,0.5) -- (1,1); \draw[black, dashed,thin, -]   (0.5,0.5) -- (1,0); \draw[black, dashed,thin, -]   (0.5,0.5) -- (0,1); \draw[black, dashed,thin, -]   (0.5,0.5) -- (0,0);
\draw[black, dashed,thin, -]   (0.25,-0.5) -- (0,0); \draw[black, dashed,thin, -]   (0.25,-0.5) -- (1,0); \draw[black, dashed,thin, -]   (0.25,-0.5) -- (1,-1); \draw[black, dashed,thin, -]   (0.25,-0.5) -- (-0.5,-1); \draw[black, dashed,thin, -]   (0.25,-0.5) -- (-0.5,-0.5);
\end{tikzpicture}
\end{center}
\end{minipage}
\begin{minipage}{0.32\textwidth}
\begin{center}
\begin{tikzpicture}[scale=1.2]
\draw[black, thick, -] (1,0) -- (1,1) -- (0,1) -- (1,0);
\draw[black, thick, -] (0,1) -- (-0.5, 1.5) -- (-1,1) -- (-1,0) -- (0,1);
\draw[black, thick, -] (-0.5,-1) -- (1,-1) -- (1,0) -- (-0.5, -1);
\draw[black, thick, -, fill=blue, opacity=0.2] (1,0) -- (0,1) -- (-1,0) -- (-0.5,-0.5) -- (-0.5,-1) -- (1,0);
\draw[black, thick, -] (-1,0)--(-0.5,-0.5) -- (-0.5,-1);
\end{tikzpicture}
\end{center}
\end{minipage}
\caption{Left: a polygonal mesh $\taun$, in red an edge $\e$. Center: the regular subtriangulation $\tautilden$.
Right: in light blue, $\omega_{\e}$, the triangular patch around edge $\e$, defined in \eqref{triangular patch around e}.} \label{figure triangular patch around e}
\end{figure}
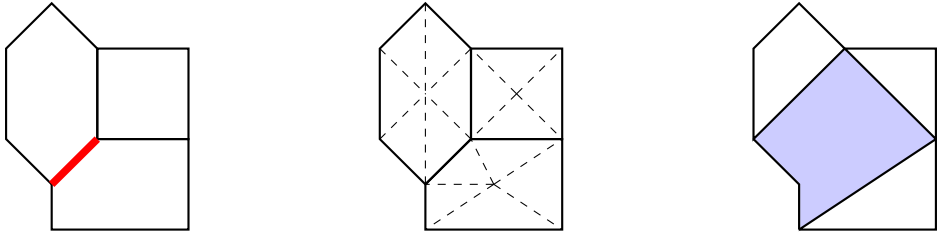

It can be proven that, using assumptions (\textbf{D1})-(\textbf{D2}),
$\text{diam}(\omega_D) \approx \text{diam}(D)$, $D$ being either a polygon or a triangle and $\omega(D)$ being the associated triangular patch.

We now recall the following polynomial \Clement-type approximation result over triangles.
\begin{lem} \label{lemma Clement Melenk}
Let the assumptions (\textbf{D1})-(\textbf{D2}) be valid.
Then, given $\u \in H^1_0(\Omega)$, there exists $\uM \in S^{\ptildebold, 0} (\Omega, \tautilden) \cap H^1_0(\Omega)$, see~\eqref{space piecewise continuous polynomials}, such that, for each triangle $\T \in \tautilden$ and for each $\e$ edge of $\widetilde{\mathcal E}_n$
(i.e. of the skeleton of $\tautilden$), the following estimates hold true:
\begin{subequations}
\begin{align}
&\Vert \u - \uM \Vert_{0,\T} + \frac{\hT}{\pT} \vert \u - \uM \vert_{1,\T} \le c_M \frac{\hT}{\pT} \vert \u \vert_{1,\omega_{\T}}, \label{Clement Melenk estimates a}\\
&\Vert \u - \uM \Vert_{0,\e} \le c_M \left( \frac{\hT}{\pT} \right)^{\frac{1}{2}} \vert \u \vert_{1,\omega_{\e}}, \label{Clement Melenk estimates b}
\end{align}
\end{subequations}
where $c_M$ is a positive constant independent on $\u$, $\hT$ and $\pT$ but depending on the shape-regularity of $\taun$ and hence of $\tautilden$.
Above, the symbols $\omega_{\T}$ and $\omega_{\e}$ represent the triangular patch around triangle $\T$
and the triangular patch around edge $\e$ defined in \eqref{triangular patch around D} and \eqref{triangular patch around e}, respectively.
\end{lem}
\begin{proof}
The proof, based on the partition of unity method \cite{BabuskaMelenk_PUMintro}, can be found in~\cite[Theorem 3.3]{melenk2003hp}.
\end{proof}
Lemma \ref{lemma Clement Melenk} is a key ingredient for the following result which asserts the existence of an $\h\p$ VEM \Clement\, quasi-interpolant.

\begin{prop} \label{proposition VEM Clement}
Let the assumptions (\textbf{D1})-(\textbf{D2}) be valid.
Given $\u$ the solution to \eqref{hp VEM}, a convex polygon $\E \in \taun$ and $\e$ any of its edges,
there exists a function $\uI \in \Vn$, such that its restriction to $\E$ satisfies the following estimates:
\begin{subequations}
\begin{align}
&\Vert \u - \uI \Vert_{0,\E} + \frac{\hE}{\pE} \vert \u - \uI \vert_{1,\E} \le c \frac{\hE}{\pE} \vert \u \vert_{1,\omega_{\E}}, \label{estimates virtual Clement a}\\
&\Vert \u - \uI \Vert_{0,\e} \le c \left( \frac{\hE}{\pE} \right)^{\frac{1}{2}} \vert \u \vert_{1,\omega_{\e}}, \label{estimates virtual Clement b}
\end{align}
\end{subequations}
where $c$ is a positive constant independent on $\u$, $\hE$ and $\pE$ but depending on the shape-regularity of $\taun$ and hence of $\tautilden$, and where $\omega_{\E}$ and $\omega_{\e}$ follow the notation introduced in \eqref{triangular patch around D} and \eqref{triangular patch around e}.
\end{prop}
\begin{proof}
Given $\E \in \taun$ convex, we define $\uI \in \VnE$ as the solution to the following Poisson problem with nonhomogeneous Dirichlet boundary conditions
\begin{equation} \label{definition virtual interpolant}
\begin{cases}
\text{find }\uI \in \VnE \text{ s. th.}\\
\uI = \uM \text{ on } \partial \E \\
\aE(\uI, w) = \aE(\uM, w) \quad \forall  w \in \VnE \cap H^1_0(\E),\\
\end{cases}
\end{equation}
where $\uM$ is the quasi-interpolant from~\cite{melenk2003hp} introduced in Lemma~\ref{lemma Clement Melenk}.

We observe that the construction of the quasi-interpolant~$\uM$ is based on the idea of subtriangulation of a polygon,
which in turns can be traced back to the interpolation theory of generalized barycentric coordinates on polygons.
In practice, we fix $\uI$ on the boundary to be the polynomial quasi-interpolant guaranteeing $hp$ approximation properties on triangles
and we fix the internal moments of $\uI$ by demanding
\begin{equation} \label{fixing the internal dofs}
\aE(\uI - \uM , w ) = 0 \quad \Longrightarrow \quad \int_\E \uI \, \Delta w = \int_{\partial \E} \uI \, \partial_\n w - \int_\E \nabla \uM \cdot \nabla w.
\end{equation}
Roughly speaking, this last condition enforces the (distributional) identity $\Delta \uI  = \Delta \uM$ in an approximated sense.

\medskip

The estimate \eqref{estimates virtual Clement b} is trivially guaranteed by \eqref{Clement Melenk estimates b}.

\medskip

Next, we investigate the bound on the $H^1$ seminorm in \eqref{estimates virtual Clement a}. Given any bubble function $w$ in $\VnE \cap H^1_0(\E)$, we observe that, owing to \eqref{definition virtual interpolant},
\[
\begin{split}
\vert \uM - \uI \vert^2_{1,\E}	& = (\nabla (\uM- \uI), \nabla (\uM- \uI))_{0,\E} \\
					&  = (\nabla (\uM- \uI) - \nabla w, \nabla (\uM- \uI))_{0,\E}\le \vert \uM -(\uI+w)\vert_{1,\E} \vert \uM-\uI \vert_{1,\E}. \\
\end{split}
\]
Therefore, since any function $\wtilde$ in the virtual element space with $\wtilde _{|\partial \E} = \u_{M|\partial \E} = \u_{I|\partial \E}$ can be written as $\uI+w$, one easily gets
\[
\vert \uM - \uI \vert_{1,\E} = \inf_{\wtilde \in \VnE,\, \wtilde_{|_{\partial \E}}= \uM} \vert \uM - \wtilde \vert_{1,\E}.
\]
As a consequence,
\begin{equation} \label{first splitting H1}
\vert  \u - \uI \vert_{1,\E} \le \vert \u - \uM \vert_{1,\E} + \vert \uM - \uI \vert_{1,\E} \le \vert \u - \uM \vert_{1,\E} + \vert \uM -\utildeI \vert_{1,\E},
\end{equation}
where we choose $\utildeI \in \VnE$ satisfying the following local problem:
\[
\begin{cases}
\Delta \utildeI = \Delta \upi & \text{in } \E\\
\utildeI = \uM & \text{on } \partial \E,\\
\end{cases}
\]
being $\upi$ equal to $\Pinabla u$, see \eqref{H1 projector}.

The Dirichlet principle, see e.g.~\cite{evansPDE}, yields
\begin{equation} \label{second splitting H1}
\vert \utildeI - \upi \vert_{1,\E} \le \vert \upi - \uM \vert_{1,\E} \le \vert \u - \upi \vert_{1,\E} + \vert \u - \uM \vert_{1,\E}.
\end{equation}
Plugging \eqref{second splitting H1} in \eqref{first splitting H1}, we get
\begin{equation} \label{some triangle}
\begin{split}
\vert \u - \uI \vert_{1,\E} 	& \le \vert \u - \uM \vert_{1,\E} + \vert \upi - \uM \vert_{1,\E} + \vert \upi - \utildeI \vert_{1,\E} \le 3 \vert \u - \uM \vert_{1,\E} + 2 \vert \u - \upi \vert_{1,\E},
\end{split}
\end{equation}
whence, applying \eqref{Clement Melenk estimates a} and \cite[Lemma 4.2]{hpVEMbasic} to the first and second terms on the right-hand side of \eqref{some triangle}, respectively,
\[
\vert \u - \uI \vert_{1,\E} \le c \vert \u \vert_{1,\omega_\E}.
\]
We underline that the constant c depends only on the shape-regularity of $\taun$ and hence of $\tautilden$.
\medskip

Finally, we investigate the bound on the $L^2$ norm in \eqref{estimates virtual Clement a}. To this purpose, we employ the hypothesis that $\E$ is convex and we consider the auxiliary problem
\begin{equation} \label{auxiliary problem}
\begin{cases}
\text{find } \psi \in H^1_0(\E) \text{ s. th.}\\
\aE(\psi, w) = (\uM - \uI, w)_{0,\E} \quad \forall w \in H^1_0(\E).\\
\end{cases}
\end{equation}
The following standard a priori bound on the solution to the auxiliary problem \eqref{auxiliary problem} for convex $\E$ is valid:
\begin{equation} \label{a priori bound}
\Vert \psi \Vert_{2,\E} \lesssim \Vert \uM - \uI \Vert_ {0,\E}.
\end{equation}
Owing to the fact that ${\uM-\uI} _{| _{\partial \E}} = 0$ and to the orthogonality of $\uM - \uI$ with respect to virtual bubble functions \eqref{fixing the internal dofs}, we have
\begin{equation} \label{L2 becomes H1}
\Vert \uM -\uI \Vert^2_{0,\E} = (\nabla \psi, \nabla(\uM-\uI))_{0,\E} = (\nabla(\psi - \psiII), \nabla(\uM - \uI))_{0,\E},
\end{equation}
where $\psiII$ is a function in $\VnE\cap H^1_0(\E)$ providing optimal $\h\p$ approximation estimates of $\psi$, see e.g. \cite[Lemma 4.3]{hpVEMbasic},
\begin{equation} \label{hp approx virtual interpolant}
\vert \psi - \psiII \vert_{1,\E} \le c \frac{\hE}{\pE} \Vert \psi \Vert_{2,\E},
\end{equation}
being $c$ a positive constant depending only on the shape-regularity of $\E$.

We deduce, plugging \eqref{hp approx virtual interpolant} into \eqref{L2 becomes H1},
\[
\Vert \uM - \uI \Vert^2_{0,\E} \le c \frac{\hE}{\pE} \Vert \psi \Vert_{2,\E} \vert \uM - \uI \vert_{1,\E},
\]
where $c$ is the same constant as in \eqref{hp approx virtual interpolant}.
%

Recalling \eqref{a priori bound}, we obtain
\[
\Vert \uM - \uI \Vert_{0,\E} \lesssim \frac{\hE}{\pE} \vert \uM - \uI \vert_{1,\E}.
\]
The assertion follows from a triangle inequality and the bound on the $H^1$ seminorm in \eqref{estimates virtual Clement a}.
\end{proof}

\begin{remark} \label{remark non convex}
If element $\E$ is nonconvex, then the internal $L^2$ bound of Proposition \ref{proposition VEM Clement} modifies to
\[
\Vert \u - \uI \Vert_{0,\E} \le c \left( \frac{\hE}{\pE}  \right)^{\frac{\pi}{\alpha_\E} - \varepsilon} \vert \u \vert_{1, \omegaE},
\]
for all $\varepsilon > 0 $ arbitrarily small, where $\alpha_\E$ denotes the largest interior angle of $\E$.
\end{remark}

\subsection{Inverse estimates on triangles and polygons} \label{subsection inverse estimates on triangles and polygons}
In this section, we recall two $\h\p$ polynomial inverse inequalities that will be instrumental in the forthcoming analysis.

We begin with the following 1D result.
\begin{lem} \label{lemma 1D hp inverse}
Given $\widehat I = [-1,1]$, given $\psi$ the quadratic bubble function over $\widehat I$ and given $0\le \alpha_1 \le \alpha_2$, there exists a positive constant such that, for all $\q\in \mathbb P_\p(\widehat I)$, $\p \in \mathbb N$,
\[
\int_{\widehat I} \psi ^{\alpha_1} \q^2   \le c (\p+1) ^{2(\alpha_2-\alpha_1)} \int_{\widehat I} \psi^{\alpha_2} \q^2.
\]
\end{lem}
\begin{proof}
See \cite[Lemma 4]{bernardi2001error}.
\end{proof}
Given~$\E \in \taun$, we define a piecewise bubble function~$\psi_{\E}$ over~$\tautilden(\E)$ as follows:
\begin{equation}  \label{polygonal cubic bubble function}
\psi_{\E|_\T} = \psi_\T,\quad \text{where } \psi_\T \text{ is the cubic bubble function on triangle $\T$}\text{ for all } \T \in \tautilden(\E),
\end{equation}
where we recall that $\tautilden(\E)$ is the subtriangulation of $\E$ introduced in Remark \ref{remark triangular subdecomposition}.

The following $\h\p$ polynomial inverse inequality over a polygon holds true.
\begin{lem} \label{thorem inverse estimates on a polygon}
Given~$\E$ a polygon in~$\taun$, let~$\psi _\E$ be the piecewise bubble function associated with the subtriangulation~$\tautilden (\E)$ of~$\E$ defined in~\eqref{polygonal cubic bubble function}.
For all $-1 < \alpha \le \beta$, there exist a positive constants $c$ depending only on $\alpha$, $\beta$, and the shape-regularity of the subtriangulation $\tautilden(\E)$, such that, for all $\q\in \mathbb P_{\pE}(\E)$,
\begin{equation} \label{inverse estimates on polygon b}
\int_{\E} \psi_{\E}^{\alpha} \vert \q \vert^2 \le c_2 (\p_{\E}+1)^{2(\beta-\alpha)} \int_{\E} \psi_{\E}^\beta \vert \q \vert^2.
\end{equation}
\end{lem}
\begin{proof}
See \cite[Theorem B.3.2]{MascottoPhDthesis}.
\end{proof}

\subsection{A lifting operator} \label{subsection a lifting operator}
In this section, we recall the existence of a lifting operator from an edge of a triangle to its interior.
\begin{lem} \label{lemma lifting operator}
Given $\T \in \tautilden$ a triangle and $\e$ any of its edges and given $1/2 < \alpha \le 2$, let $\psi_{\e}$ be the quadratic bubble function associated with edge $\e$.
Then, for every $\q \in \mathbb P_{\p_{\e}}(\e)$ and for every $\varepsilon >0$, there exist a lifting $E(\q) \in H^1(\T)$, such that
\begin{equation} \label{estimates on extension operator a}
E(\q) _{|_{\e}} = \q \, \psi_{\e} ^\alpha,\quad \quad \quad E(\q) _{|_{\partial \T \setminus \e}} = 0,\\
\end{equation}
\begin{equation} \label{estimates on extension operator b}
\Vert E(\q) \Vert^2_{0,\T} \le c_\alpha \hT \varepsilon \Vert \q \psi_{\e}^{\frac{\alpha}{2}} \Vert_{0,\e}^2,\\
\end{equation}
\begin{equation} \label{estimates on extension operator c}
\vert E(\q) \vert^2_{1,\T} \le c_\alpha \hT^{-1} (\varepsilon \p_{\e}^{2(2-\alpha)} + \varepsilon ^{-1}) \Vert \q \psi_{\e}^{\frac{\alpha}{2}} \Vert_{0,\e}^2,\\
\end{equation}
where $c_\alpha$ is a positive constant depending only on $\alpha$ and on the shape of $\T$.
\end{lem}
\begin{proof}
The proof follows from \cite[Lemma 2.6]{MelenkWohlmuth_hpFEMaposteriori} and a scaling argument.
\end{proof}
Note that, since $E(\q)$ vanishes on $\partial \T \setminus \e$, then it can be extended to~$0$ on the remaining part of the polygon~$\E$ containing~$\T$ as part of its subtriangulation whenever~$\e$ lies on the boundary of~$\E$.
As a consequence, \eqref{estimates on extension operator b} and~\eqref{estimates on extension operator c} can be ``generalized'' in a straightforward way,
employing on the left-hand side (semi)norms on the complete polygon \emph{in lieu} of their counterparts on triangle $\T$.

\section{A posteriori error analysis} \label{section a posteriori error analysis}
In this section, we build an error estimator and we prove that it can be upper and lower bounded by the $H^1$ error of the method, with constants explicit in terms of the discretization parameters.

The remainder of the the section is structured as follows.
In Section \ref{subsection the residual equation}, we write the residual equation, whereas,
in Section \ref{subsection reliability}, we construct an error estimator and we show that the energy error can be bounded in terms of such an estimator
with an explicit dependence in terms of the discretization parameters.
Next, in Section \ref{subsection efficiency}, we show that the \emph{local} estimator can be bounded by the $H^1$ seminorm of the \emph{quasi-local} error
plus a couple of terms involving the oscillation of the right-hand side of the problem \eqref{strong Poisson problem} and the stabilization of the method.
Finally, in Section \ref{subsection conclusions}, we summarize the foregoing results.

Henceforth, we assume for the sake of clarity that all the polygons in $\taun$ are convex. The nonconvex case is discussed in Remark \ref{remark theorem non convex}.
Furthermore, we also assume in the forthcoming analysis that $\pE \ge 2$ for all $\E\in \taun$; this allows us to avoid a cumbersome notation when treating the discrete right-hand side, cf. Section \ref{section discrete right-hand side}.
Thanks to this assumption, we will write $\langle \fn, \vn \rangle_n$ as $(\fn, \vn)_{0,\Omega}$, where $\f_{n|_{\E}} = \Pizpmd \f$ for all $\E\in \taun$.

Finally, we set the jump across an edge $\e\in \mathcal E_n$ as
\begin{equation} \label{jump}
\llbracket \vv \rrbracket_\e =
\begin{cases}
\vv & \text{if }\e \in \mathcal E_n^b\\
\vv^+ - \vv^- & \text{otherwise},\\
\end{cases}
\end{equation}
where $\vv^+$ and $\vv^-$ are the restriction (assumed sufficiently regular) of $\vv$ over $\partial \E^+$ and $\partial \E^-$, respectively, and where $\e \subseteq \partial \E^+ \cap \partial \E^-$.
Note that we are assuming that the unit normal $\n$ is pointing outside $\E^+$ and inside $\E^-$.

\subsection{The residual equation} \label{subsection the residual equation}
We begin by introducing the residual equation. Let $e := \u - \un$ be the difference between the solution to the continuous \eqref{Poisson problem homogeneous Dirichlet weak formulation} and discrete \eqref{hp VEM} problems, respectively.
One has, for any $\vv \in H^1_0(\Omega)$ and $\chin \in \Vn$,
\begin{equation} \label{residual equation}
\begin{split}
\a(e,\vv)	& = (\f,\vv)_{0,\Omega} - a(\un, \vv) = (\f,\vv)_{0,\Omega} - \a(\un, \chin) - \a(\un,\vv-\chin)\\
		& = (\f,\vv) _{0,\Omega} - (\fn,\chin)_{0,\Omega} + \an (\un,\chin) - \a(\un,\chin) - \a(\un,\vv-\chin)\\
		& = (\f-\fn, \chin)_{0,\Omega} + (\f, \vv-\chin)_{0,\Omega} + \an(\un,\chin) - \a(\un,\chin) - \a(\un,\vv-\chin).\\
\end{split}
\end{equation}
We rewrite the last term on the right-hand side of \eqref{residual equation} by observing that, for all $\w \in H^1_0(\Omega)$, the following holds true:
\begin{equation} \label{first integration by parts}
\begin{split}
\a(\un,\w)	& = \sum_{\E \in \taun} \aE(\un, \w) = \sum_{\E \in \taun} \left\{ \aE(\PinablapuE \un, \w) + \aE((I - \PinablapuE )\un, \w)  \right\}\\
		& = \sum_{\E\in \taun} \left\{ -(\Delta \PinablapuE \un, \w)_{0,\E} + \aE((I-\PinablapuE) \un, \w)  \right\} + \sum_{\e \in \EnI} \left( \lllbracket \partial _\n \PinablapuE \un \rrrbracket_\e,\w \right)_{0,\e},
\end{split}
\end{equation}
where we recall that the jump $\llbracket \cdot \rrbracket_\e$ is defined in \eqref{jump} and where we recall that we are assuming that the unit normal $\n$ of each edge is fixed once and for all.
Note that~$\EnI$ represents the set of edges of~$\En$ not lying on~$\partial \Omega$.

We highlight the presence of a polynomial projector in \eqref{first integration by parts}; without such a projector we would not be able to compute exactly the jump across the edges
of normal derivatives of functions in the virtual element space, which we anticipate will be part of the error residual of the method.

Plugging \eqref{first integration by parts} in \eqref{residual equation} with $\w = \vv - \chin$, we deduce
\[
\begin{split}
a(e,\vv) 	& = \sum_{\E \in \taun}  \left\{  {(\Delta \PinablapuE \un + \fn, \vv -\chin )_{0,\E}} - \aE((I-\PinablapuE) \un, \vv- \chin ) + {(\f-\fn, \vv- \chin)_{0,\E}} \right\}\\
		& \quad\quad\quad+ {(\f- \fn, \chin)_{0,\Omega}} + {\an(\un, \chin) - \a(\un, \chin)} - \sum_{\e \in \EnI} {\left( \lllbracket \partial _\n \PinablapuE \un \rrrbracket, \vv- \chin  \right)_{0,\e}},\\
\end{split}
\]
whence
\begin{equation} \label{equazione da cui tutto parte}
\begin{split}
a(e,\vv) 	& = \sum _{\E \in \taun}  \left\{  {(\Delta \PinablapuE \un + \fn, \vv -\chin )_{0,\E}} + {(\f-\fn, \vv)_{0,\E}} - { \aE((I-\PinablapuE) \un, \vv- \chin )} \right.\\
		& \quad \quad\quad \quad\left.+ \anE (\un, \chin) - \aE(\un, \chin) \right\} - \sum_{\e \in \EnI} {\left( \lllbracket \partial _\n \PinablapuE \un  \rrrbracket_\e, \vv- \chin  \right)_{0,\e}}.
\end{split}
\end{equation}
The first two and the last term on the right-hand side of \eqref{equazione da cui tutto parte} are analogous to the terms appearing in the FEM counterpart of the residual equation, see \cite[Lemma 3.1]{MelenkWohlmuth_hpFEMaposteriori}.
The only difference here, is that we consider the $H^1$ projection of the discrete solution.
Note that in the FEM framework such a projection applied to the solution coincides with the solution itself (since the solution is a piecewise polynomial).

The remaining terms (virtual element consistency terms) on the right-hand side of \eqref{equazione da cui tutto parte} are instead typical of the virtual element setting
and they take into account the fact that the discrete bilinear form is only an approximation to the exact one.

\subsection{Error estimator and upper bounds} \label{subsection reliability}
In this section, we introduce a computable error estimator and we prove the lower and upper bounds of the energy error in terms of such estimator.

To this purpose, we use the residual equation \eqref{equazione da cui tutto parte} with $\vv=e:= \u-\un$ and $\chin = \errI = (\u - \un)_I$,
where we recall that the local $\h\p$ approximation properties of $\errI$ are described in Proposition \ref{proposition VEM Clement}.
We obtain
\begin{equation} \label{4}
\begin{split}
\vert \err \vert ^2_{1,\Omega} 	& = \sum_{\E \in \taun} \left\{ (\Delta \PinablapuE \un +\fn, \err -\errI)_{0,\E}+ (\f-\fn, \err)_{0,\E} \right .\\
					& \quad \quad \quad \quad \left .  (\anE(\un, \errI) - \aE(\un, \errI)) -\aE( (I-\PinablapuE)\un, \err- \errI) \right\}\\
					& \quad - \sum_{\e\in \EnI} \left( \lllbracket \partial _\n \PinablapuE \un \rrrbracket_\e, \err- \errI   \right)_{0,\e} \\
					& =: \sum_{\E \in \taun} \left\{  I + II + III + IV \right\} + \sum_{\e\in \EnI} V.
\end{split}
\end{equation}
We estimate the five local terms separately. We start with the term $I$. Applying Proposition \ref{proposition VEM Clement}, which is valid since we are assuming $\E$ convex, we get
\begin{equation} \label{bounding term I}
\begin{split}
I:= (\Delta \PinablapuE \un + \fn, \err - \errI)_{0,\E} 	& \le \Vert \Delta \PinablapuE \un + \fn \Vert_{0,\E} \Vert \err - \errI \Vert_{0,\E} \\
									& \lesssim \frac{\hE}{\pE} \Vert \Delta \PinablapuE \un + \fn \Vert_{0,\E}  \vert \err \vert_{1,\omegaE},\\
\end{split}
\end{equation}
where we recall that the hidden constant depends solely on the shape-regularity of $\taun$ and hence of $\tautilden$, see Remark \ref{remark triangular subdecomposition}.
In the following, when no confusion occurs and when unnecessary, we will omit to highlight such dependence.

Secondly, we investigate $II$, the term involving the oscillation of the right-hand side. Owing to $L^2$ orthogonality and $\h\p$ approximation properties of such projector, see e.g. \cite[Lemma 4.2]{hpVEMbasic}, we write
\[
II := (\f-\fn, \err)_{0,\E} = (\f-\fn, \err - \PizpmduE \err)_{0,\E} \lesssim \Vert \f- \fn \Vert_{0,\E} \frac{\hE}{\pE} \vert \err \vert_{1,\E},
\]
where we recall that the $L^2$ projector $\PizpmduE$ is defined in \eqref{L2 projection}.

Next, we deal with $V$, the projected jump edge residual term: for all~$\e\in\EnI$,
\[
\begin{split}
V &:= -\left( \left\llbracket  \partial _\n \PinablapuE \un  \right\rrbracket_\e , \err - \errI  \right)_{0,\e} \le \left \Vert \left\llbracket  \partial _\n \PinablapuE \un  \right\rrbracket_\e   \right \Vert_{0,\e} \Vert \err -\errI\Vert_{0,\e}\\
   & \lesssim \left( \frac{\he}{\pe} \right) ^{\frac{1}{2}} \left \Vert \left\llbracket  \partial _\n \PinablapuE \un  \right\rrbracket_\e \right\Vert _{0,\e} \vert \err \vert_{1,\omegae},
\end{split}
\]
where in the last but one inequality we employed \eqref{estimates virtual Clement b}.

The first virtual element consistency term $IV$ can be bounded instead using \eqref{estimates virtual Clement a} (in the last but one inequality) and the stability property \eqref{stability} (in the last inequality):
\[
\begin{split}
IV 	&:= -\aE((I-\PinablapuE)\un, \err -\errI) \le \vert (I-\PinablapuE) \un \vert_{1,\E}  \vert \err - \errI \vert _{1,\E} \\
	&\lesssim \vert (I - \PinablapuE) \un \vert_{1,\E} \vert \err \vert_{1,\omegaE} \lesssim \alpha_*^{-\frac{1}{2}}(\pE) \SE((I-\PinablapuE) \un, (I-\PinablapuE) \un)^{\frac{1}{2}} \vert \err \vert_{1,\omegaE}.
\end{split}
\]
We emphasize that we make appear the stabilization term since in the definition of the error residual we want to have computable quantities only.

Finally, we study $III$, the second virtual element consistency term:
\[
\begin{split}
III	& := \anE(\un, \errI) - \aE (\un, \errI) \\
	& = \aE (\PinablapuE \un, \PinablapuE \errI) + \SE((I-\PinablapuE) \un, (I-\PinablapuE) \errI) - \aE(\un, \errI)\\
	& = \aE(\PinablapuE \un - \un, \errI) + \SE ((I-\PinablapuE) \un, (I-\PinablapuE) \errI) \\
	& \le \vert \un - \PinablapuE \un \vert_{1,\E} \vert \errI \vert_{1,\E} + \SE((I-\PinablapuE) \un, (I-\PinablapuE) \errI),\\
\end{split}
\]
whence, recalling Proposition \ref{proposition VEM Clement} and the stability property \eqref{stability},
\[
\begin{split}
III	& \lesssim \alpha_*^{-\frac{1}{2}}(\pE) \SE ((I-\PinablapuE) \un, (I-\PinablapuE) \un)^{\frac{1}{2}} \vert \err \vert_{1,\omegaE} \\
	&\quad+ \SE((I-\PinablapuE) \un, (I-\PinablapuE) \un)^{\frac{1}{2}} \alpha^*(\pE)^ {\frac{1}{2}} \vert (I-\PinablapuE) \errI \vert_{1,\E}\\
	& \lesssim \max(\alpha^{-1}_*(\pE), \alpha^*(\pE))^{\frac{1}{2}}  \SE((I-\PinablapuE) \un, (I-\PinablapuE) \un)^{\frac{1}{2}} \vert \err \vert_{1,\omegaE}.
\end{split}
\]
In order to simplify the notation, we write henceforth
\[
\RE := (\Delta \PinablapuE \un + \fn)_{|_\E} \quad \forall\, \E \in \taun, \quad \quad \quad \Re := \left\llbracket  \partial _\n \PinablapuE \un \right\rrbracket_\e \quad \forall \,\e \in \mathcal E^\E,\, \e\not\subset \partial \Omega.
\]
Collecting the estimates on the five terms on the right-hand side of \eqref{4}, we get, after some trivial algebra,
\[
\begin{split}
\vert \err \vert^2_{1,\Omega} 	& \lesssim \sum_{\E \in \taun} \left\{ \left( \frac{\hE}{\pE} \right)^2 \Vert \RE \Vert ^2_{0,\E} + \frac{\hE^2}{\pE^2} \Vert \f - \fn \Vert_{0,\E}^2 \right.\\
						& \quad\quad\quad\quad \left.+ \max(\alpha^{-1}_*(\pE), \alpha^*(\pE)) \SE((I-\PinablapuE)\un, (I-\PinablapuE)\un)   \right\} \\
						& \quad + \sum_{\e \in \EnI} \frac{\he}{\pe} \Vert \Re \Vert_{0,\e}^2,\\
\end{split}
\]
whence
\begin{equation} \label{error estimator}
\begin{split}
\vert \err \vert^2_{1,\Omega} 	& \lesssim \sum_{\E \in \taun} \left\{ \eta_\E^2 +  \rhoE^2 + \max(\alpha^{-1}_*(\pE), \alpha^*(\pE)) \zetaE^2 \right\} + \sum_{\e \in \EnI} \eta_\e^2,\\
\end{split}
\end{equation}
where we have set the local error estimators as
\begin{equation} \label{local error estimators}
\begin{split}
&\eta_\E = \frac{\hE}{\pE} \Vert \RE \Vert_{0,\E},\quad \eta_\e = \left( \frac{\he}{\pe}  \right)^{\frac{1}{2}} \Vert \Re \Vert_{0,\e},\\
&\zetaE^2= \SE((I-\PinablapuE)\un, (I-\PinablapuE)\un), \quad \rhoE = \frac{\hE}{\pE} \Vert\f-\fn \Vert_{0,\E}.\\
\end{split}
\end{equation}
For ease of notation, we define next $\eta_{\pE}$, which takes into account both the bulk and the edge error estimators $\eta_\E$ and $\eta_\e$, $\e \in \mathcal E^\E$, defined in \eqref{local error estimators}, as
\begin{equation} \label{complete local error estimator}
\eta_{\pE}^2 = \eta_\E^2 + \sum_{\e \in \mathcal E^\E,\, \e \not \subset \partial \Omega} \frac{1}{2} \eta_\e^2.
\end{equation}
We highlight two facts. The first one is that, apart from the term involving the stabilization, the upper bound is analogous to the one of $\h\p$ FEM, see \cite[Lemma 3.1]{MelenkWohlmuth_hpFEMaposteriori}.
The second observation is that for nonconvex $\E$, we would have an additional factor $\pE^{2\frac{\pi}{\alpha_\E} - \varepsilon}$ for all $\varepsilon >0$ arbitrarily small in front of $\eta_\E^2$
in \eqref{error estimator}, where $\alpha_\E$ denotes the largest interior angle of $\E$, see Remark \ref{remark non convex}.

\subsection{Lower bound} \label{subsection efficiency}
In this section, we bound the local error estimator $\eta_{\pE}$ introduced in \eqref{complete local error estimator}
with the quasi-local $H^1$ seminorm of the error of the method plus a term related to the oscillation of the right-hand side and a term related to the stabilization of the method.
After recalling a technical tool in Section \ref{subsubsection an inverse inequality}, we prove in Sections \ref{subsubsection bounding the internal residual} and \ref{subsection bounding the edge residual} the bounds on the internal and edge residuals, respectively, in terms of the local energy errors.
\subsubsection{An inverse inequality} \label{subsubsection an inverse inequality}
Here, we recall an $\h\p$ polynomial inverse inequality, which will be used in Section \ref{subsubsection bounding the internal residual}.
\begin{lem} \label{lemma technical inverse inequality}
Given a polygon~$\E \in \taun$, we set~$\psi_\E$ the piecewise bubble function associated with polygon~$\E \in \taun$ as in~\eqref{polygonal cubic bubble function}.
Then, for all  $\q \in \mathbb P_{\pE}(\E)$ and $1/2 < \alpha \le 2$, the following holds true:
\begin{equation} \label{inverse estimate Melenk}
\vert \psi_\E^\alpha \q \vert_{1,\E} \lesssim \frac{(\pE+1)^{2-\alpha}}{\hE} \Vert \psi_\E^{\frac{\alpha}{2}} \q \Vert_{0,\E},
\end{equation}
where the hidden constant depends on the shape-regularity of $\taun$ and hence of $\tautilden$.
\end{lem}
\begin{proof}
For a complete proof, see \cite[Lemma B.3.3]{MascottoPhDthesis}, which is the ``polygonal'' counterpart of~\cite[Lemma 3.4]{MelenkWohlmuth_hpFEMaposteriori}.
\end{proof}

\subsubsection{Bounding the internal residual} \label{subsubsection bounding the internal residual}
In this section, we bound the local internal residual $\RE$ appearing on the right-hand side of \eqref{complete local error estimator}.
To this purpose, we note that plugging $\chin = 0$ in \eqref{equazione da cui tutto parte}, we have
\begin{equation} \label{starting internal residual}
\a(\err, \vv) = \sum_{\E \in \taun} \left\{ (\RE, \vv)_{0,\E} + (\f- \fn, \vv)_{0,\E} - \aE((I-\PinablapuE) \un, \vv)   \right\} - \sum_{\e \in \EnI} (\Re, \vv)_{0,\e}.
\end{equation}
Let us focus our attention on a single element~$\E$. Given~$\bE$ the piecewise bubble function associated with element~$\E$ defined as in~\eqref{polygonal cubic bubble function},
we extend it to $0$ outside $\E$. We then choose $\vv = \bE^{\alpha} \RE$ in \eqref{starting internal residual}, with $1/2 < \alpha \le 2$, obtaining
\begin{equation} \label{6}
\begin{split}
\a(\err, \bE^{\alpha} \RE) 	& = \sum_{\E' \in \taun} \left\{ (\REp, \bE^{\alpha} \RE)_{0,\E'} + (\f-\fn, \bE^{\alpha} \RE)_{0,\E'} - \aEp ((I-\Pi^\nabla_{\E'}) \un, \bE^{\alpha} \RE)   \right\}\\
					& = (\RE, \bE^{\alpha} \RE)_{0,\E} + (\f-\fn, \bE^{\alpha} \RE)_{0,\E} - \aE ((I-\PinablapuE) \un, \bE^{\alpha} \RE).
\end{split}
\end{equation}
From \eqref{6}, we get
\[
\begin{split}
\Vert \bE^{\frac{\alpha}{2}} \RE \Vert^ 2_{0,\E} 	
	& = (\RE, \bE ^{\alpha} \RE)_{0,\E} = \aE(\err, \bE^\alpha \RE) - (\f-\fn, \bE^\alpha \RE)_{0,\E} + \aE((I-\PinablapuE) \un, \bE^\alpha \RE)\\
	& \lesssim \vert \bE ^\alpha \RE\vert_{1,\E} \vert \err \vert_{1,\E} + \Vert (\f - \fn) \bE^{\frac{\alpha}{2}}\Vert_{0,\E} \Vert \bE ^{\frac{\alpha}{2}} \RE \Vert_{0,\E} + \vert (I-\PinablapuE) \un \vert_{1,\E} \vert \bE^\alpha \RE \vert _{1,\E}.\\
\end{split}
\]
As a consequence, applying the stability bounds \eqref{stability}, recalling that $\RE \in \mathbb P_{\pE}(\E)$ and applying the $\h\p$ polynomial inverse estimate on polygons \eqref{inverse estimate Melenk}, one has, for all $1/2 < \alpha \le 2$,
\[
\Vert \bE ^{\frac{\alpha}{2}} \RE \Vert _{0,\E} \lesssim \frac{\pE^{2-\alpha}}{\hE} \left\{ \vert \err \vert_{1,\E} + \alpha_*^{-\frac{1}{2}}(\pE) \SE((I-\PinablapuE)\un, (I-\PinablapuE)\un)^{\frac{1}{2}}  \right\} +\Vert \f-\fn \Vert_{0,\E},
\]
or, equivalently,
\[
\frac{\hE}{\pE} \Vert \bE ^{\frac{\alpha}{2}} \RE \Vert _{0,\E} \lesssim \pE^{1-\alpha} \left\{ \vert \err \vert_{1,\E} + \alpha_*^{-\frac{1}{2}}(\pE) \SE((I-\PinablapuE)\un, (I-\PinablapuE)\un)^{\frac{1}{2}}   \right\} 
+ \frac{\hE}{\pE} \Vert \f - \fn \Vert_{0,\E}.
\]
We are now ready to prove the bound on the internal residual. Recalling that $1/2< \alpha \le 2$, the $\h\p$ polynomial inverse estimate on polygons \eqref{inverse estimates on polygon b} implies
\[
\Vert \RE \Vert_{0,\E} \lesssim \pE^\alpha \Vert \bE^{\frac{\alpha}{2}} \RE \Vert_{0,\E}.
\]
Hence,
\[
\frac{\hE}{\pE} \Vert \RE \Vert_{0,\E} \lesssim \pE
\left\{ \vert \err \vert_{1,\E} + \alpha_*^{-\frac{1}{2}}(\pE) \SE((I-\PinablapuE)\un, (I-\PinablapuE)\un)^{\frac{1}{2}}   \right\} + \hE \pE^{\alpha - 1} \Vert \f - \fn \Vert_{0,\E}.
\]
We pick $\alpha = 1/2 + \varepsilon$, with $\varepsilon >0$ arbitrarily small, and get
\begin{equation} \label{final bound on internal residual}
\frac{\hE^2}{\pE^2} \Vert \RE \Vert_{0,\E} ^2 \lesssim \pE^2 \left( \vert \err \vert _{1,\E}^2 + \alpha_*^{-1}(\pE) \SE((I-\PinablapuE) \un, (I-\PinablapuE) \un)   \right) + \frac{\hE^2}{\pE^{1 - 2 \varepsilon}} \Vert \f - \fn \Vert^2_{0,\E}.
\end{equation}

\subsubsection{Bounding the edge residual} \label{subsection bounding the edge residual}
Next, we bound the edge residual $\Re$ appearing on the right-hand side of \eqref{complete local error estimator}.
We henceforth consider a function $\Rebar$ given by $E(\Re)$,
where the lifting operator $E$ is defined in Lemma \ref{lemma lifting operator}.
We recall that the restriction of $\Rebar$ on $\e$ is equal to $\be^\alpha\, \Re$, being $\be$ defined as the quadratic edge bubble function on $\e$ and where $1/2 < \alpha \le 2$.

In the following, we assume that $\Rebar$ can be extended to $0$ outside $\overline \T_1 \cup \overline \T_2$, see \eqref{triangular patch around e}.
Let us denote by $\E_1$ and $\E_2$ the two polygons containing the triangles $\T_1$ and $\T_2$.

We substitute $\vv = \Rebar$ and $\chin=0$ in \eqref{equazione da cui tutto parte}, obtaining for all~$\e \in \EnI$
\begin{equation} \label{8}
\a(\err, \Rebar) = \sum_{i=1}^2 \left\{ (\REi, \Rebar)_{0,\Ei} + (\f- \fn, \Rebar)_{0,\Ei} - \aEi((I-\Pinablai) \un, \Rebar)   \right\} - (\Re, \Rebar)_{0,\e}.
\end{equation}
We observe that the following bound of the third term on the right-hand side of \eqref{8} holds true:
\begin{equation} \label{9}
\begin{split}
\sum_{i=1}^2 \aEi((I- \Pinablai) \un, \Rebar) 	& \le \sum_{i=1}^2 \vert (I - \Pinablai) \un \vert_{1,\Ei} \vert \Rebar \vert_{1,\Ei}\\
							& \lesssim \sum_{i=1}^2 \left\{ \alpha_*^{-\frac{1}{2}}(\p_{\Ei}) \SEi ((I-\Pinablai)\un, (I-\Pinablai)\un)^{\frac{1}{2}} \right\} \vert \Rebar \vert_{1,\Ei}.
\end{split}
\end{equation}
We deduce from \eqref{8} that
\[
\begin{split}
& \Vert \be ^{\frac{\alpha}{2}}\Re \Vert_{0,\e}^2  = (\Re, \Rebar)_{0,\e} \\
& = -\a(\err, \Rebar) + \sum_{i=1}^2 \left\{ (\REi, \Rebar)_{0,\Ei} + (\f- \fn, \Rebar)_{0,\Ei} - \aEi ((I - \Pinablai)\un, \Rebar)   \right\}.\\
\end{split}
\]
and, combining this with \eqref{9}, we arrive at
\[
\begin{split}
\Vert \be ^{\frac{\alpha}{2}}\Re \Vert_{0,\e}^2  & \lesssim \sum_{i=1}^2 \left\{ \left( \vert \err \vert_{1,\Ei} +\alpha_*^{-\frac{1}{2}}(\p_{\Ei})  \SEi((I-\Pinablai)\un, (I-\Pinablai)\un)^{\frac{1}{2}} \right)   \right. \vert \Rebar \vert_{1,\Ei}\\
& \quad \quad \quad \quad+\left. \left(  \Vert \REi \Vert_{0,\Ei} + \Vert \f - \fn \Vert_{0,\Ei} \right) \Vert  \Rebar \Vert_{0,\Ei} \right\}.\\
\end{split}
\]
Recalling that $\pe \approx \pEi$ for all edges $\e\subset \partial \Ei$, $i=1,2$, see~\eqref{assumption local degrees of accuracy}, and
applying Lemma \ref{lemma lifting operator} with $1/2 < \alpha \le 2$ on $\vert \Rebar\vert_{1,\Ei}$ and $\Vert \Rebar\Vert_{0,\Ei}$, we obtain, for every $\varepsilon>0$,
\[
\begin{split}
&\Vert \be^{\frac{\alpha}{2}}\Re \Vert^2_{0,\e}\\
				& \lesssim \sum_{i=1}^2 \left\{ \left( \vert \err \vert_{1,\Ei} +\alpha_*^{-\frac{1}{2}}(\p_{\Ei})  \SEi((I-\Pinablai)\un, (I-\Pinablai)\un)^{\frac{1}{2}} \right)   \right. \he^{-\frac{1}{2}}
										\left( \varepsilon \pe^{2(2-\alpha)} + \varepsilon ^{-1}  \right) ^{\frac{1}{2}} \Vert \be^{\frac{\alpha}{2}} \Re \Vert_{0,\e}\\
				& \quad\quad\quad\quad+\left. \left(  \Vert \REi \Vert_{0,\Ei} + \Vert \f - \fn \Vert_{0,\Ei} \right) \he^{\frac{1}{2}} \varepsilon ^{\frac{1}{2}} \Vert \be^{\frac{\alpha}{2}} \Re \Vert_{0,\e} \right\}.
\end{split}
\]
Therefore, we can write
\[
\begin{split}
& \left( \frac{\he}{\pe} \right)^{\frac{1}{2}}\Vert \be ^{\frac{\alpha}{2} } \Re \Vert_{0,\e}	\\
& \lesssim \sum_{i=1}^2 \left\{ \left( \vert \err \vert_{1,\Ei} +\alpha_*^{-\frac{1}{2}}(\pE)  \SEi((I-\Pinablai)\un, (I-\Pinablai)\un)^{\frac{1}{2}} \right)   \right. \left( \varepsilon \pe^{2(2-\alpha)} + \varepsilon ^{-1}  \right) ^{\frac{1}{2}} \pe^{-\frac{1}{2}}\\
& \quad\quad\quad\quad+ \left. \varepsilon^{\frac{1}{2}} \frac{\he}{\pe^{\frac{1}{2}}}\left(  \Vert \REi \Vert_{0,\Ei} + \Vert \f - \fn \Vert_{0,\Ei} \right) \right\}.
\end{split}
\]
Hence, squaring both sides, one deduces
\[
\begin{split}
\frac{\he}{\pe}  \Vert  \be^{\frac{\alpha}{2}} \Re \Vert_{0,\e}^2
& \lesssim \sum_{i=1}^2 \left\{ \left( \vert \err \vert_{1,\Ei}^2 +\alpha_*^{-1}(\p_{\Ei})  \SEi((I-\Pinablai)\un, (I-\Pinablai)\un)\right)   \right. \left( \varepsilon \pe^{2(2-\alpha)} + \varepsilon ^{-1}  \right)  \pe^{-1}\\
& \quad\quad\quad\quad+ \left. \varepsilon\frac{\he^2}{\pe}\left(  \Vert \REi \Vert_{0,\Ei}^2 + \Vert \f - \fn \Vert_{0,\Ei} ^2 \right) \right\}.
\end{split}
\]
Using that $\he \le \hEi$ for $i=1,2$, also recalling \eqref{assumption local degrees of accuracy} and \eqref{final bound on internal residual}, we get
\[
\begin{split}
&\frac{\he}{\pe}  \Vert  \be^{\frac{\alpha}{2}} \Re \Vert_{0,\e}^2\\
& \lesssim \sum_{i=1}^2 \left\{ \left( \vert \err \vert_{1,\Ei}^2 +\alpha_*^{-1}(\p_{\Ei})  \SEi((I-\Pinablai)\un, (I-\Pinablai)\un)\right)   \right. \left( \varepsilon \pe^{2(2-\alpha)} + \varepsilon ^{-1}  \right)  \pe^{-1}\\
& \quad \quad \quad \quad \left. +\varepsilon \pe \left(  \pEi^2 \left( \vert \err \vert _{1,\Ei}^2 + \alpha_*^{-1}(\pEi) S^{\Ei}((I-\PinablapuEi) \un, (I-\PinablapuEi) \un)   \right) + \frac{\hEi^2}{\pEi^{1 - 2 \varepsilon}} \Vert \f - \fn \Vert^2_{0,\Ei}  \right) \right\},
\end{split}
\]
whence, using again \eqref{assumption local degrees of accuracy},
\[
\begin{split}
&\frac{\he}{\pe}  \Vert  \be^{\frac{\alpha}{2}} \Re \Vert_{0,\e}^2\\
& \lesssim \sum_{i=1}^2 \left\{ \left( \vert \err \vert_{1,\Ei}^2 +\alpha_*^{-1}(\p_{\Ei})  \SEi((I-\Pinablai)\un, (I-\Pinablai)\un)\right)   \right. \left[  \pe^{-1} \left( \varepsilon \pEi^{2(2-\alpha)} + \varepsilon ^{-1}  \right)  + \pEi ^3 \varepsilon \right]\\
& \quad \quad \quad \quad  \left. +\varepsilon \pEi ^{2(\varepsilon +1)} \left( \frac{\h_{\E_i}}{\pEi} \right)^2 \Vert \f - \fn \Vert^2_{0,\E_i}  \right\}, \quad \frac{1}{2} < \alpha \le 2.\\
\end{split}
\]
Selecting $\varepsilon = \pe^{-2}$ and using once more \eqref{assumption local degrees of accuracy}, one obtains
\[
\begin{split}
&\frac{\he}{\pe}  \Vert  \be^{\frac{\alpha}{2}} \Re \Vert_{0,\e}^2\\
& \lesssim \sum_{i=1}^2 \left\{ \pEi \left( \vert \err \vert_{1,\Ei}^2 +\alpha_*^{-1}(\p_{\Ei})  \SEi((I-\Pinablai)\un, (I-\Pinablai)\un)\right)  + \pEi ^{\frac{2}{\pE^2}} \left(\frac{\h_{\E_i}}{\pEi}\right )^{2} \Vert \f - \fn \Vert^2_{0,\E_i}    \right\}.\\
\end{split}
\]
So far, we have assumed that $1/2 < \alpha \le 2$. In order to get the desired bound on the edge residual, i.e. the one with $\alpha=0$,
we apply Lemma \ref{lemma 1D hp inverse} with $\alpha_1=0$ and $\alpha_2 = 1/2+\varepsilon$, getting
\[
\begin{split}
\frac{\he}{\pe} \Vert \Re \Vert^2_{0,\e} 	& \lesssim \pe^{1+2\varepsilon} \frac{\he}{\pe} \Vert \be^{\frac{1}{2}}\Re \Vert^2_{0,\e}\\
							& \lesssim \sum_{i=1}^2 \pEi^{1+2\, \varepsilon} \left\{  \pEi \left( \vert \err \vert_{1,\Ei}^2 +\alpha_*^{-1}(\p_{\Ei})  \SEi((I-\Pinablai)\un, (I-\Pinablai)\un)  \right) \right.\\
							& \quad \quad \quad \quad \quad \quad \left.+ \pEi ^{\frac{2}{\pE^2}} \left(\frac{\h_{\E_i}}{\pEi}\right )^{2} \Vert \f - \fn \Vert^2_{0,\E_i} \right\}.
\end{split}
\]
Finally, note that for any value of $\pEi \in \mathbb N$ it holds that $\pEi^{\frac{2}{\pEi^2}} \le 3$ and thus such term can be discarded.

\subsection{Conclusions} \label{subsection conclusions}
We collect here the lower and upper bounds discussed in the foregoing Sections \ref{subsection reliability} and \ref{subsection efficiency}.
Note that such bounds are explicit in $\h$ and $\p$.

\begin{thm} \label{theorem hp VEM a posteriori error estimates}
Assume that the assumptions (\textbf{D1})-(\textbf{D2})-(\textbf{P1}) hold true.
Let $\u$ and $\un$ be the solutions to \eqref{Poisson problem homogeneous Dirichlet weak formulation} and \eqref{hp VEM}, respectively, and let $\err = \u - \un$.
For all $\E \in \taun$, let $\eta_{\pE}$ be the error residual defined in \eqref{complete local error estimator} and let $\rhoE$ and $\zetaE$ be defined in \eqref{local error estimators}.
Then, assuming that all the polygons $\E \in \taun$ are convex, the following global upper bound holds true:
\begin{equation} \label{hp VEM upper bound}
\begin{split}
\vert \err \vert^2_{1,\Omega} 	& \lesssim \sum_{\E \in \taun} \left\{ \eta_{\pE}^2 + \rhoE^2 + \max(\alpha_*^{-1}(\pE), \alpha^*(\pE)) \zetaE^2  \right\}.\\
\end{split}
\end{equation}
Further, for every $\E \in \taun$ and for all $\varepsilon > 0$, the following local lower bounds hold true:
\begin{equation} \label{hp VEM lower bound}
\eta_{\pE}^2 +\zetaE^2 + \rhoE^2 \lesssim  \sum_{\E' \in \widetilde \omega _\E} \p_{\E'}^{1+2\, \varepsilon} \left\{  \p_{\E'} \left( \vert \err \vert_{1,\E'}^2 +\alpha_*^{-1}(\p_{\E'})  \zeta_{\p_{\E'}^2)}\right) 
			+ \rho_{\p_{\E'}}^2 \right\},
\end{equation}
where $\widetilde \omega_\E = \cup\{\E' \in \taun \mid \partial \E' \cap \partial \E \ne \emptyset\}$.
The hidden constants in \eqref{hp VEM upper bound} and \eqref{hp VEM lower bound} depend solely on $\varepsilon$ and on the shape-regularity of $\taun$ and hence of $\tautilden$, see Remark \ref{remark triangular subdecomposition}.
\end{thm}
%

On the light of Theorem \ref{theorem hp VEM a posteriori error estimates}, the global (computable) error estimator that we propose is
\begin{equation} \label{actual error estimator}
\eta_{comp}^2 = \sum_{\E\in \taun} \eta^2_{comp,\E}=\sum_{\E\in \taun} ( \eta_{\pE}^2 +\zetaE^2+\rhoE^2).
\end{equation}
Note that in the bounds \eqref{hp VEM upper bound} and \eqref{hp VEM lower bound} there appear additional multiplicative terms depending on $\max (\alpha_*^{-1}(\pE), \alpha^*(\pE))$.
On the other hand, such terms were numerically shown in \cite{hpVEMcorner} to have a very mild dependence in terms of $\pE$.
We refer to Remark \ref{remark on pollution factor} for additional comments regarding the effects of the ``stability'' terms $\alpha_*^{-1}(\pE)$ and $\alpha^*(\pE)$.


{We also underline that on the right-hand side of~\eqref{hp VEM lower bound}, in addition to the energy error, other two terms appear.
The term $\rho_{\p_{\E}}$ is related to the oscillation of $\f$, the datum of the problem~\eqref{strong Poisson problem}, and is typical also in the finite element framework.
On the other hand, the term $\zeta_{\p_{\E}}$ deals with the nonexactness of the discrete bilinear form and is standard in a posteriori error analysis of VEM, see~\cite{ManziniBeirao_VEMresidualaposteriori, cangianigeorgulispryersutton_VEMaposteriori}.}

\begin{remark} \label{remark theorem non convex}
In presence of nonconvex polygons, the bound \eqref{hp VEM upper bound} needs to be modified; in particular, it appears in front of $\eta_{\pE}^2$ a suboptimal factor $\p^{2\frac{\pi}{\alpha_\E} - \varepsilon}$
for all $\varepsilon >0$ arbitrarily small, $\alpha_\E$ being the largest interior angle of $\E$.
Moreover, assuming that the assumption (\textbf{D2}) does not hold true, the estimates would get more involved, as one should take care of different scaling in terms of the size of elements and edges
and of the effects of small edges on the stabilization.
\end{remark}

\section{Numerical results} 
\label{section numerical results}
In this section, we present a set of numerical experiments investigating on the performances of the error estimator introduced in~\eqref{actual error estimator}.
More precisely, we validate in Section~\ref{subsection performances of the error estimators in terms of p} the estimates presented in Theorem~\ref{theorem hp VEM a posteriori error estimates},
whereas, in Section \ref{subsection refinement strategy}, we recall from \cite{MelenkWohlmuth_hpFEMaposteriori}
an $\h\p$ refinement algorithm which permits us to apply the adaptive $\h\p$ virtual element method on a number of test cases.

Before presenting the results, we have to settle some features of the method, which so far were kept at a very general level. First of all, we underline that in the virtual element setting, it is not possible to compute explicitly the exact $H^1$ error (since functions in virtual element spaces are not known in closed-form) and therefore we compute instead the following quantity, that is a classical choice in the VEM literature:
\begin{equation} \label{computable error}
\vert u - \Pinablaa \un \vert_{1,\taun},
\end{equation}
where $(\Pinablaa \un)_{|{\E}} = \Pinabla(\un)_{|_{\E}}$ for all $\E \in \taun$, see~\eqref{H1 projector}.
By a triangle inequality and continuity of the $\Pi^\nabla$ operator, it can be easily checked that the quantity above differs with $\vert u -  \un \vert_{1,\Omega}$ by a piecewise polynomial approximation term (and thus holds the same behaviour essentially in all cases of interest).


Another important aspect of the method is the choice of the stabilization in \eqref{stabilizing bilinear form}. We adopt here the so called ``D-recipe'', which was firstly introduced in \cite{VEM3Dbasic} and whose performances were investigated in deep in \cite{fetishVEM,fetishVEM3D}.
If we denote by $\mathbf{S^\E}$ the matrix representing the stabilization $\SE$ on element $\E$ with respect to the canonical basis \eqref{canonical basis}, then we set
\begin{equation} \label{D-recipe}
\mathbf\S^{\mathbf\E}_{i,j} =\max\{ 1,\aE(\Pinabla\varphi_j, \Pinabla \varphi_i)   \} \delta_{i,j},
\end{equation}
where $\delta_{i,j}$ denotes the Kronecker delta.
Among the possible stabilizations available in the virtual element literature, the one defined in \eqref{D-recipe} is one of the most appealing in terms of robustness of the method,
both when considering high degrees of accuracy and in presence of badly-shaped elements.

Finally, we address the issue of picking a ``clever'' polynomial basis dual to the internal moments \eqref{internal moments}.
In particular, following again \cite{fetishVEM, fetishVEM3D}, we employ a basis $\{ \m_{\boldalpha} \}_{\vert\boldalpha\vert=0}^{\p-2}$ which is $L^2(\E)$ orthonormal for all $\E \in \taun$.
Such a basis can be built, for instance, by orthonormalizing a basis of scaled and shifted monomials.
Picking an orthonormal basis dual to internal moments is a choice particularly suited for the $\p$ version of the method since it allows to effectively damp the condition number of the stiffness matrix for high values of the polynomial degree.

In the numerical experiments, both the error estimator and the computable error \eqref{computable error} are normalized by $\vert u \vert_{1,\Omega}$.

\subsection{Performances in terms of $\p$ of the error estimator} \label{subsection performances of the error estimators in terms of p}
In this section, we investigate the performances of the error estimator for uniform $\p$ refinements; that is, we fix a coarse mesh and we achieve convergence by raising the polynomial degree in all mesh elements
(note that the performances of $\h$ refinements where the topic of \cite[Section 6.1]{cangianigeorgulispryersutton_VEMaposteriori} and therefore are not explored here).
To this end, we consider three test cases with known exact solution
\begin{equation} \label{3 solutions bonta}
\begin{split}
& u_1(x,y) = \sin(\pi\,x) \, \sin (\pi\, y) \quad \text{in } \Omega_1 = [0,1]^2,\\
& u_2(r,\theta) = r^2(\log(r) \sin (2\theta) + \theta \cos(2\theta)) \quad \text{in } \Omega_1,\\
& u_3(r,\theta) = r^{\frac{2}{3}} \sin\left(\frac{2}{3}\left(\theta+\frac{\pi}{2} \right) \right) \quad \text{in } \Omega_2 = [-1,1]^2\setminus[-1,0]^2.\\
\end{split}
\end{equation}
As usual, the loading term and the (Dirichlet) boundary conditions are set in accordance with the exact solution.
Function $u_1$ is analytic, function $u_2$ belongs to $H^{3-\varepsilon}(\Omega_1)$ for all $\varepsilon >0$ and is the ``natural'' singular solution to an elliptic problem over square $\Omega_1$,
function $u_3$ belongs to $H^{\frac{5}{3}-\varepsilon}(\Omega_2)$ for all $\varepsilon >0$ and is the ``natural'' singular solution to an elliptic problem over the L-shaped domain $\Omega_2$.

We test the performances of the error indicator employing two types of mesh, namely a Cartesian and a Voronoi mesh. In Figure \ref{figure meshes}, we depict the two meshes (on domain $\Omega_1$), 
their counterparts in the L-shaped domain $\Omega_2$ being analogous.
\begin{figure}  [h]
\centering
\subfigure {\includegraphics [draft=false, angle=0, width=0.49\textwidth]{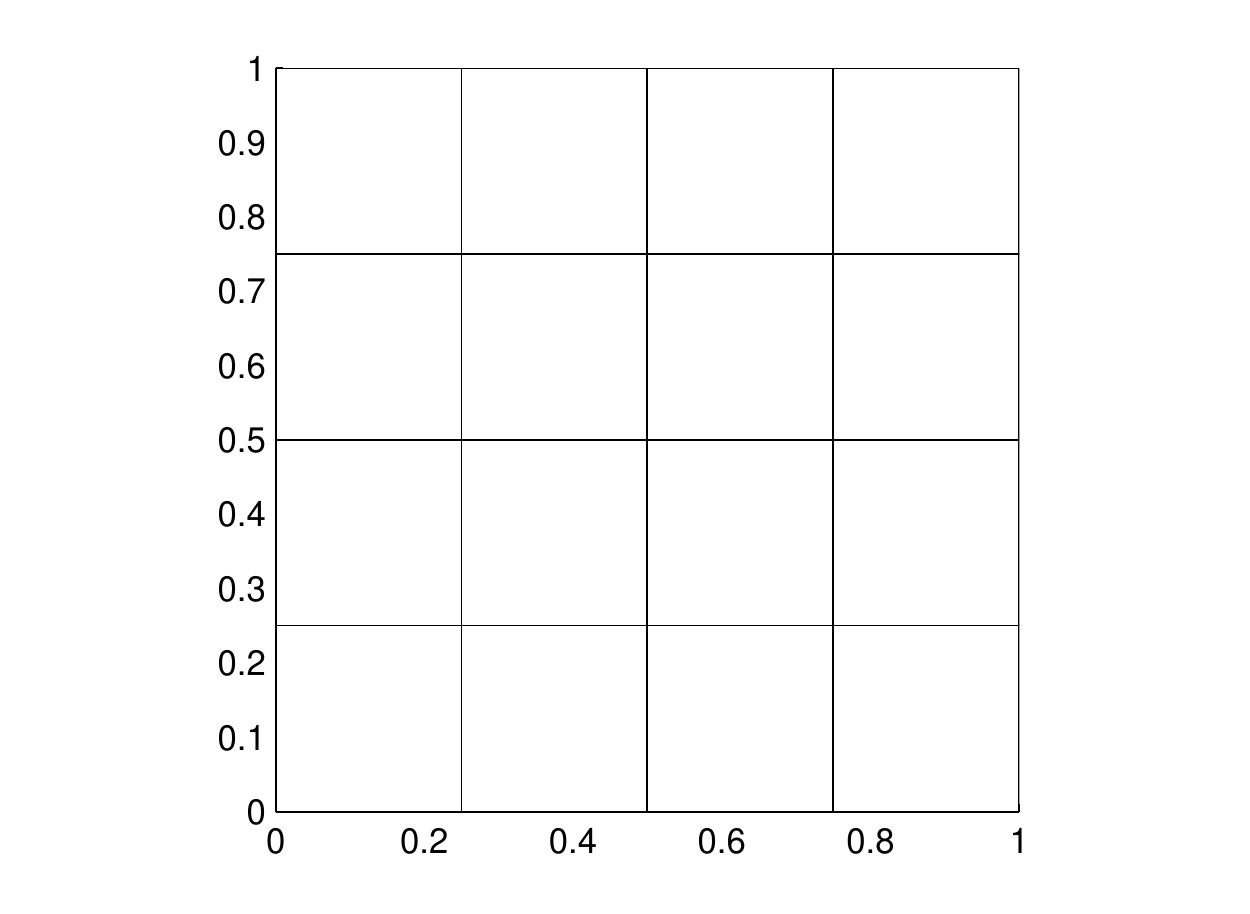}}
\subfigure {\includegraphics [draft=false, angle=0, width=0.49\textwidth]{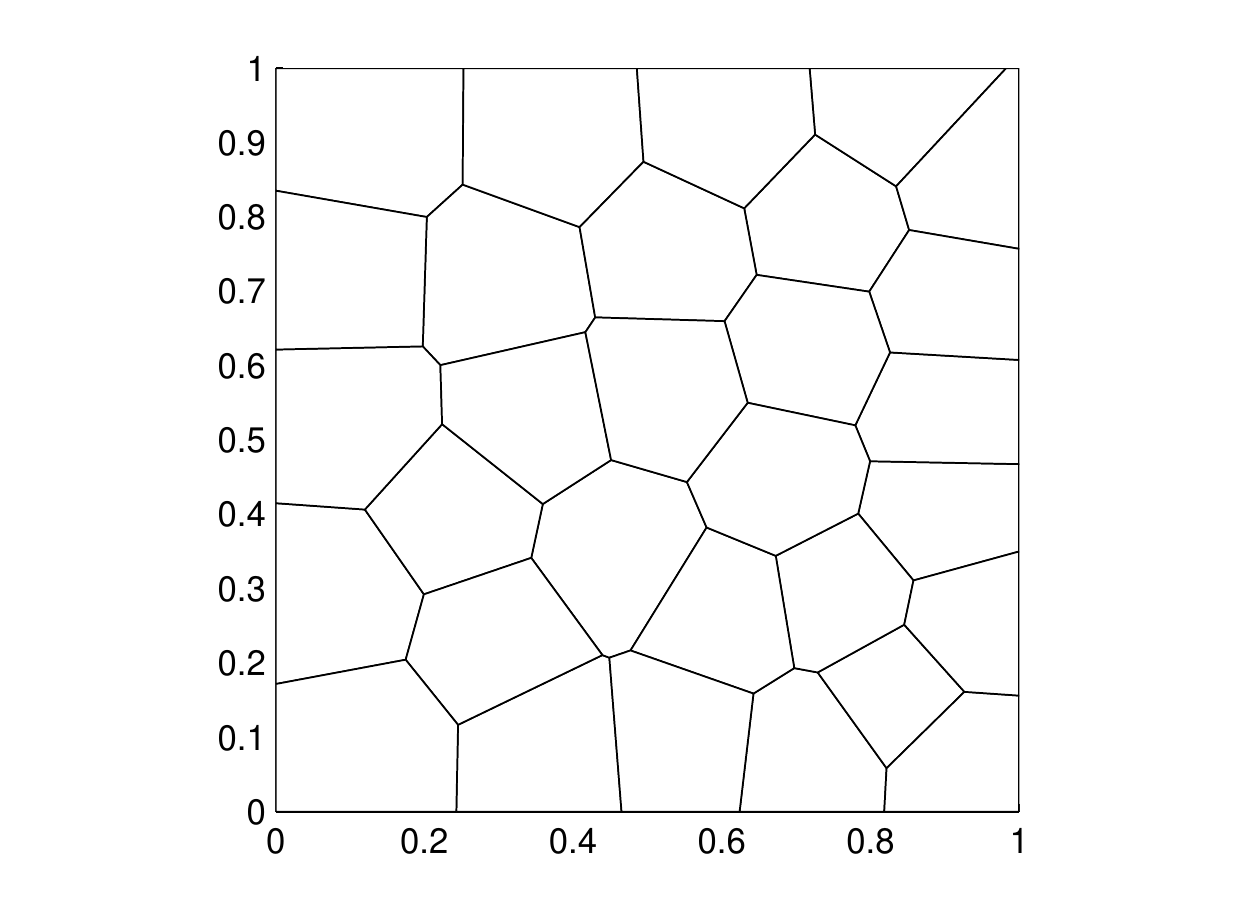}}
\caption{Left: Cartesian mesh. Right: Voronoi mesh.} \label{figure meshes}
\end{figure}

In Figures \ref{figure bonta square} and \ref{figure bonta Voronoi}, we plot (for different values of the degree $p$) the computable error \eqref{computable error} versus the computed error estimator \eqref{actual error estimator} for the three exact solutions in \eqref{3 solutions bonta} on the two meshes of Figure \ref{figure meshes}.
\begin{figure}  [h]
\centering
\subfigure {\includegraphics [draft=false, angle=0, width=0.48\textwidth]{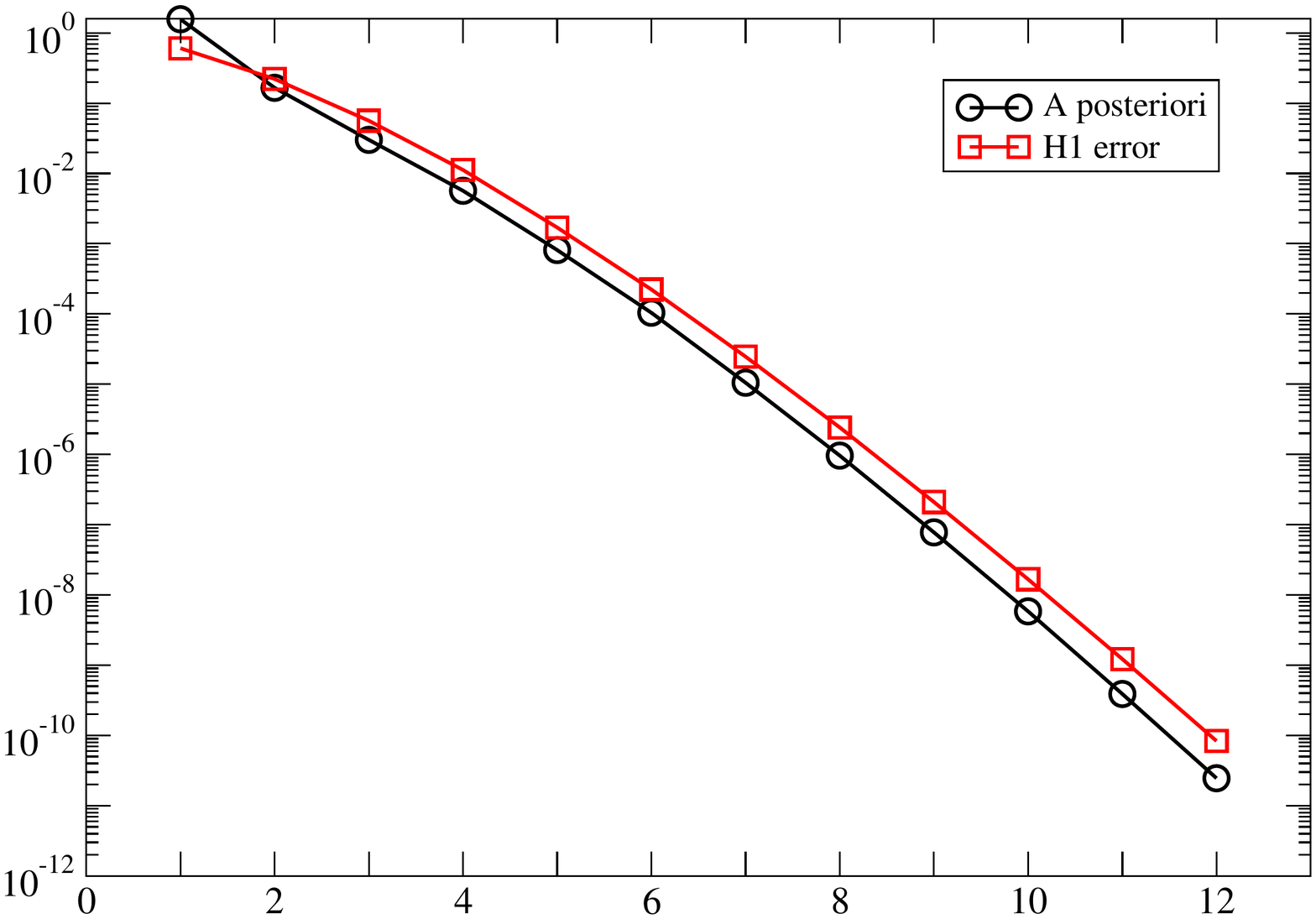} \put(-103,10){\p}}
\subfigure {\includegraphics [draft=false, angle=0, width=0.48\textwidth]{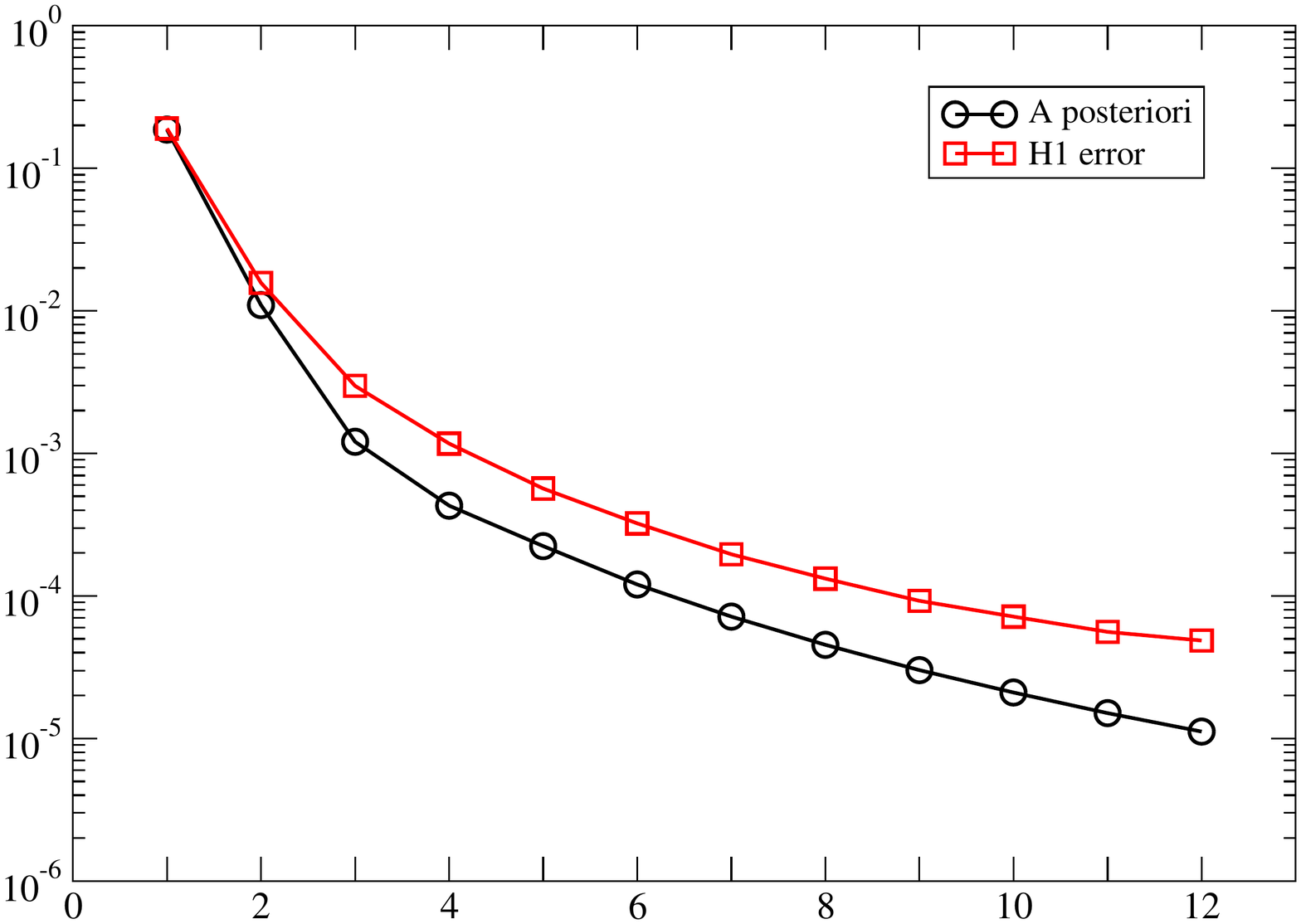} \put(-103,10){\p}}
\subfigure {\includegraphics [draft=false, angle=0, width=0.48\textwidth]{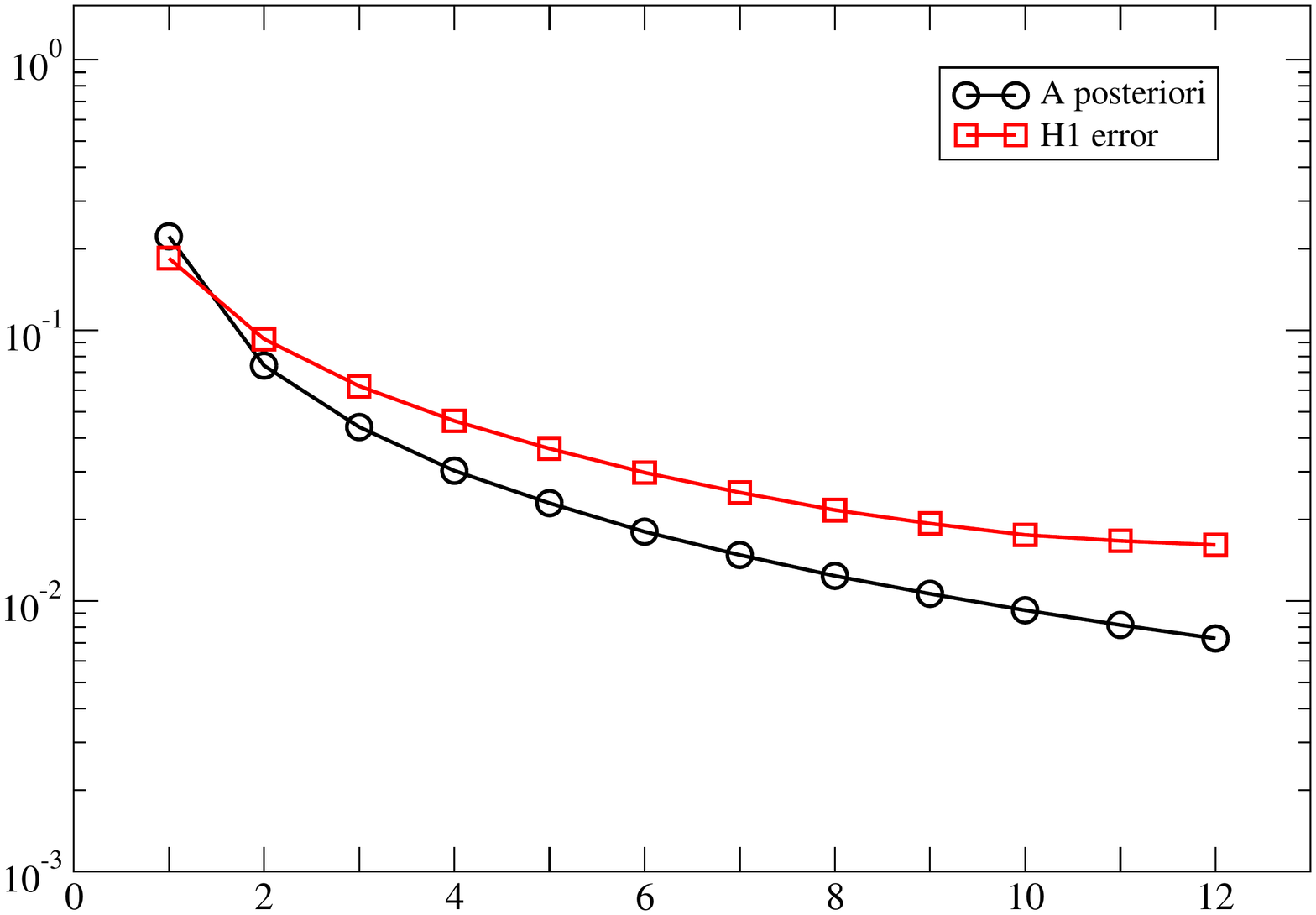} \put(-103,10){\p}}
\caption{Cartesian mesh. Computable error \eqref{computable error} versus error estimator \eqref{actual error estimator} in terms of $p$. Left: solution $u_1$. Right: solution $u_2$. Below: solution $u_3$.} \label{figure bonta square}
\end{figure}
\begin{figure}  [h]
\centering
\subfigure {\includegraphics [draft=false, angle=0, width=0.48\textwidth]{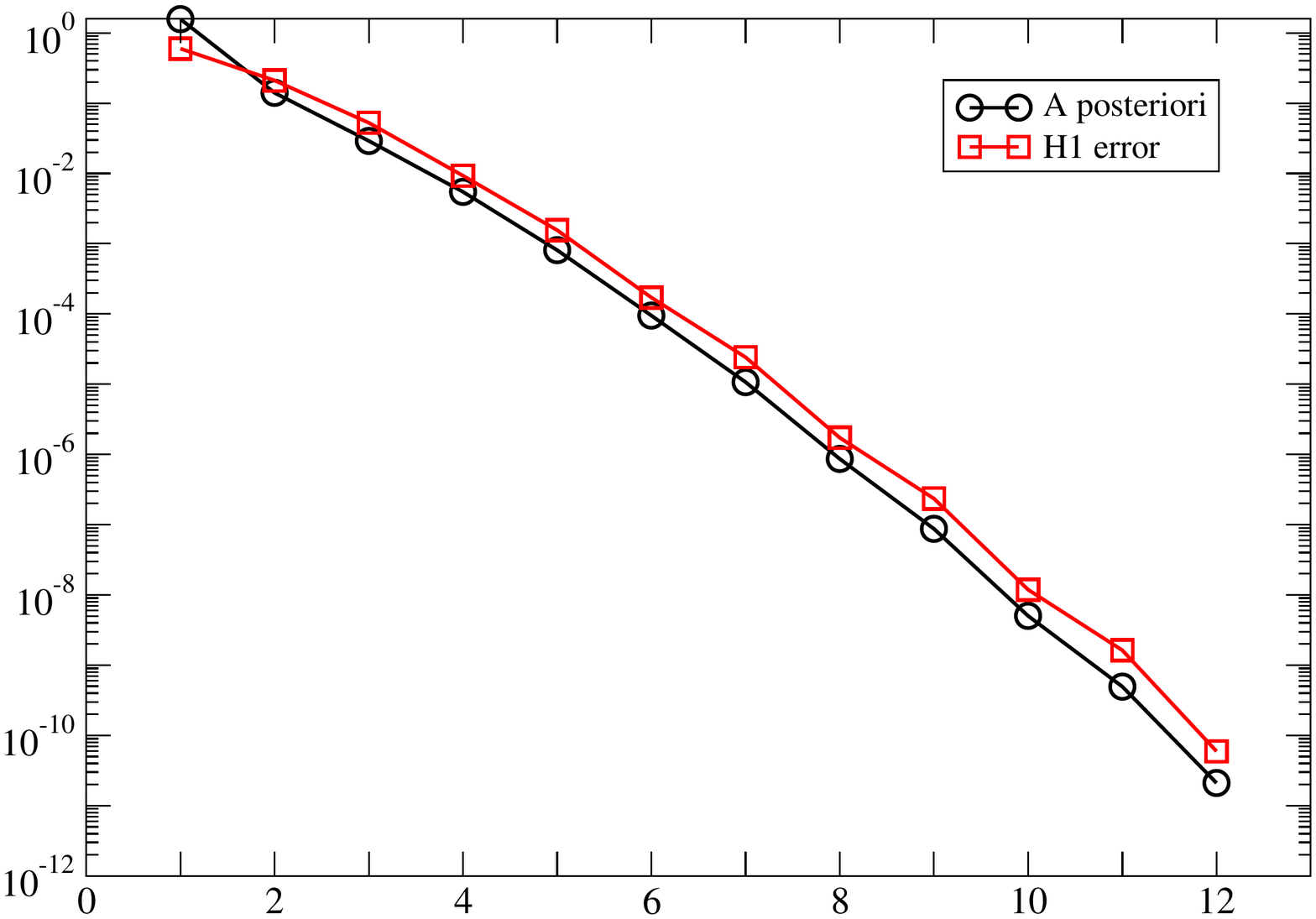} \put(-103,10){\p}}
\subfigure {\includegraphics [draft=false, angle=0, width=0.48\textwidth]{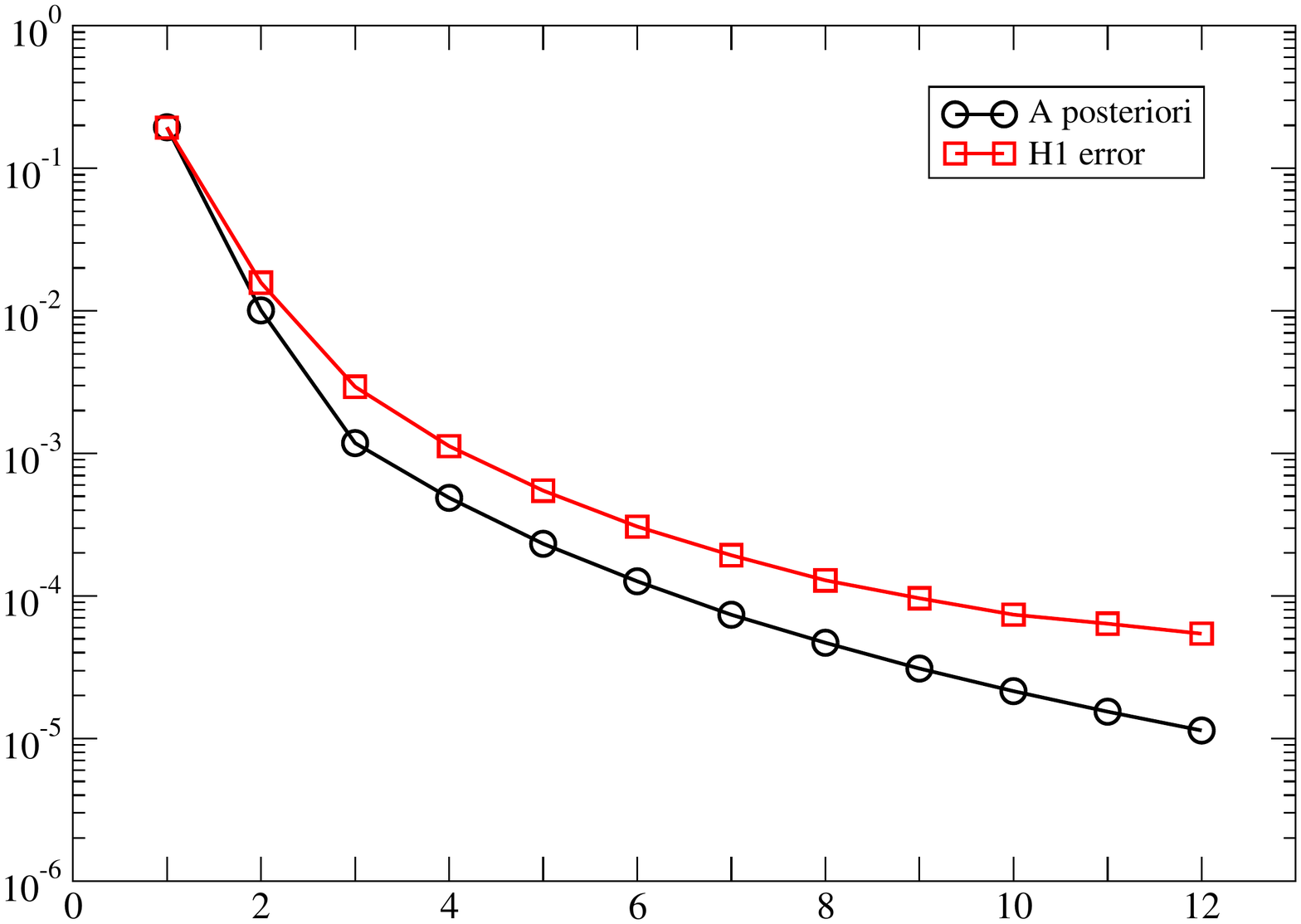} \put(-103,10){\p}}
\subfigure {\includegraphics [draft=false, angle=0, width=0.48\textwidth]{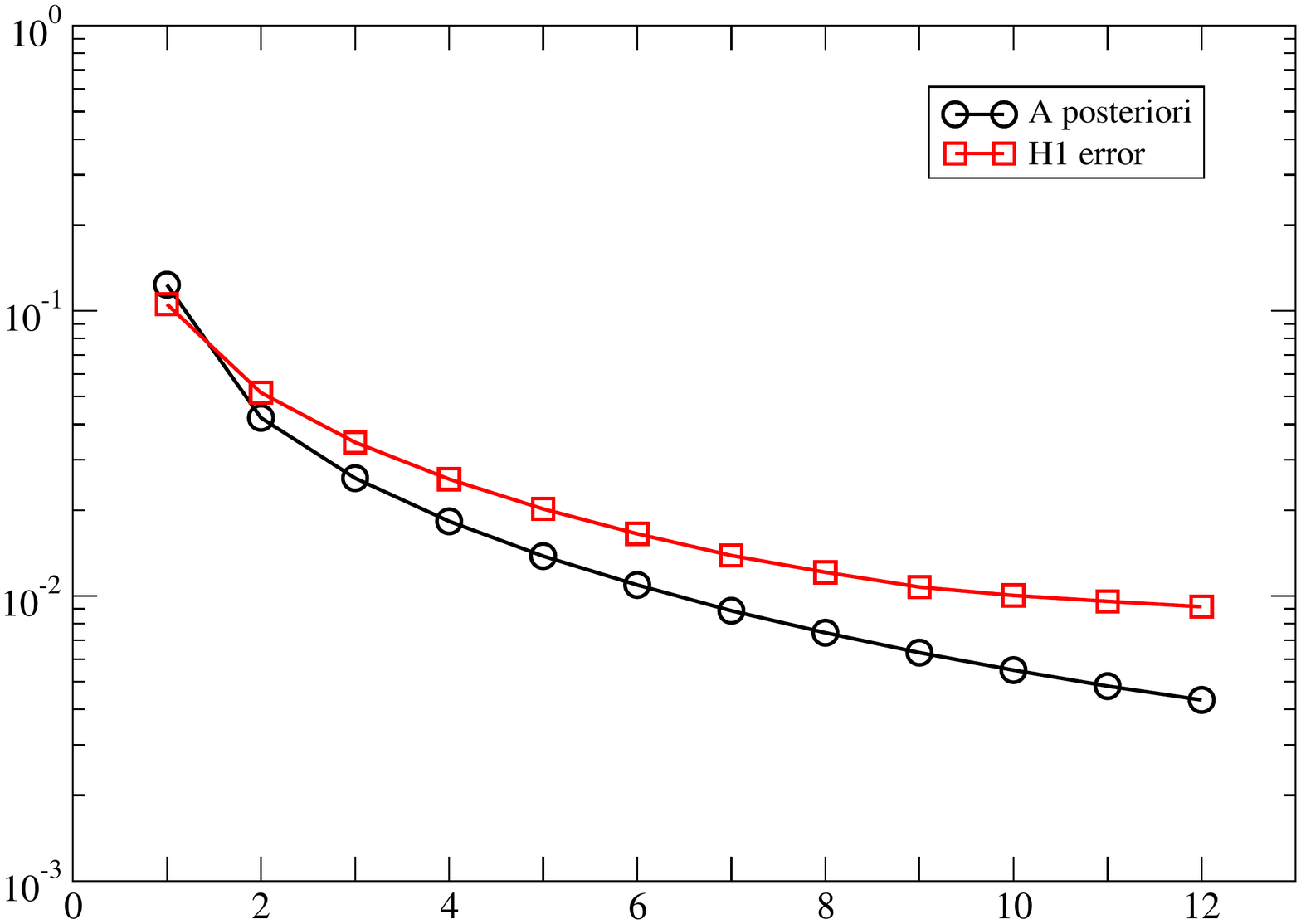} \put(-103,10){\p}}
\caption{Voronoi mesh. Computable error \eqref{computable error} versus error estimator \eqref{actual error estimator} in terms of $p$. Left: solution $u_1$. Right: solution $u_2$. Below: solution $u_3$.} \label{figure bonta Voronoi}
\end{figure}

We observe that, when approximating analytic solutions, the error estimator and the computed error behave practically in the same way.
For singular functions, the behaviour is instead slightly different and in particular the error estimator is slightly smaller than the computed error.
This is due to the use of inverse estimates producing extra factors of~$\p$ and to the pollution effect of the stabilization, see \eqref{hp VEM upper bound}-\eqref{hp VEM lower bound};
however, if the solution is analytic, the pollution factors due to the inverse estimates and the stabilization are negligible, since all the local error estimators, as well as the exact error, decrease exponentially in terms of $\p$, see \cite{hpVEMbasic}.

\subsection{The $\h\p$ adaptive refinement strategy} \label{subsection refinement strategy}
In this section, we present an $\h\p$ refinement strategy, based on the standard procedure
\[
\textrm{SOLVE} \rightarrow \textrm{ESTIMATE} \rightarrow \textrm{MARK}\rightarrow\textrm{REFINE}.
\]
Before discussing the $\h\p$ refinement strategy, we describe how to perform $\h$ refinements.
The standard strategy that one can follow consists in subdividing a polygon $\E$ with $N^\E$ ``edges'' into a number of ``quadrilaterals'' smaller than or equal to $N^\E$ obtained by connecting the barycenter of $\E$ to the midpoints of its ``edges''.
Note that 
by ``edge'' we mean a straight line on the boundary of a polygonal element, possibly consisting of more edges belonging to the polygonal decomposition; in particular, an edge could contain hanging nodes.
Besides, by ``quadrilateral'', we mean polygons with four ``edges'', that is, such ``quadrilateral'' could have more than four edges, but precisely four ``edges''. See Figure~\ref{figure standard h refinement}.

We observe that this procedure may generate very small edges, thus contradicting the assumption (\textbf{D2}).
Notwithstanding, we will not experience in the forthcoming numerical experiments any loss in accuracy due to the presence of small edges, which in any case appear rarely.
\begin{figure}[h]
\centering
\begin{minipage}{0.32\textwidth}
\begin{center}
\begin{tikzpicture}[scale=1.2]
\draw[black, thick, -] (-1,0) -- (1,0) -- (2,2) -- (0,4) -- (-2,2) -- (-1,0);
\draw[black, thick, -] (-1,0) -- (-1.5,-0.5); \draw[black, thick, -] (1,0) -- (1.5, -0.5); \draw[black, thick, -] (0,4) -- (0, 4.5); \draw[black, thick, -] (2,2) -- (2.5,2); \draw[black, thick, -] (-2,2) -- (-2.5,2);
\draw[black, dashed] (0,8/5) -- (3/2,1); \draw[black, dashed] (0,8/5) -- (1, 3); \draw[black, dashed] (0, 8/5) -- (-1, 3); \draw[black, dashed] (0, 8/5) -- (-3/2, 1);
\draw[black, dashed] (0, 8/5) -- (0, 0);
\draw[fill=red] (-1,0) circle (3pt);
\draw[fill=red] (0,0) circle (3pt);  \draw[fill=red] (1,0) circle (3pt); 
\draw[fill=red] (2,2) circle (3pt); \draw[fill=red] (0,4) circle (3pt); \draw[fill=red] (-2,2) circle (3pt); 
\end{tikzpicture}
\end{center}
\end{minipage}
\quad\quad\quad\quad\quad\quad\quad\quad\quad\quad\quad
\begin{minipage}{0.32\textwidth}
\begin{center}
\begin{tikzpicture}[scale=1.2]
\draw[black, thick, -] (-0.4,0) -- (4.4,0); \draw[black, thick, -] (0, -0.4) -- (0, 4.4); \draw[black, thick, -] (4,-.4) -- (4,4.4); \draw[black, thick, -] (-0.4,4) -- (4.4,4);
\draw[black, dashed] (2,2) -- (2,0);  \draw[black, dashed] (2,2) -- (2,4);  \draw[black, dashed] (2,2) -- (0,2);  \draw[black, dashed] (2,2) -- (4,2); 
\draw[fill=red] (0,0) circle (3pt); \draw[fill=red] (4,0) circle (3pt); \draw[fill=red] (0,4) circle (3pt); \draw[fill=red] (4,4) circle (3pt); \draw[fill=red] (2,0) circle (3pt); \draw[fill=red] (1,0) circle (3pt); \draw[fill=red] (0.5,0) circle (3pt); \draw[fill=red] (0.25,0) circle (3pt);
\end{tikzpicture}
\end{center}
\end{minipage}
\caption{Standard~$\h$ refinement strategy. One connects the barycenter of the element with the midpoints of every ``edge''.
As an exception, whenever two or more edges belong to the same straight line (``edge''), one connects the barycenter with the midpoint of such straight line.
The red dots denote the vertices; the original polygon is denoted by solid lines; the new ``quadrilateral'' are highlighted with dashed lines.}
\label{figure standard h refinement}
\end{figure}
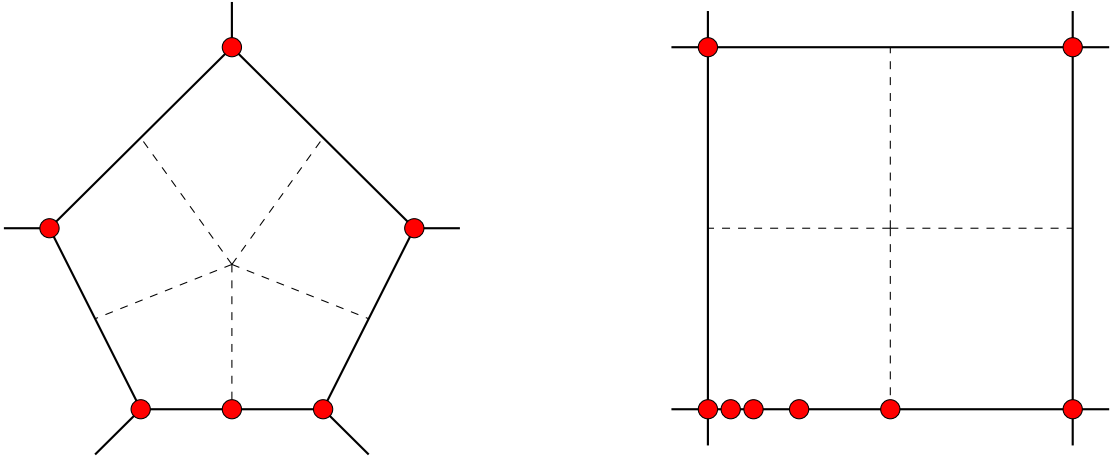

We stress that this procedure can be performed only on elements containing their barycenter (e.g. convex elements), but, importantly, starting from a convex element this construction generates only convex elements.
Besides, the hanging nodes popping-up with this strategy fit naturally in the polygonal framework of VEM.
Note that, when dealing with Cartesian meshes, the refinement strategy of Figure \ref{figure standard h refinement} reduces to standard square refinement.

We now briefly review the $\h\p$ refinement strategy of~\cite[Section 4.2]{MelenkWohlmuth_hpFEMaposteriori}, here adapted to the present method.
The basic idea behind this approach is that, on the elements on which the exact solution to the continuous problem is ``regular'', one performs $\p$ refinement
(since the $\p$ version of the method leads to exponential convergence of the error in terms of $\p$ when approximating analytic solutions, see~\cite{hpVEMbasic}),
whereas, on the elements where the solution has a ``singular'' behaviour, one performs $\h$ refinement
(since geometric mesh refinements lead to exponential convergence of the error in terms of the number of degrees of freedom when approximating singular functions, cf. \cite{hpVEMcorner}).


The $\h\p$ refinement strategy reads as follows.
At the $n$-th level of the algorithm, one marks the elements on which the error estimator is sufficiently large. To this end,
one marks all the elements such that $\eta^2_{comp,\E,n} \ge \sigma \etabarn^2$ for some $\sigma\in (0,1)$, where
\begin{equation} \label{average error estimator}
\etabarn^2 = \frac{1}{\text{card}(\taun)} \eta_{comp,n}^2,
\end{equation}
being $\eta^2_{comp,\E,n}$ and $\eta_{comp,n}$ computed as in \eqref{actual error estimator} for all $\E \in \taun$ and for all $n \in \mathbb N$.

Let us now be given, at the $n$-th adaptive refinement step, a ``prediction-estimator'' $\etapredpE$ on each element $\E$ of the mesh $\taun$. Such estimator is instrumental in order to heuristically decide whether (elementwise) the exact solution is regular or not.

If one performs on element $\E$ an $\h$ refinement, then the expectations are, under the hypothesis of analyticity of $\u$,
that the error reduces by a factor decreasing exponentially in terms of $\p$, see \cite{hpVEMbasic,hpVEMcorner}.
Therefore, given $N^\E$ the number of son elements of $\E$, one sets the prediction-estimator $\etapredpEi$ on $\Ei$ to:
\[
\eta_{\text{pred}, \ES ,n+1}^2 = \frac{1}{N^\E}\gamma_\h (0.5)^{2\pE} \eta^2_{comp,\E,n},
\]
for some $\gamma_\h$ fixed positive parameter. Note that the term $0.5$ in the above equation stems from the fact that we are roughly dividing the element diameter $h_K$ by two when performing the procedure in Figure~\ref{figure standard h refinement}.

Instead, if one performs on element $\E$ a $\p$ refinement, one expects, under the hypothesis of analyticity of $\u$, a reduction of the error by a fixed factor $\gamma_\p$.
Thus, one sets the prediction-estimator $\eta_{\text{pred},\pE+1}$ on $\E$ to:
\[
\eta_{\text{pred}, \E ,n+1}^2 = \gamma_\p \, \eta^2_{comp,\E,n},
\]
for some $0 < \gamma_\p < 1$ fixed parameter.

On the elements that are not marked for refinement, one updates in a trivial fashion as
\[
\eta_{\text{pred}, \E, n+1}^2 = \gamma_n \, \eta_{\text{pred}, \E, n}^2,
\]
for some $\gamma_n$ fixed positive parameter.

So far, we have constructed, at the $n$-th level of the $\h\p$ refinement process, a set of prediction-estimators for the $n+1$-th level;
as already underlined, such prediction-estimators have the scope of ``guessing'' elementwise the behaviour of the exact solution.

What one does at this point is that on all marked elements he checks whether $$\eta^2_{comp,\E, n} \ge  \, \eta^2_{pred,\E,n}.$$
If this is the case, the actual error estimator is ``larger'' than the predicted one, where we have assumed that the solution was analytic;
this means that the solution is not sufficiently ``regular'' on the element and therefore an $\h$ refinement has to be performed.
The $\p$ refinement is effectuated otherwise.

Since in the first refinement step the prediction-estimators are not available (since there is not a 0 level in the adaptive construction of the spaces),
the $\h\p$ refinement boils down to a simple $\h$ refinement. To this purpose, in the first iteration of the adaptive algorithm, it suffices to set $\eta^2_{pred,\E,0} = \frac{\eta^2_{comp,\E,0}}{2}$ on all elements $\E$.

\medskip

In Algorithm \ref{algorithm hp refinement strategy}, we present the $\h\p$ refinement process discussed so far.

\begin{algorithm}[H]
\caption{$\h\p$ refinement algorithm.}
\label{algorithm hp refinement strategy}
\begin{algorithmic}
\State given fixed positive parameters $\sigma$, $\gamma_\h$, $\gamma_\p$, and $\gamma_n$:
\State fix $\eta^2_{pred, \E, 0} = \frac{\eta^2_{comp, \E, 0}}{2}$ on each $\E\in \taun$;
\For {$n\in \mathbb N$ (until some stop criterion is fulfilled)}
\If{$\eta_{comp, \E, n}^2\ge \sigma \etabarn^2$}
\State mark element $\E$ for refinement
\If{$\eta_{comp, \E, n}^2 \ge  \eta^2_{pred,\E,n}$}
\State $\h$ refinement ($\E$ subdivided into $N^\E$ sons $\E_S$)
\State $\eta_{\text{pred}, \E_S, n+1}^2 = \frac{1}{N^\E}\gamma_\h \, (0.5)^{2\pE} \eta^2_{comp, \E, n}$
\Else
\State $\p$ refinement
\State degree of accuracy on $\E$ increased by $1$
\State $\eta^2_{pred,\E,n+1} = \gamma_\p\, \eta^2_{comp,\E,n}$
\EndIf
\Else{}
\State no refinement 
\State $\eta^2_{pred,\E,n+1} = \gamma_n\, \eta^2_{pred,\E,n}$
\EndIf
\EndFor
\end{algorithmic}
\end{algorithm}
\begin{remark} \label{remark small edges}
It may occur within the adaptive algorithm that the assumptions (\textbf{D2}), guaranteeing the presence of edges with size comparable to that of the element to whom they belong, and (\textbf{P1}), guaranteeing comparable degrees of accuracy on neighbouring elements, are not valid.
However, from the forthcoming numerical experiments, it is evident that the method is  robust in this respect.
\end{remark}
\begin{remark} \label{remark australia}
The approach of Algorithm \ref{algorithm hp refinement strategy} is not the only one available in the framework of $\h\p$ Galerkin methods.
We refer to \cite{mitchell2011survey}  for an overview of different $\h\p$ refinement approaches.
\end{remark}
In our experiments, we set the parameters in Algorithm \ref{algorithm hp refinement strategy} to
\cite{MelenkWohlmuth_hpFEMaposteriori}
\[
\sigma=0.75,\quad \gamma_\h = N^{\E},\quad \gamma_\p =0.4,\quad \gamma_n=1,
\]
where we recall that $N^\E$ is the number of elements into which polygon $\E$ is split in case of $\h$ refinement.

Moreover, we test the method on the test case with known solution $u_3$, see \eqref{3 solutions bonta}, and
\begin{equation} \label{u4}
u_4(x,y) = x(1-x)y(1-y)\exp{(-100(y-0.5)^2-100(x-0.5)^2)} \quad \text{in } \Omega_1,
\end{equation}
where we recall that $\Omega_1$ is the square domain introduced in \eqref{3 solutions bonta}.
We underline that the solution $u_4$ is analytic but has a steep derivatives around $(0.5,0.5)$.

\begin{remark} \label{remark on pollution factor}
Algorithm \ref{algorithm hp refinement strategy} is based on $\h$ refinements where the solution is assumed to have a ``singular'' behaviour, whereas it is based on $\p$ refinements where the solution is assumed to be ``smooth''.
Therefore, the pollution factor $\max(\alpha_*^{-1}(\p), \alpha^*(\p))$ appearing in the estimates of Theorem \ref{theorem hp VEM a posteriori error estimates}, behaves in the following \emph{heuristic} fashion.
It does not blow up when doing $\h$ refinements since it is independent of $\h$.
Instead, when doing $\p$ refinements, it multiplies the factor $\zetaE$ defined in \eqref{local error estimators} which, since we are assuming that the solution is analytic in case of $\p$ refinements, is converging exponentially, in terms of $\p$.
Thus, the pollution factor, which typically grows algebraically in terms of $\p$, see \cite[Lemma 2.5]{pVEMmultigrid} and \cite[Theorem 2]{hpVEMcorner}, is not in principle spoiling the behaviour of the adaptive scheme. This observation is confirmed by the numerical experiments below.
\end{remark}

In Figures \ref{figure apos hp u3 square} and \ref{figure apos hp u4 square}, we show the $\h\p$ mesh refinement after 3 and 10 steps of Algorithm \ref{algorithm hp refinement strategy} applied
to the two test cases with known solution $u_3$ and $u_4$, starting from a coarse Cartesian mesh.
\begin{figure}  [h]
\centering
\subfigure {\includegraphics [draft=false, angle=0, width=0.48\textwidth]{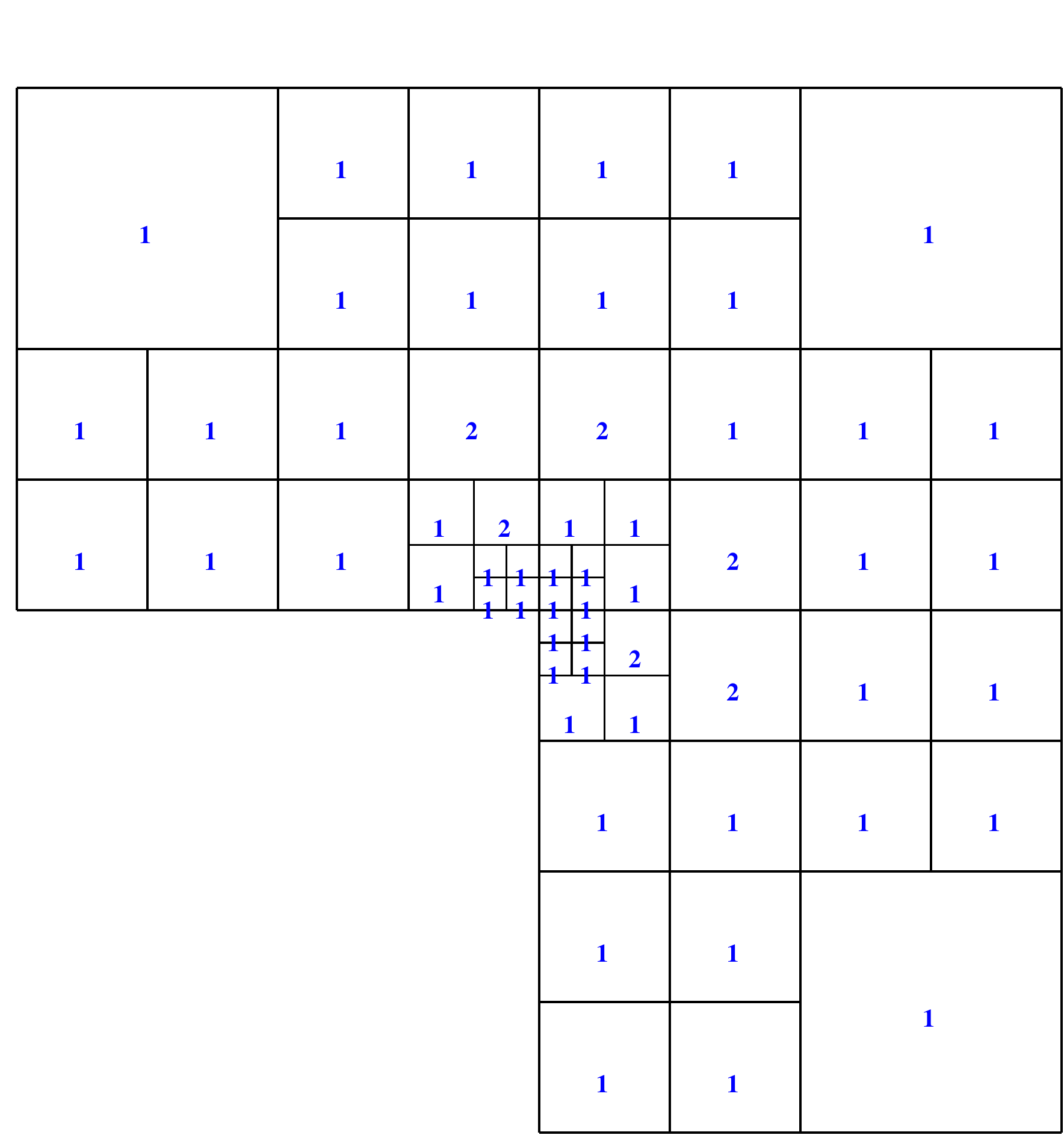}}
\subfigure {\includegraphics [draft=false, angle=0, width=0.48\textwidth]{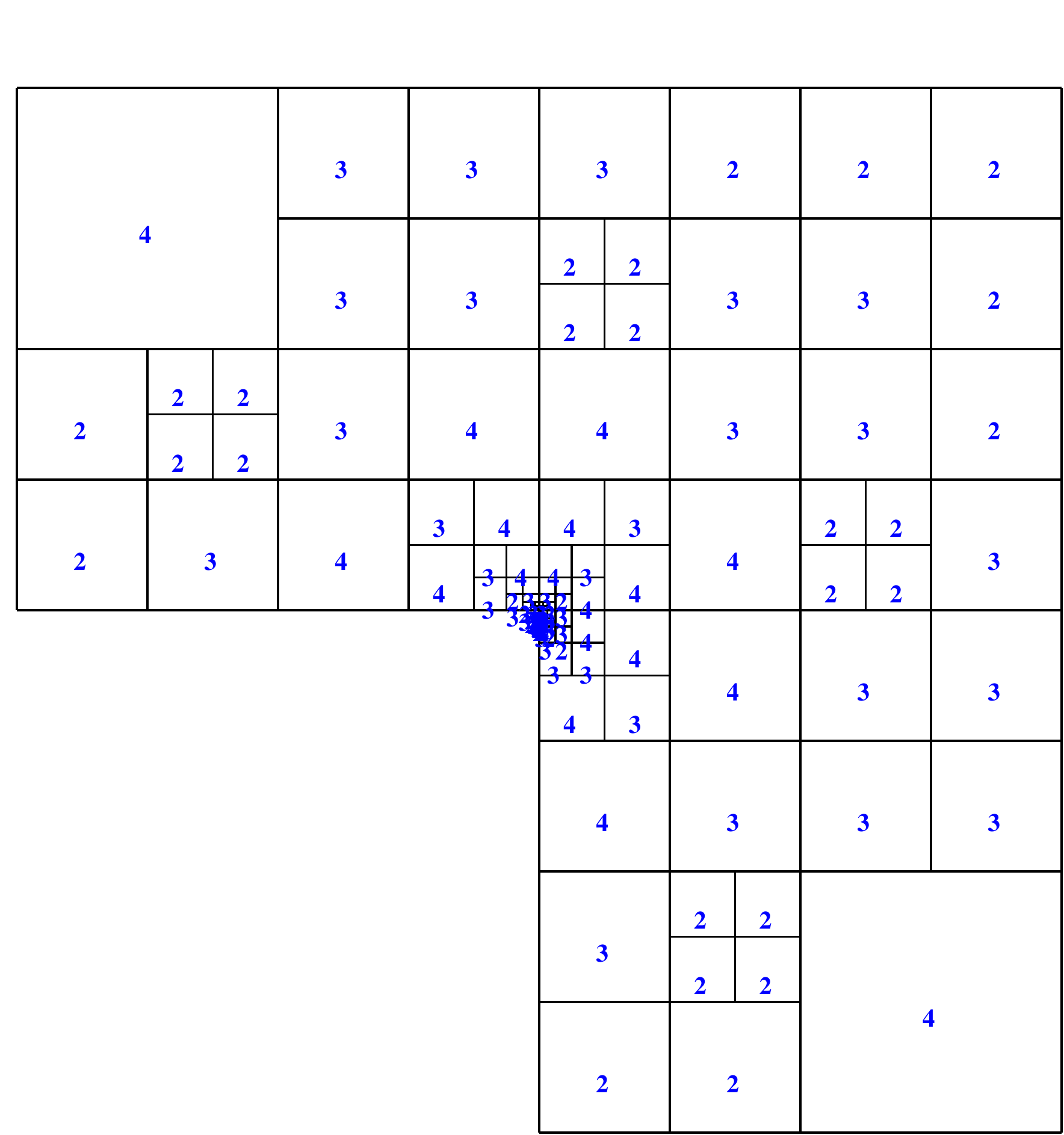}}
\caption{Steps 3 and 10 of the $\h\p$ refinement algorithm applied to the test case with known solution $u_3$ introduced in \eqref{3 solutions bonta}. Starting mesh: a coarse Cartesian mesh. Refinement strategy as in Figure \ref{figure standard h refinement}.}
\label{figure apos hp u3 square}
\end{figure}
\begin{figure}  [h]
\centering
\subfigure {\includegraphics [draft=false, angle=0, width=0.48\textwidth]{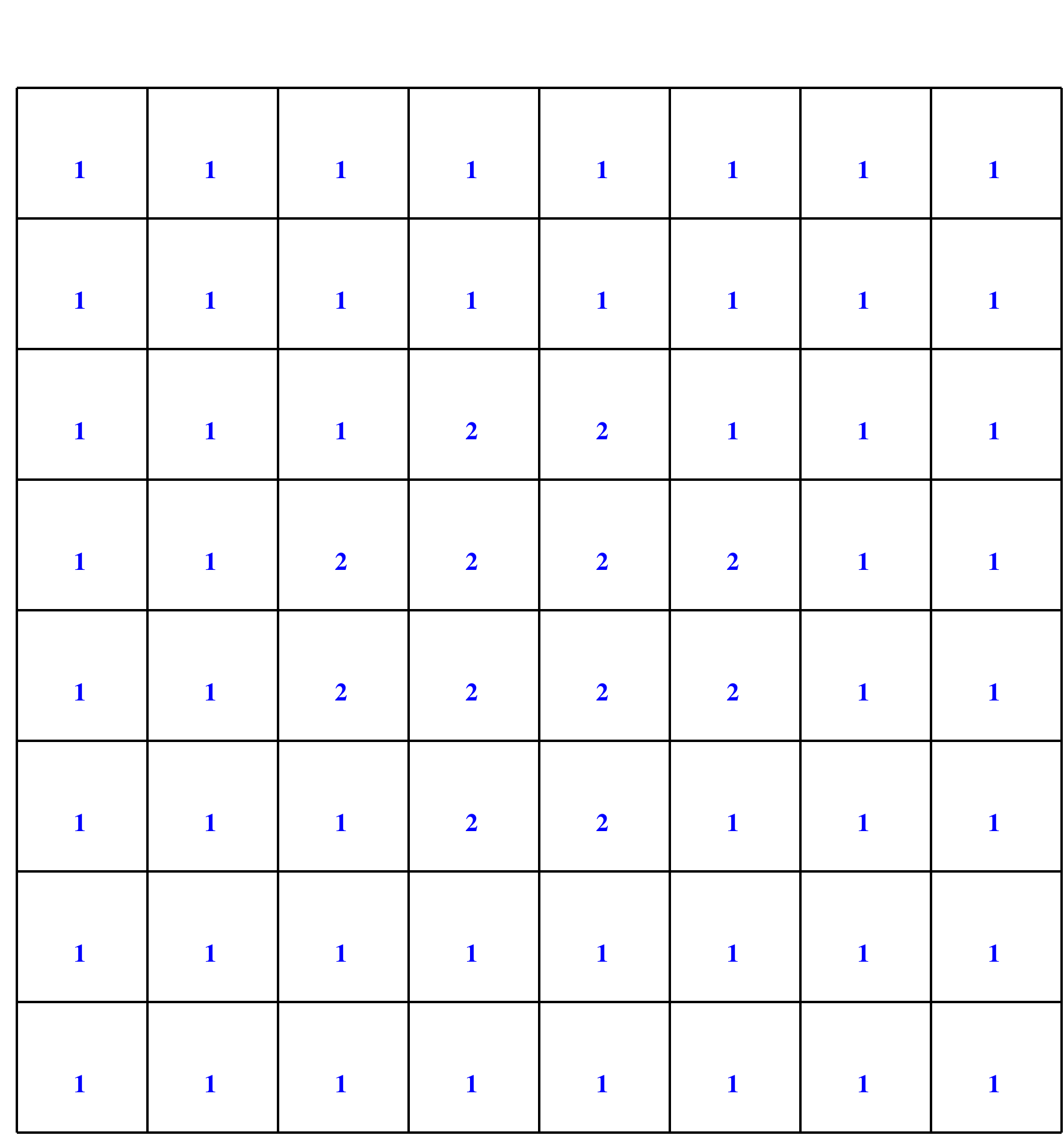}}
\subfigure {\includegraphics [draft=false, angle=0, width=0.48\textwidth]{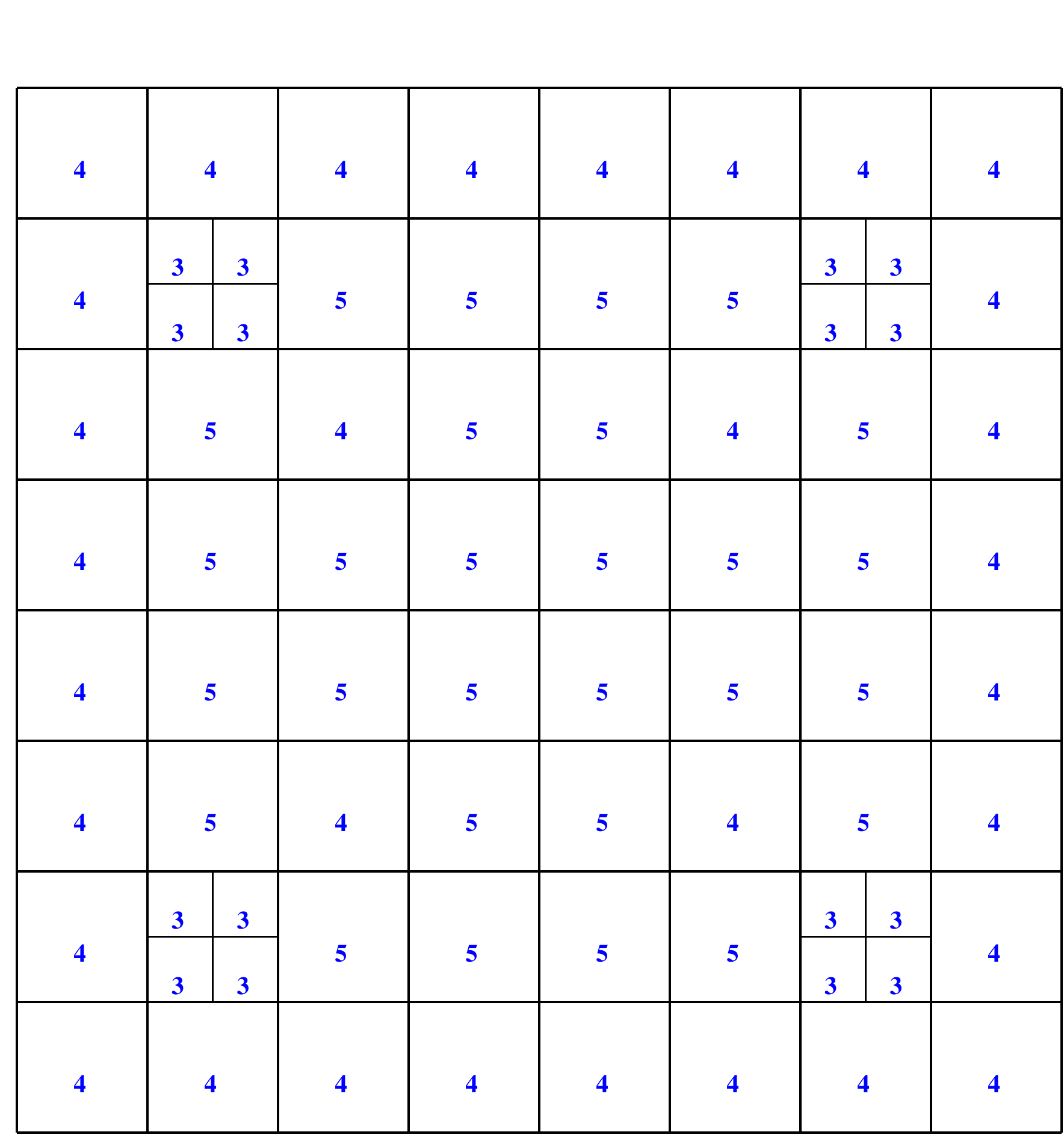}}
\caption{Steps 3 and 10 of the $\h\p$ refinement algorithm applied to the test case with known solution $u_4$ introduced in \eqref{u4}. Starting mesh: a coarse Cartesian mesh. Refinement strategy as in Figure \ref{figure standard h refinement}.}
\label{figure apos hp u4 square}
\end{figure}
Furthermore, we depict in Figures \ref{figure apos hp u3 Voronoi} and \ref{figure apos hp u4 Voronoi} the same set of experiments starting from a coarse Voronoi mesh.
In both cases, the $\h$ refinements are performed as in Figure \ref{figure standard h refinement}.
\begin{figure}  [h]
\centering
\subfigure {\includegraphics [draft=false, angle=0, width=0.488\textwidth]{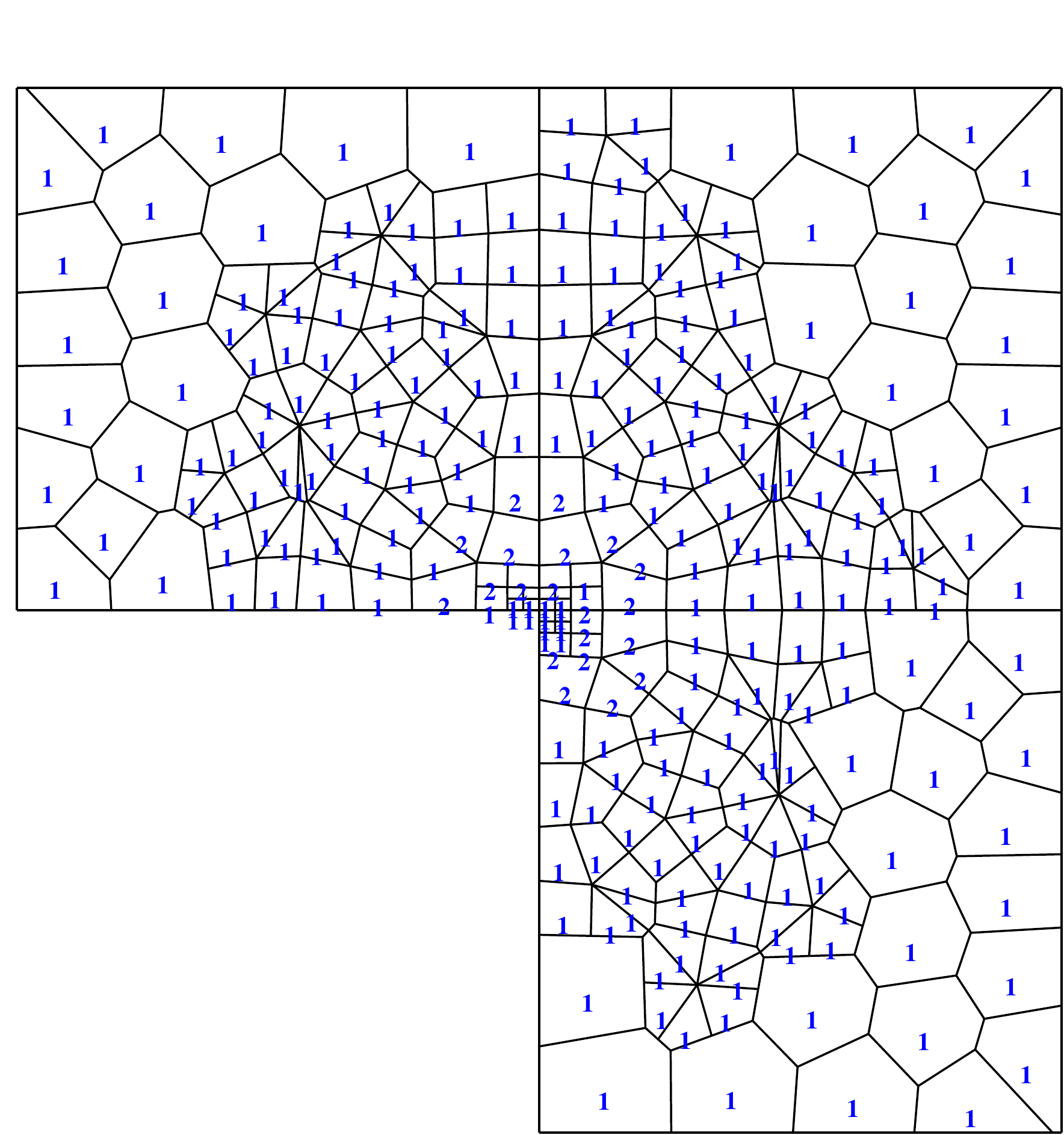}}
\subfigure {\includegraphics [draft=false, angle=0, width=0.48\textwidth]{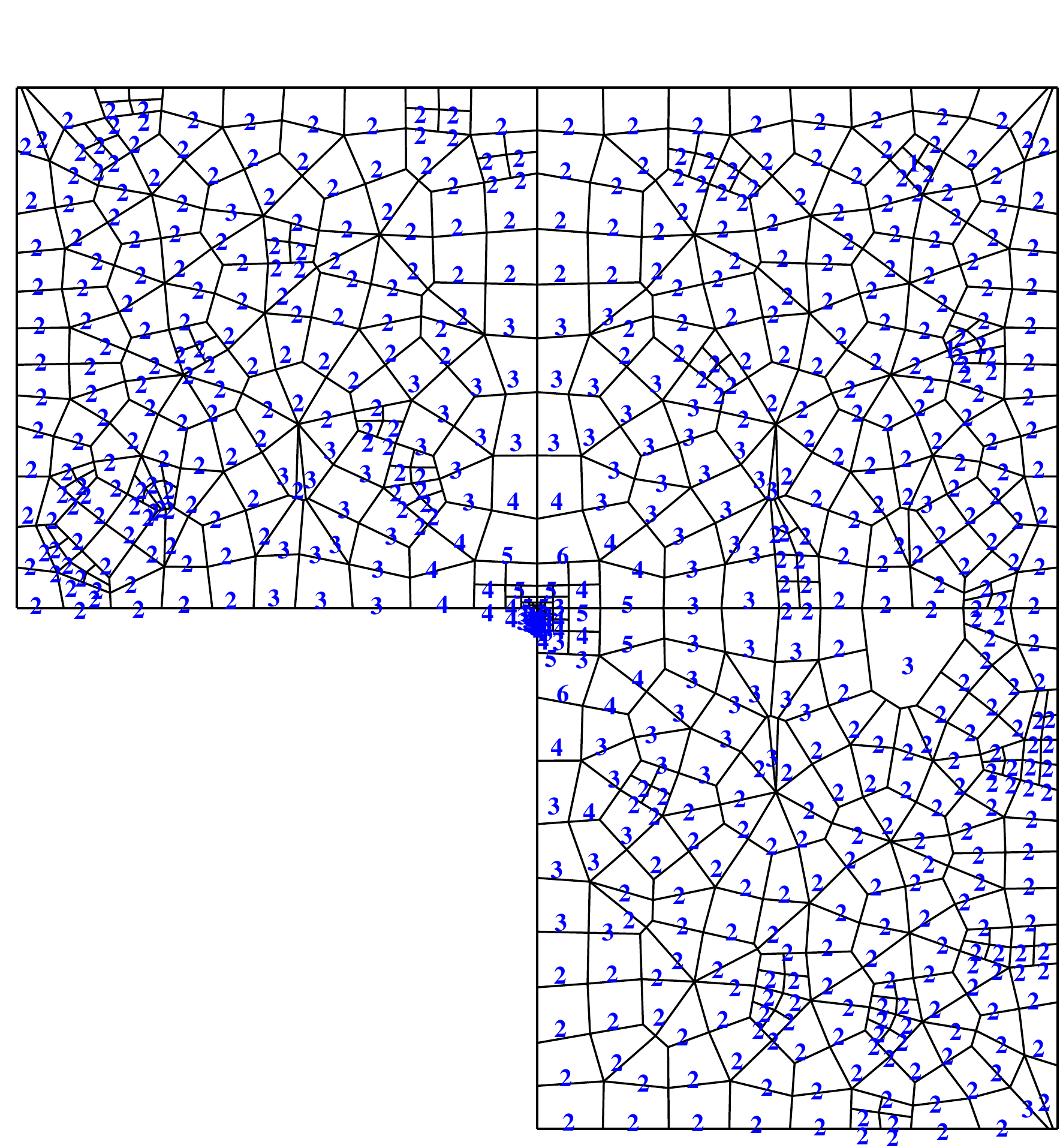}}
\caption{Steps 3 and 10 of the $\h\p$ refinement algorithm applied to the test case with known solution $u_3$ introduced in \eqref{3 solutions bonta}. Starting mesh: a coarse Voronoi mesh. Refinement strategy as in Figure \ref{figure standard h refinement}.}
\label{figure apos hp u3 Voronoi}
\end{figure}
\begin{figure}  [h]
\centering
\subfigure {\includegraphics [draft=false, angle=0, width=0.48\textwidth]{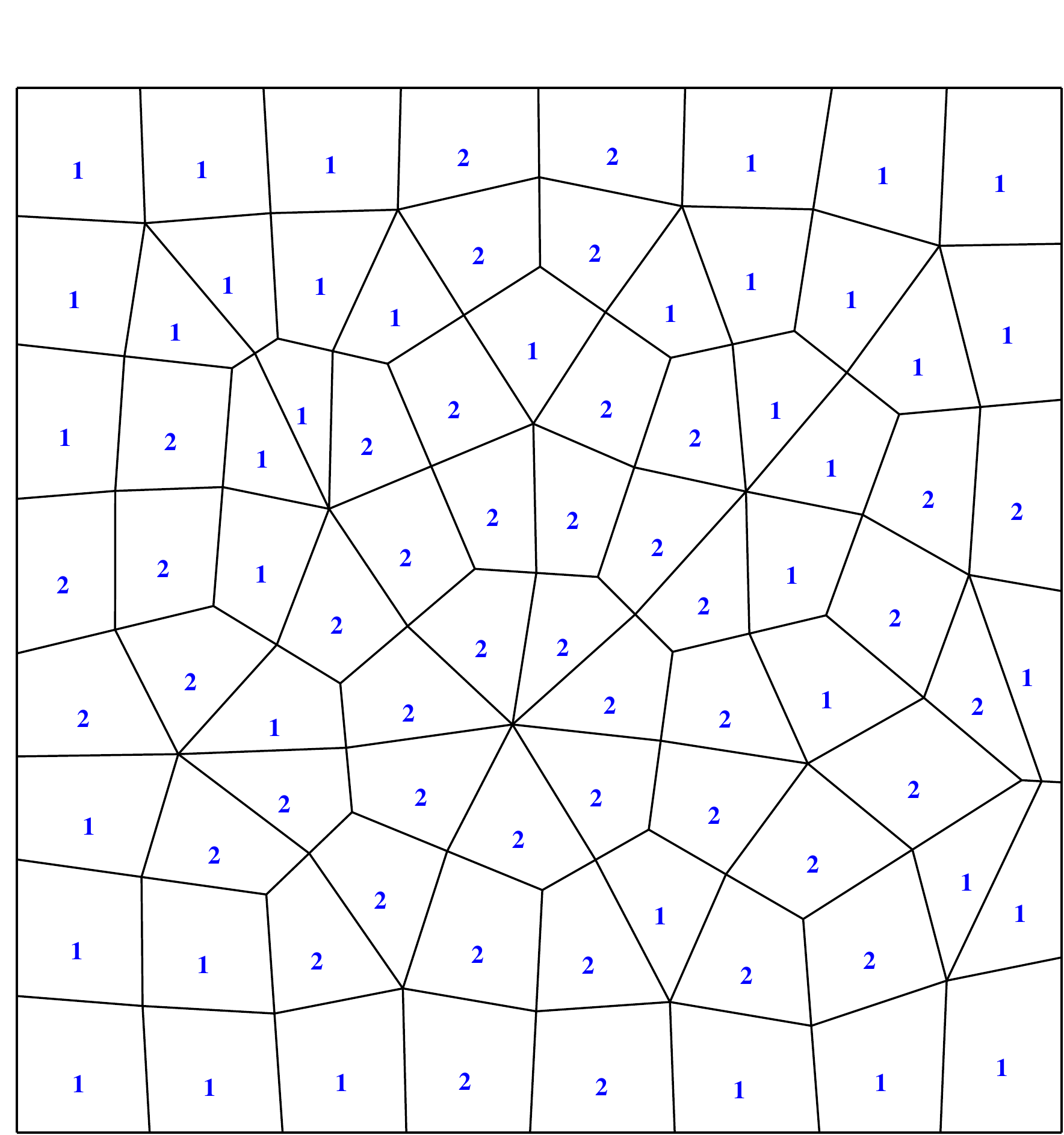}}
\subfigure {\includegraphics [draft=false, angle=0, width=0.48\textwidth]{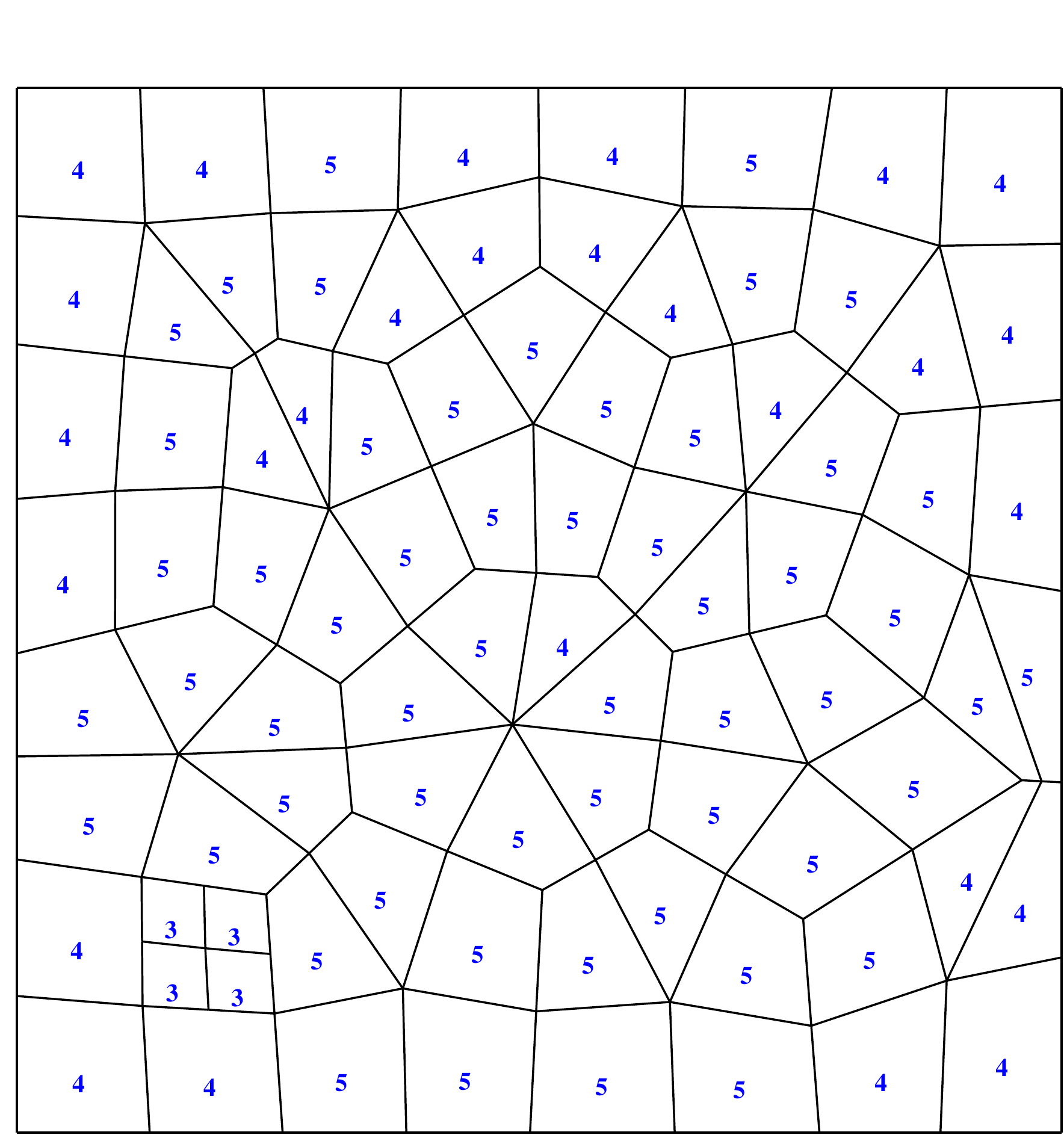}}
\caption{Steps 2 and 10 of the $\h\p$ refinement algorithm applied to the test case with known solution $u_4$ introduced in \eqref{u4}. Starting mesh: a coarse Voronoi mesh. Refinement strategy as in Figure \ref{figure standard h refinement}.}
\label{figure apos hp u4 Voronoi}
\end{figure}

From Figures \ref{figure apos hp u3 square} and \ref{figure apos hp u3 Voronoi}, it is possible to appreciate the fact that Algorithm \ref{algorithm hp refinement strategy} is actually catching the singular behaviours of the solution $u_3$.
In particular, the mesh is geometrically refined towards such singularity.  Note that, taking into account that the degrees of accuracy increase where the solution is smooth, 
it may seem that $\p$ refinement are effectuated also towards the re-entrant corner. This is not what actually happens, since lots of geometric refinements are performed and high degree of accuracy is picked only in outer layers.

Regarding Figures~\ref{figure apos hp u4 square} and~\ref{figure apos hp u4 Voronoi}, we can observe that the algorithm indeed detects that the target function is regular and therefore mainly applies $\p$ refinements; such refinements are slightly more concentrated in the area with higher derivatives.
In order to have a more quantitative evaluation on the algorithm performance, we plot in Figure \ref{new-fig:error} the error \eqref{computable error} and the error estimator \eqref{actual error estimator}
in terms of the square root of the total number of degrees of freedom, for the case with exact solution $u_4$.
We can appreciate the exponential convergence of the method, which resembles that of the $\p$ version of VEM when approximating analytic solutions \cite{hpVEMbasic}.
\begin{figure}  [h]
\centering
\subfigure {\includegraphics [draft=false, angle=0, width=0.48\textwidth]{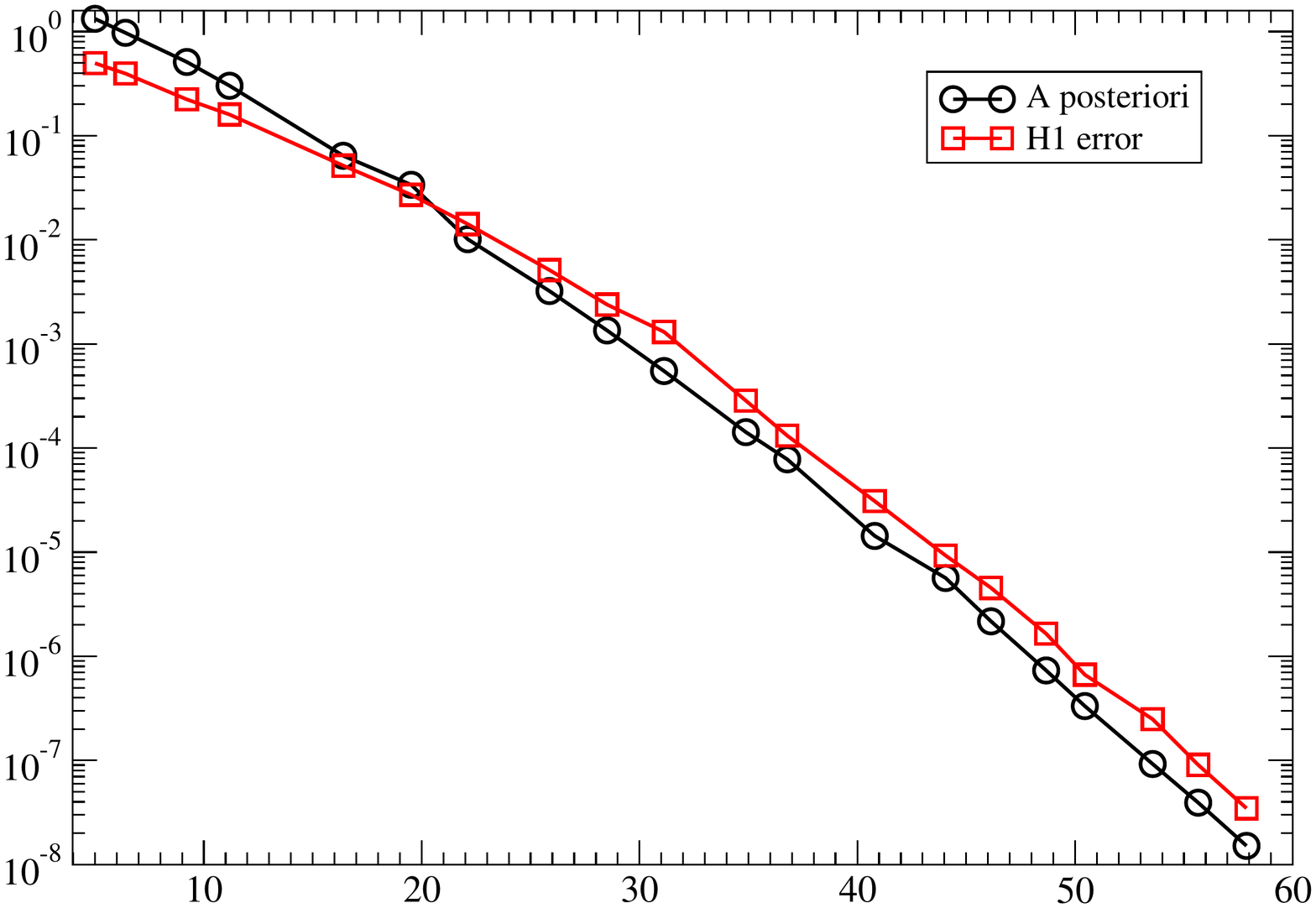} \put(-107,10){$\sqrt{\#\text{dofs}}$}}
\subfigure {\includegraphics [draft=false, angle=0, width=0.48\textwidth]{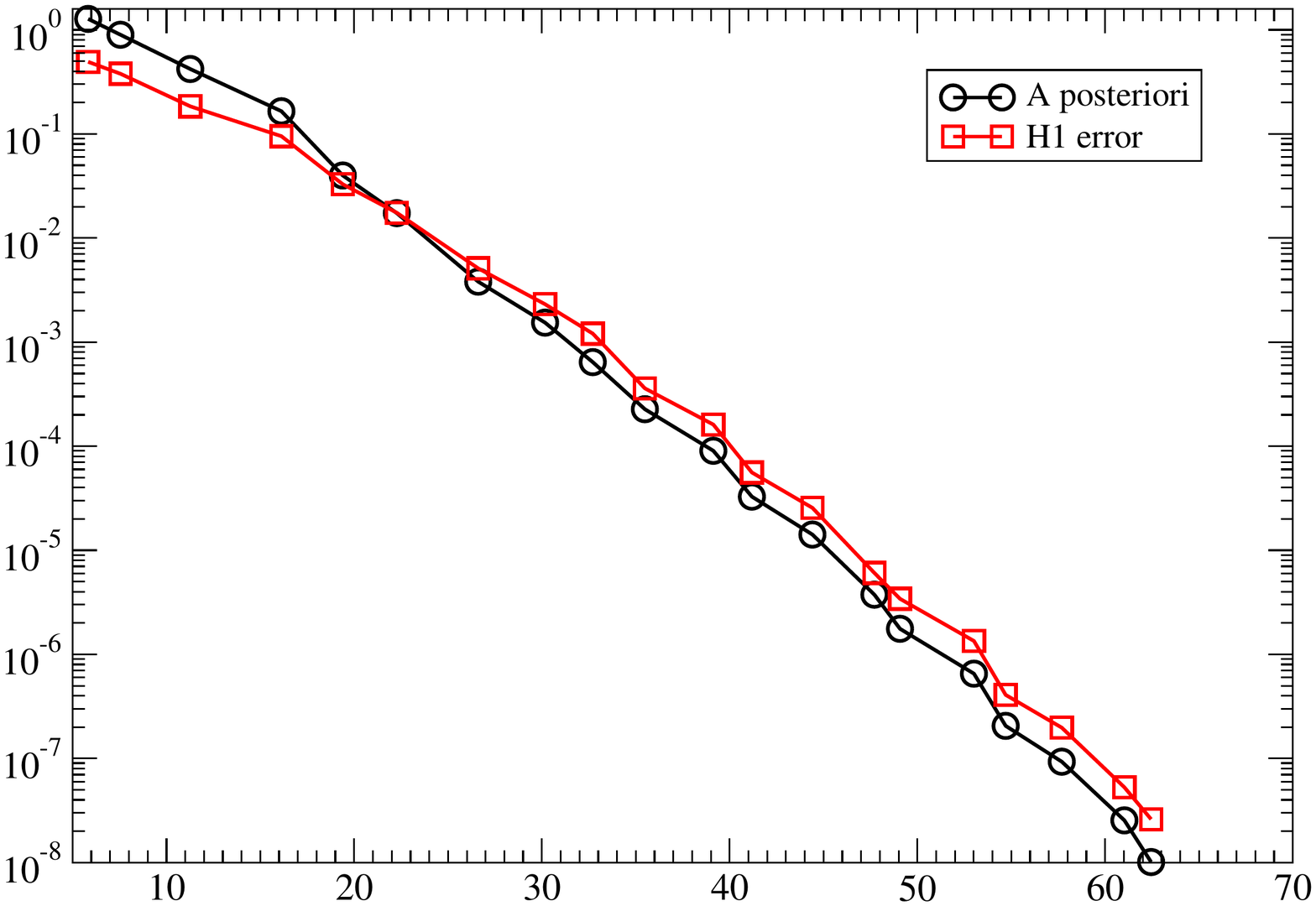} \put(-107,10){$\sqrt{\#\text{dofs}}$}}
\caption{Error decay of the computable error \eqref{computable error} and of the error estimator \eqref{actual error estimator} in terms of the square root of the degrees of freedom
employing the $\h\p$ adaptive refinement strategy in Algorithm \ref{algorithm hp refinement strategy}. Exact solution: $u_4$ introduced in \eqref{u4} on Cartesian (left) and Voronoi (right) meshes.}
\label{new-fig:error}
\end{figure}

Finally, we concentrate on the singular solution case $u_3$. We aim to investigate whether Algorithm \ref{algorithm hp refinement strategy} leads to a decay of the error which is exponential in terms of the cubic root of the overall degrees of freedom also in this more complex situation. This is indeed what one expects from the a priori error analysis of $\h\p$ FEM \cite{SchwabpandhpFEM} and of $\h\p$ VEM \cite{hpVEMcorner},
employing graded meshes towards the singularity and increasing elementwise the degrees of accuracy.
In Figure \ref{figure check decay error}, we plot the computable error \eqref{computable error} and the error estimator \eqref{actual error estimator}. We observe that, after some adaptive refinement steps, the decay gets the desired exponential slope.
\begin{figure}  [h]
\centering
\subfigure {\includegraphics [draft=false, angle=0, width=0.48\textwidth]{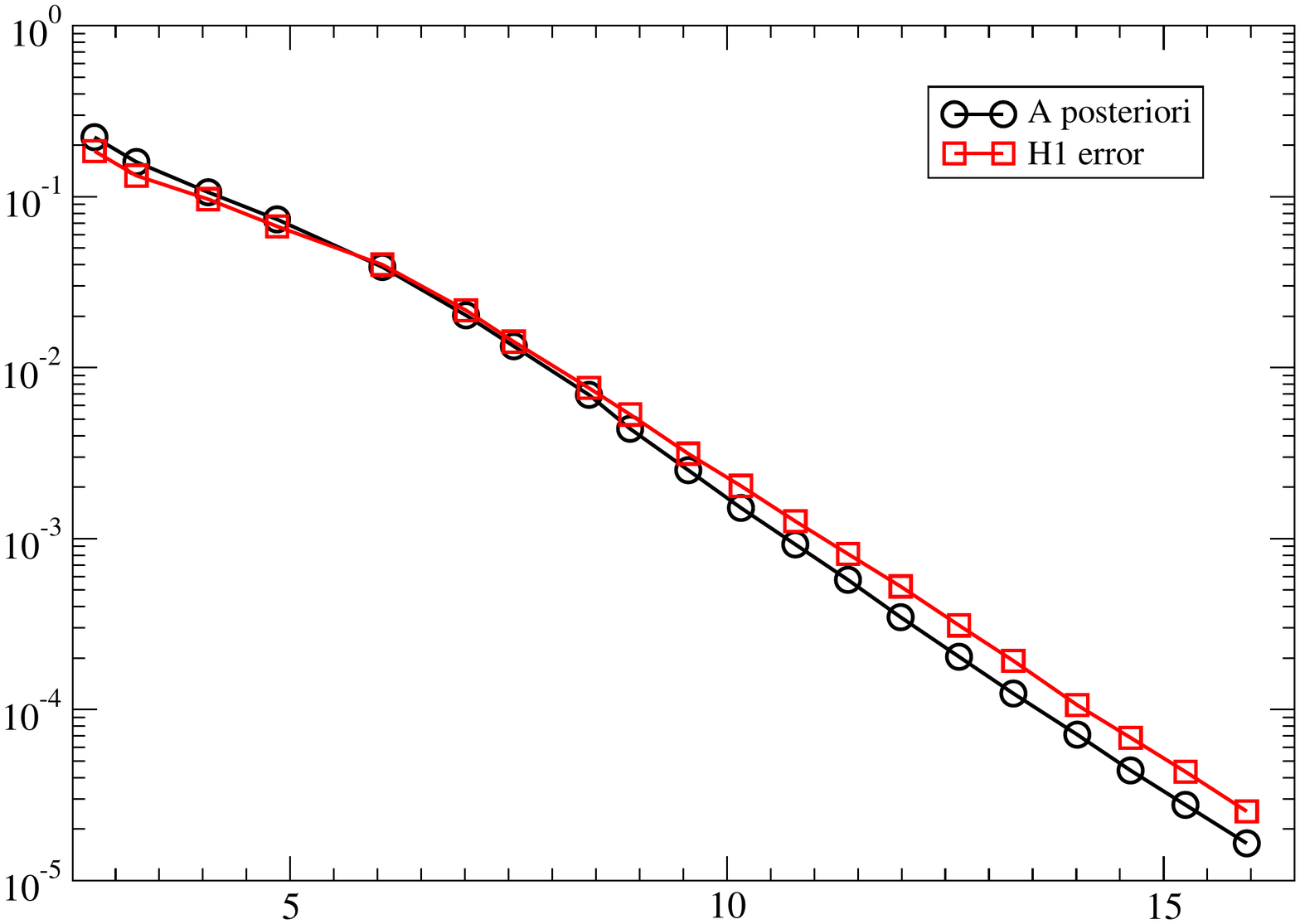} \put(-114,10){$\sqrt[3]{\#\text{dofs}}$}}
\subfigure {\includegraphics [draft=false, angle=0, width=0.48\textwidth]{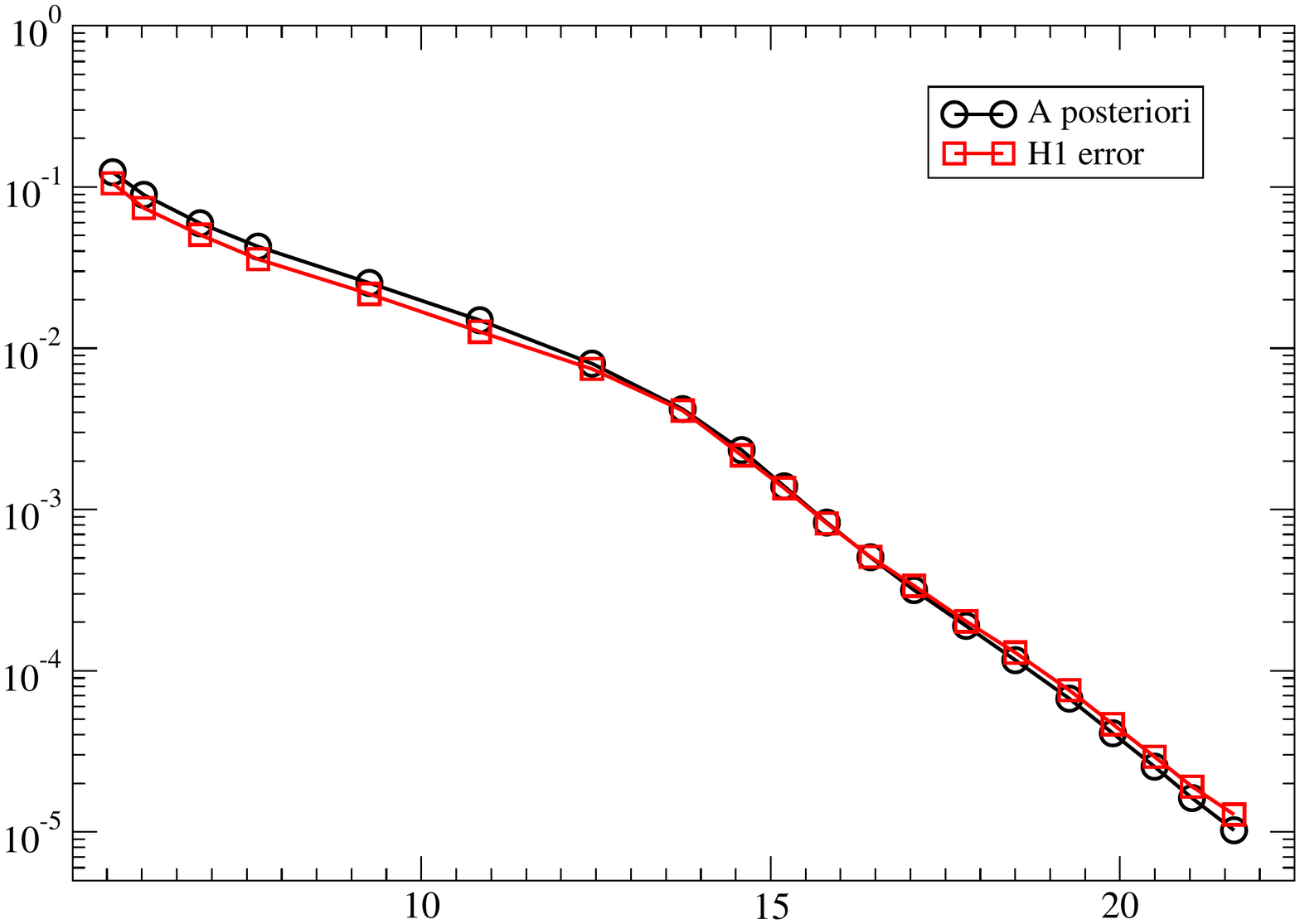} \put(-114,10){$\sqrt[3]{\#\text{dofs}}$}}
\caption{Error decay of the computable error \eqref{computable error} and of the error estimator \eqref{actual error estimator} in terms of the cubic root of the degrees of freedom
employing the $\h\p$ adaptive refinement strategy in Algorithm \ref{algorithm hp refinement strategy}. Exact solution: $u_3$ introduced in \eqref{3 solutions bonta} on Cartesian (left) and Voronoi (right) meshes.}
\label{figure check decay error}
\end{figure}

As a final experiment, we compare in Figure \ref{figure hp VS h} the performances of the $\h\p$ adaptive refinement strategy of Algorithm \ref{algorithm hp refinement strategy}, with a pure $\h$ adaptive refinement strategy, which is equivalent to the scheme presented in Algorithm \ref{algorithm hp refinement strategy} setting $\eta^2_{pred,\E, n} = 0$ for all $\E \in \taun$ and for all $n\in \mathbb N$.
\begin{figure}  [h]
\centering
\subfigure {\includegraphics [draft=false, angle=0, width=0.48\textwidth]{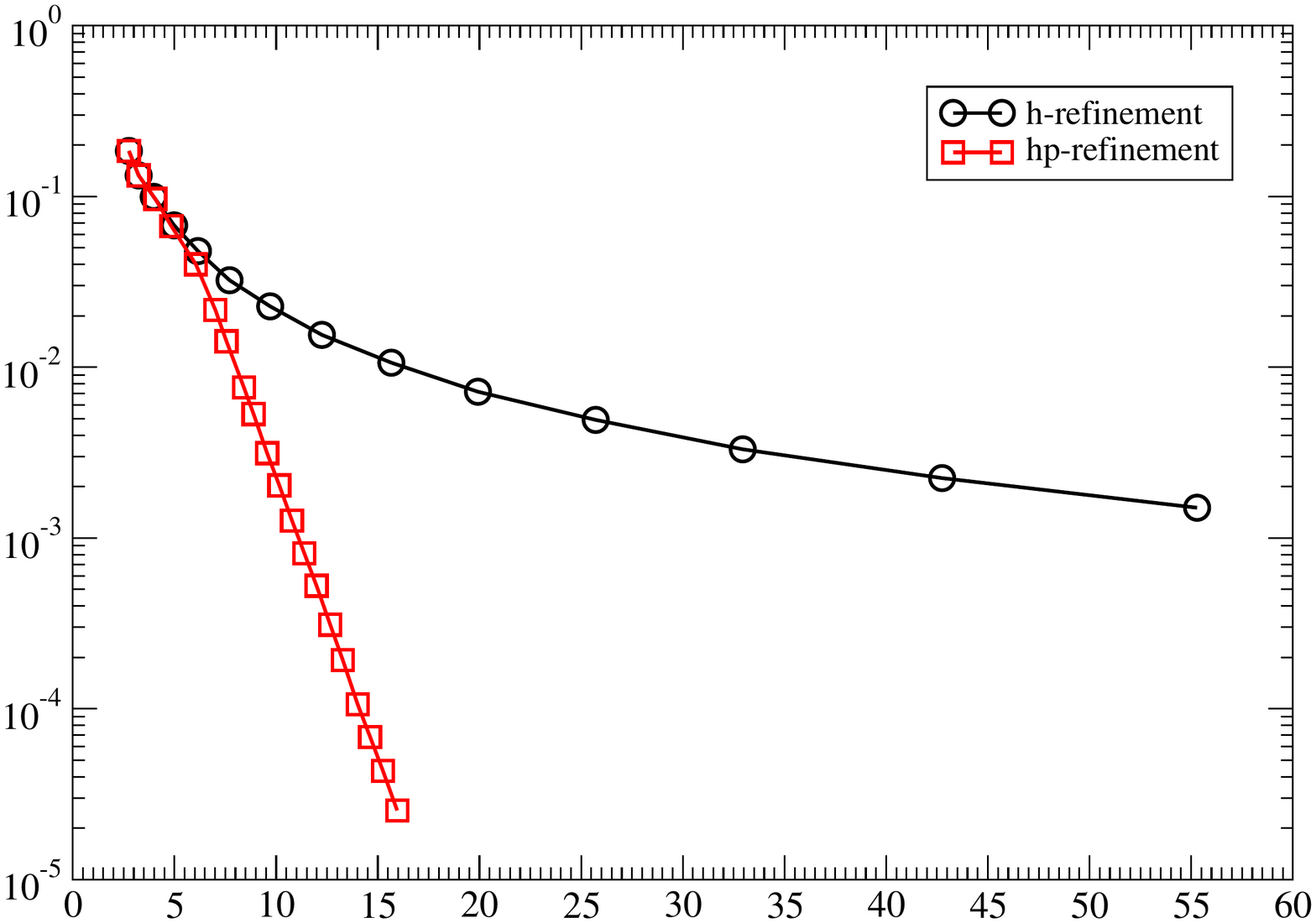} \put(-117,10){$\sqrt[3]{\#\text{dofs}}$}}
\subfigure {\includegraphics [draft=false, angle=0, width=0.48\textwidth]{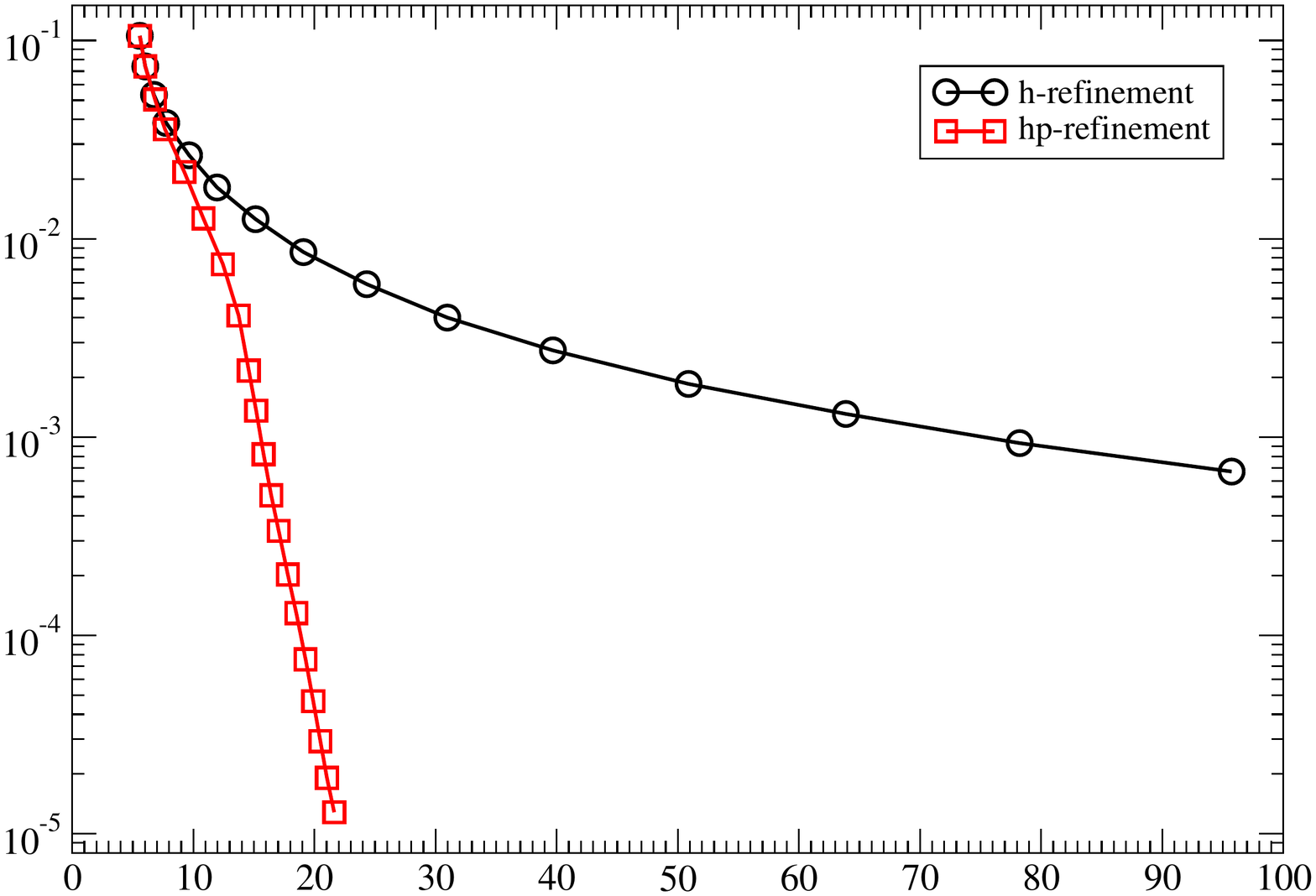} \put(-117,10){$\sqrt[3]{\#\text{dofs}}$}}
\caption{Comparison between the $\h\p$ and $\h$ adaptive refinement in Algorithm \ref{algorithm hp refinement strategy}.
Exact solution: $u_3$ introduced in \eqref{3 solutions bonta} on Cartesian (left) and Voronoi (right) meshes. On the $x-$axis, we depict the cubic root of the number of degrees of freedom.}
\label{figure hp VS h}
\end{figure}

From Figure~\ref{figure hp VS h}, it is evident that the~$\h\p$ adaptive strategy overperforms the~$\h$ refinement strategy, since, on the portions of the domain where the solution is sufficiently smooth,
the~$\p$ version is able to capture with very few degrees of freedom the same accuracy of an~$\h$ version with many more degrees of freedom.
It is however worth to underline that this comes at the price of having, for high~$\p$, a more densely populated stiffness matrix.

\section*{Acknowledgement}
L. B. d. V was partially supported by the European Research Council through the H2020 Consolidator Grant (grant no. 681162) CAVE, Challenges and Advancements in Virtual Elements. This support is gratefully acknowledged.
G. M. was supported by the Laboratory Directed Research and Development Program (LDRD), U.S. Department of Energy Office of Science, Office of Fusion Energy Sciences, and the DOE Office of Science Advanced Scientific Computing Research (ASCR) Program in Applied Mathematics Research,
under the auspices of the National Nuclear Security Administration of the U.S. Department of Energy by Los Alamos National Laboratory, operated by Los Alamos National Security LLC under contract DE-AC52-06NA25396.
L. M. has been funded by the Austrian Science Fund (FWF) through the project F 65.

{\footnotesize
\bibliography{bibliogr}
}
\bibliographystyle{plain}

\end{document}